\documentclass[12pt]{amsart}%

\usepackage{amsmath, amssymb, amsthm, amsfonts, bbm, dsfont}
\usepackage{graphicx}
\usepackage{wrapfig}
\usepackage{mathrsfs}

\usepackage{hyperref}

\usepackage{a4wide}

\usepackage{marginnote}

\numberwithin{equation}{section}
\theoremstyle{plain}
\newtheorem{theorem}[subsubsection]{Theorem}

 \newtheorem{proposition}[subsubsection]{Proposition}
 \newtheorem{corollary}[subsubsection]{Corollary}
 \newtheorem{conjecture}[subsubsection]{Conjecture}

 \theoremstyle{definition}
 \newtheorem{defn}[subsubsection]{Definition}

\newtheorem{remark}[subsubsection]{Remark}

\usepackage{pictexwd}
\usepackage{stackengine}
\usepackage{mathtools}
\usepackage{amssymb}
\usepackage{dsfont}

\usepackage{mathabx}

\newcommand{\rmT}{\mathrm{T}}

\newcommand{\CC}{\mathbb{C}}

\newcommand{\QQ}{\mathbb{Q}}
\newcommand{\ZZ}{\mathbb{Z}}

\newcommand{\scZ}{\mathscr{Z}}
\newcommand{\scz}{\mathscr{Z}^0}
\newcommand{\scC}{\mathscr{C}}
\newcommand{\scc}{\mathscr{C}}
\newcommand{\scX}{\mathscr{X}}

\newcommand{\calc}{\mathcal{C}}
\newcommand{\calE}{\mathcal{E}}
\newcommand{\calF}{\mathcal{F}}

\newcommand{\calG}{\mathcal{G}}
\newcommand{\calO}{\mathcal{O}}

\newcommand{\calR}{\mathcal{R}}
\newcommand{\calS}{\mathcal{S}}
\newcommand{\calB}{\mathcal{B}}

\newcommand{\calL}{\mathcal{L}}
\newcommand{\calD}{\mathcal{D}}

\newcommand{\caln}{\mathcal{N}}

\newcommand{\frg}{\mathfrak{g}}

\newcommand{\frb}{\mathfrak{b}}

\newcommand{\frn}{\mathfrak{n}}

\newcommand\aff{\textup{aff}}

\newcommand{\gr}{\textup{gr}}

\newcommand\Lie{\textup{Lie}\ }
\newcommand\loc{\textup{loc}}
\newcommand\Loc{\textup{loc}}
\newcommand\Mod{\textup{Mod}}

\newcommand\Spec{\textup{Spec} }

\newcommand\st{\textup{st}}

\newcommand{\Tr}{\textup{Tr}}

\newcommand{\Coh}{\textup{Coh}}

\newcommand\Hom{\textup{Hom}}

\newcommand\GL{\textup{GL}}
\newcommand\gl{\mathfrak{gl}}

\newcommand{\Ad}{\textup{Ad}}

\newcommand{\Id}{\mathrm{Id}}

\newcommand{\quash}[1]{}

\newcommand{\pt}{\mathsf{p}}

\newcommand{\wt}{\widetilde}

\newcommand{\und}[1]{\underline{#1}}

\newcommand{\Br}{\mathfrak{Br}}
\newcommand{\PBr}{\mathfrak{PBr}}

\usepackage{tikz}
\usetikzlibrary{shapes,fit,intersections}
\usetikzlibrary{decorations.markings}
\usetikzlibrary{decorations.pathreplacing}
\usetikzlibrary{patterns}
\usetikzlibrary{matrix,arrows}
\usepgflibrary{decorations.pathmorphing}

\usepackage{tikz-cd}

\newcommand{\Hilb}{\textup{Hilb}}

\newcommand{\MF}{\mathrm{MF}}
\newcommand{\MMF}{\mathbf{MF}}

\newcommand{\calX}{\mathcal{X}}
\newcommand{\MFs}{\mathrm{MF}}
\newcommand{\calZ}{\mathcal{Z}}
\newcommand{\calC}{\mathcal{C}}
\newcommand{\calY}{\mathcal{Y}}

\newcommand{\dg}{\mathrm{dg}}
\newcommand{\per}{\mathrm{per}}

\newcommand{\scY}{\mathscr{Y}}
\newcommand{\FHilb}{\mathrm{FHilb}} 
\newcommand{\Fl}{\mathrm{Fl}}
 
\newcommand{\frh}{\mathfrak{h}}

\newcommand{\CE}{\mathrm{CE}}
\newcommand{\ChE}{Chevalley-Eilenberg }

\newcommand{\odel}{\overset{\scriptstyle\Delta}{\otimes}}

\newcommand{\CH}{\mathrm{CH}}
\newcommand{\Dr}{\mathrm{Dr}}
\newcommand{\dr}{\mathrm{Dr}}
\newcommand{\HC}{\mathrm{HC}}

\def\Hilb{ \mathrm{Hilb}}

\title{Categorical Chern character and braid groups}
\author{A. Oblomkov}
\address{
A.~Oblomkov\\
Department of Mathematics and Statistics\\
University of Massachusetts at Amherst\\
Lederle Graduate Research Tower\\
710 N. Pleasant Street\\
Amherst, MA 01003 USA
}
\email{oblomkov@math.umass.edu}

\author{L. Rozansky}
\address{
L.~Rozansky\\
Department of Mathematics\\
University of North Carolina at Chapel Hill\\
CB \# 3250, Phillips Hall\\
Chapel Hill, NC 27599 USA
}
\email{rozansky@math.unc.edu}

\begin{document}
\maketitle
\begin{abstract}
To a braid \(\beta\in \Br_n\) we associate a complex of sheaves \(S_\beta\) on \(\Hilb_n(\CC^2)\) such that the previously defined triply graded link homology of the closure \(L(\beta)\) is isomorphic to the homology of \(S_\beta\).
%
  The construction of \(S_\beta\) relies on the Chern functor
  \(\CH\colon \MF_n^{\mathrm{st}}\to \mathrm{D}^{\mathrm{per}}_{\CC^*\times\CC^*}(\Hilb_n(\CC^2))\)  defined in the paper together with its adjoint functor
  \(\HC\).
  We prove a
  formula for the closure  of  sufficiently   positive elements of the Jucys-Murphy algebra previously conjectured
  by Gorsky, Negut and Rasmussen.

\end{abstract}




\section{Introduction}
\label{sec:introduction}

\def\rloc{\mathrm{loc}}
\def\rfs{\mathrm{fs}}
\def\raf{\mathrm{af}}
\def\rst{\mathrm{st}}
\def\rfr{\mathrm{fr}}
\def\rper{\mathrm{per}}
\def\rloc{\mathrm{loc}}
\def\rcnv{\mathrm{cnv}}
\def\rcv{\mathrm{cv}}
\def\raff{\mathrm{aff}}
\def\rfree{\mathrm{free}}
\def\rlin{\mathrm{lin}}
\def\reven{\mathrm{even}}
\def\rfc{\mathrm{fc}}

\def\MFfs{\MF^{\rfs}}
\def\rmT{ \mathrm{T}}
\def\Ts{ \rmT^* }
\def\TsFl{ \Ts\Fl}
\def\GLn{ \GL_n }
\def\gl{\mathfrak{gl}}
\def\gln{\gl_n}
\def\glns{ \gln^* }
\def\CCn{ \CC^n }
\def\Wfr{W^\rfr}
\def\CHfl{\CH^{\rfs}_{\rloc}}
\def\HCfl{\HC^{\rfs}_{\rloc}}
\def\MFG{\MF_G}
\def\Wdr{ W_{\dr} }

\def\rmD{\mathrm{D}}
\def\Dper{ \rmD^{\rper}}
\def\DperG{ \Dper_{G}}

\def\amu{ \mu }

\def\Tqt{ \mathbb{T}_{q,t} }
\def\CCsq{ \CC^*_q }
\def\CCst{ \CC^*_t }
\def\DT{ \mathrm{D}_{\Tqt} }
\def\DperT{ \Dper_{\Tqt} }
\def\DperGBd{ \Dper_{\GBd} }
\def\Bd{ B_{\Delta}}
\def\GBd{ G\times \Bd }
\def\MFT{\MF^{\Tqt}}
\def\MFGLT{\MFT_{\GLn}}
\def\MFGsLT{\MFT_{G}}
\def\MFGBtT{ \MFT_{\GBt}}
\def\MFGBdT{ \MFT_{\GBd}}
\def\GBt{G\times B^2}

\def\xId{ \mathbbm{1}}
\def\Adv#1{ \mathrm{Ad}_{#1}}

\def\sczCH{ \scz_{\CH} }
\def\scZCH{ \scZ_{\CH} }
\def\scZCHd{ \scX_{\Delta}}
\def\scXdg{\tilde{\scX}_{\Delta}}

\def\grb{ b }
\def\grbv#1{ \grb_{#1}}
\def\grbo{ \grbv{1}}
\def\grbt{ \grbv{2}}

\def\fD{ f_{\Delta}}
\def\fDs{\fD^*}
\def\fDls{ f_{\Delta,*}}
\def\pDr{ \pi_\Dr}
\def\tpDr{ \tilde{\pi}_{\Dr}}
\def\pDrs{ \pi_{\Dr,*} }
\def\pDrus{ \pDr^* }

\def\frnD{ \frn_{\Delta}}

\def\WCHx{ W_{\CH}}
\def\WCHxd{ W_{\Delta}}
\def\gp{ g' }
\def\cNt{ \tilde{\mathcal{N}}}

\def\HHH{ \mathrm{H}}
\def\HHHb{\HHH^\bullet }

\def\emi{\iota}

\def\lhs{l.h.s.\ }
\def\aff{\mathrm{aff}}

\def\mJM{ \delta}
\def\mJMv#1{ \mJM^{#1}}
\def\mJMa{ \mJMv{\vec{a}}}
\def\fJM{ \mathfrak{JM}}
\def\fJMn{\fJM_n}

\def\ytb{\mathbf{T}}

\def\vrho{\vec{\rho}}
\def\cLva{ \calL^{\vec{a}} }
\def\cLvr{ \calL^{\vrho}}

\def\Lva{ \mJMa }

\def\NS{ N\!S }
\def\NSn{ \NS_n }

\def\tzeta{\tilde{\zeta}}
\def\shQ{\mathbf{Q}}
\def\shT{ \mathbf{T}}
\def\xTot{ \mathrm{Tot}}

\def\vek{\vec{k}}
\def\dvek{\delta^{\vek}}

In \cite{OblomkovRozansky16} we constructed a homomorphism \(\Phi\) from the braid group \(\Br_n \) to the convolution algebra of
\(\GL_n\times \Tqt\)-equivariant, where  \(\Tqt=\CCsq\times\CCst\),
matrix factorizations:
\begin{equation}
\label{eq:stfrmf}
\MF^{\st}=\MF\left(\big(\gln\times V_G\times \TsFl\times\TsFl\big)^{\rst},\GL_n\times \Tqt, tq^0, W\right), \quad
W=\amu_1-\amu_2,
\end{equation}
where \(\TsFl=(\frn\times\GLn)/B\) is the  cotangent vector bundle to the flag variety, $\amu\colon\TsFl\to \glns
$ is the corresponding generalized moment map and $(\cdots)^{\rst}$ denotes an open  stability condition defined later. The factor \(V_G\) in the definition of the
last space is the tautological representation \(\CC^n\) of \(\GL_n\)
The categorical representation of the braid group was used to construct a triply-graded link invariant (homology):
\[\HHH(\beta)=\mathbb{H}(\mathcal{E}\mathrm{xt}(\Phi(\beta),\Phi(1))\otimes \Lambda^\bullet V_G^\vee)^{\GL_n},\]
where \(V_G^\vee\) is the \(\GL_n\)-repsentation dual to \(V_G\)
and $\mathbb{H}(\dots)$ is the hypercohomology.

The $q$- and $t$-gradings on \(\HHH(\beta)\)
are weights of $\Tqt$, while the a-grading is the exterior power in $\Lambda^\bullet V^\vee_G$. It was shown in \cite{OblomkovRozansky20}
 that the triply-graded homology
discussed in this paper coincide with the original categorification of HOMFLYPT polynomial \cite{KhovanovRozansky08b}, \cite{Khovanov07}.

In this paper we construct a pair of functors which we call a Chern functor and a co-Chern functor:
\begin{equation}
\label{eq:mnchcchf}
  \begin{tikzcd}
    \MF^{\st}\arrow[rr,bend left,"\CH^{\st}_{\loc}"]&& \mathrm{D}^{\rper}(\Hilb)\arrow[ll,bend left,"\HC^{\st}_{\loc}",pos=0.435]
  \end{tikzcd},
\end{equation}
where \(\Hilb\)  is a  Hilbert scheme of \(n\) points on \(\CC^2\), while \(\mathrm{D}^{\rper}(\Hilb)=\mathrm{D}^{\rper}(\Hilb,\Tqt,tq^0)\) is the derived category 
two-periodic \(\Tqt\)-equivariant complexes on \(\Hilb\).
In a subsequent paper~\cite{OblomkovRozansky18b} we  put these functors in context of a 2-category of a 3D TQFT (equivariant B-model).
Here we explore the properties of the functors  and their interaction with the monoidal
structure \(\star\) on \(\MF^{\st}\) used in the construction of \(\Phi\).
We define monoidal structure \((\MF^{\st},\star)\) in section~\ref{sec:properties} and 
show the following

\begin{theorem}[Corollary~\ref{cor:main1}]\label{thm:HC-CH}
 
There is a monoidal structure \(\star\) on \(\mathrm{D}^{\rper}(\Hilb)\) such that
\begin{itemize}
\item the functor \(\CH^{\rst}_{\rloc}\) is a right adjoint of \(\HC^{\rst}_{\rloc}\),
\item the functor \(\HC^{\rst}_{\rloc}\) is monoidal,
\item the image of \(\HC^{\rst}_{\rloc}\) commutes with the elements \(\Phi(\beta)\), \(\beta\in \Br_n\).
\end{itemize}
  
\end{theorem}

The monoidal structure \(\star\) on \(\mathrm{D}^{\rper}(\Hilb)\) is defined in section~\ref{sec:local-funct},
and is expected to be a  deformation of \(\otimes\), see proposition~\ref{prop:compare-star} for the details. The existence of a non-trivial deformations of the tensor
product for sheaves on the holomorphic symplectic variety was predicted by Kapustin-Rozansky in \cite{KapustinRozansky10}.

Finally, as a manifestation of the categorified Riemann-Roch formula, we obtain a new interpretation for the triply-graded homology as hyper-cohomology of
a particular two-periodic complex of \(\Tqt\)-equivariant sheaves:
\begin{theorem}[Theorem~\ref{thm:HC-homology}]\label{thm:main}
  For any \(\beta\in \Br_n\) we have:
  \[\HHH(\beta)=\mathbb{H}(\CH^{\rst}_{\rloc}(\Phi(\beta))\otimes \Lambda^\bullet\calB))
        \]
    Moreover, for any \(k\) we have
    \[
    \CH^{\rst}_{\rloc}(\Phi(\beta\cdot FT^k))=
    \CH^{\rst}_{\rloc}(\Phi(\beta))\otimes \det(\calB)^k\]
  where \(FT\) is the full-twist braid.
\end{theorem}

In the last theorem we used tautological vector bundle on \(\calB\) on \(\Hilb\).
For \(I\subset\CC[x,y]\), \(I\in \Hilb\) the fiber \(\calB_I\) at \(I\) is dual to the
quotient \(\CC[x,y]/I\).

The advantage of this new interpretation is that the Hilbert scheme is smooth, unlike the very singular  flag Hilbert scheme  \(\mathrm{FHilb}_n\) which 
is a homological support of the complex of sheaves \(\mathbb{L}(\Phi(\beta))\) used in our previous construction of knot homology \cite{OblomkovRozansky16}. In section~\ref{sec:knot} we remind the construction of two periodic complex \(\mathbb{L}(\Phi(\beta))\) of \(B\)-equivariant coherent sheaves on
the stable part of  \(\frb\times \frn\), the homology of complex are sheaves supported on the locus  of commuting matrices.

As an example, we apply the Chern functor to  sufficiently positive Jucys-Murphy braids in $\Br_n$ and use a combination of the
vanishing higher cohomology of an 'unfolding' of the resulting complexes of sheaves together with localization formula in order to establish an explicit formula for their link homology. Jucys-Murphy elements of \(\Br_n\) are naturally labeled by
the \(n\)-tuple of integers \(\vec{a}\in \ZZ^{n-1}\) and the sufficiently positive elements have all entries sufficiently positive and sufficiently positive differences \(a_i-a_{i-1}\),
as explained below.

Recall that the JM subgroup $\fJMn\subset\Br_n$ is generated by elementary JM braids
\[
\delta_i=\sigma_i\sigma_{i+1}\dots\sigma_{n-1}^2\dots\sigma_{i+1}\sigma_i,
\]
and a JM braid can be presented as their product:
\[
\delta^{\vec{a}}:= \prod_{i=2}^{n}\delta^{a_i}_i,\qquad \text{where}\quad
\vec{a} = (a_1,\ldots,a_{n-1})\in\ZZ^{n-1}.
\]
Let notice that \(FT= \delta^{\vec{1}}\),\(\vec{1}=(1,\dots,1)\).
%
%



\begin{theorem}[Corollary~\ref{cor:chi-JM} and section~\ref{sec:proof-JM}]\label{thm:loc}
  For any $\vec{a}\in \ZZ^{n-1}$ such that
  \(a_i\ge a_j-1\) for any \(i,j\) such that
  \(i>j\).
  there is a  integer $M$ such that for any
  \(c>M\), if \(\vec{f}=\vec{a}+c \vec{1}\),
    then
 $(q,t,a)$-character of the homology of $\mJMa$  is given by the formula
 \[
  \chi_{a,q,t}\left(\HHH(\delta^{\vec{a}+c\vec{1}})\right) =
\sum_{\ytb} \prod_i\frac{z_i^{f_i}(1+az_i^{-1})}{1-z^{-1}_i}\prod_{1\le i<j\le n}\zeta(\frac{z_i}{z_j}),\]
  where \[\zeta(x)=\frac{(1-x)(1-QTx)}{(1-Qx)(1-Tx)},\qquad Q=q^2,\qquad T=t^2/q^2,\]
  while $\ytb$ is the set of all standard Young tableaux, we denote
  \(z_i=Q^{\mathbf{a}'(i)}T^{\mathbf{l}'(i)}\) and \(\mathbf{a}'(i),\mathbf{l}'(i)\) are co-arm and co-leg of the $i$-labeled  square in the standard tableau with \(n\) squares.
\end{theorem}

The  formula in the theorem was conjectured in \cite{GorskyNegutRasmussen16} for the Soergel bimodule triply-graded homology,
some variants of this formula were proven in \cite{Mellit17}, \cite{Hogancamp17}, \cite{EliasHogancamp16}. Let us also
point out that the relation between the coherent sheaves on the Hilbert scheme of points and the triply-graded homology
was suggested in various forms in \cite{OblomkovRasmussenShende12}, \cite{GorskyOblomkovRasmussenShende14}, \cite{GorskyNegut15},
\cite{AganagicShakirov12}, \cite{GorskyNegutRasmussen16}.

Let us also mention \cite{GorskyHogancamp17} where
a link between the sheaves on the isospectral Hilbert schemes and Soergel bimodule triply-graded homology is established.
By pushing forward the sheaf to the usual Hilbert scheme the authors obtain the result analogous to our
theorem \ref{thm:main}. It worth noting that the construction of \cite{GorskyHogancamp17} results in the localization
formulas like the ones in \cite{GorskyNegutRasmussen16} only if the parity condition for the homology holds. 
Thus it is natural to expect a connection between the results of this paper and the methods of \cite{GorskyHogancamp17}.

%

Altogether we construct the Chern and co-Chern functors for three braid related categorifies of matrix factorizations: the unframed, the unstable framed and the stable framed category~\eqref{eq:stfrmf}.
We showed in \cite{OblomkovRozansky17} that the first two categories
provide  representations of the affine braid group, while the third one provides a representation of the ordinary braid group, see section~\ref{sec:braid-realization}.

Let us also mention some technical results in this paper. In section~\ref{sec:comp-push-forw} we introduce weakly equivariant matrix factorizations from our previous work as well as dg weakly equivariant matrix factorizations. We define tamely equivariant matrix factorizations in this section and show that the push-forward and pull-back functors
preserve tameness. We also construct strictification functor that turns tame dg weakly equivariant matrix factorizations to strictly  equivariant matrix factorizations if the action is free. We use the strictification functor to show  in section~\ref{sec:braid-realization} that  the matrix factorizations which  realize the braids are isomorphic to
strongly equivariant matrix factorizations.

In section~\ref{sec:chern-character} we recall the basic theory of
equivariant matrix factorizations. In particular, in section~\ref{sec:MFs} we discuss various versions of the equivariant matrix factorizations and the relations between the
corresponding categories and in \ref{sec:glob-matr-fact} we discuss several definition of global matrix factorizations that were previously studied. 

In section~\ref{sec:comp-push-forw} we remind the main properties of the weakly equivariant matrix factorizations that we used in our previous work.
In particular, in section~\ref{sec:push-f-constr} we also remind a construction of the equivariant push-forward in the setting when weak equivariance defined by means of dg Chevalley-Eilenberg algebra. In  section~\ref{sec:strictification} we study tamely equivariant matrix factorizations introduced in section~\ref{sec:push-f-constr}
and construct  the above mentioned strictification functor.
We also prove a base change statement for matrix matrix factorizations in section~\ref{prop:base-change} and discuss a Kn\"orrer periodicity in \ref{prop:base-change}.

In section~\ref{sec:eq-strs} we remind the construction of the categories of matrix factorizations from our previous work and
in section~\ref{sec:braid-realization} we revisit our previous construction for realization of the braid group inside the category of matrix factorizations.
In section~\ref{sec:construction} define our `unframed' categorify together with its Chern and co-Chern functors \(\HC\)
and \(\CH\), proving their properties.

In section \ref{sec:coh-shv} we introduce the
framed and stable versions of our category and its Chern and co-Chern functors. In this section
we also discuss the linear Koszul duality relating the category of matrix factorizations to the category of coherent sheaves and
we show that we can combine the linear Koszul duality functor with the (co-) Chern functor to obtain the functor~\eqref{eq:mnchcchf}.
In the section~\ref{sec:knot} we prove theorem~\ref{thm:main}.
Finally, in section~\ref{sec:computations} we use Chern functor theory to prove theorem~\ref{thm:loc}.


{\bf Acknowledgments}
We would like to thank Dmitry Arinkin, Ivan Losev, Roman Bez\-ru\-kav\-ni\-kov, Andrei Negu{\c t} and Matt Hogencamp for useful discussions.
We are especially thankful to Tina Kanstrup for reading the manuscript and suggesting many corrections and improvements.
Multiple discussions with Eugene Gorsky guided us through from the start of the project and we  are extremely grateful to him.
Also we are very grateful to Leonid Positselski for his patient explanations about the categories of matrix factorizations,
the strictification construction from section~\ref{sec:strictification} was one of many outcomes of these discussions.
The authors are extremely grateful to an anonymous referee who studied the manuscript very thoroughly and pointed out
many inaccuracies in the earlier versions of the manuscript. 
We  also would like to thank the same referee for many suggestions concerning exposition as well for some improved
versions of our results (see for example remark~\ref{rem:EG}).
The work of A.O. was supported in part by  the NSF CAREER grant DMS-1352398, NSF FRG grant DMS-1760373 and Simons Fellowship. 
The work of L.R. was supported in part by  the NSF grant DMS-1108727.



\section{Matrix factorizations}
\label{sec:chern-character}

\subsection{General facts about matrix factorizations}
\label{sec:categories}

In this section we recall conventions of matrix factorization theory and
 we remind the main properties of the weakly equivariant matrix factorizations that we used in our previous work.

\subsubsection{}
For an affine  algebraic variety \(\calZ\) and a polynomial \(F\in \CC[\calZ]\) Orlov \cite{Orlov04} defines a DG category \(\MMF(\calZ,F)\) whose objects
are `curved' homologically $\ZZ_2$-graded free finite rank  differential modules
\begin{equation}
\label{eq:mfdf}
(M,D) = \bigl(\begin{tikzcd}M_0 \arrow[r,shift left=0.5ex,"D_{01}"] & M_1  \arrow[l,shift left=0.5ex,"D_{10}"]\end{tikzcd}\bigr)
\end{equation}
That is $M = M_0 \oplus M_1$ is equipped with the \(\CC[\calZ]\)-linear homomorphism \(D_{01},D_{10}\) and "curved" means that the square of the total differential $D = D_{01} + D_{10}$ is equal to $F$: $D^2 = F\xId_M$. Let \((M,D),(M',D')\in \MMF(\calZ,F)\) then \[\phi\in \Hom^j((M,D),(M',D'))=\oplus_{i=0,1}\Hom_{\CC[\calZ]}(M_{i},M'_{i+j}).\]
The dg structure of \(\Hom^*((M,D),(M',D'))\) is defined by 
\begin{equation}\label{eq:Hom-dif}
  d\phi=D'\phi-\phi D,\end{equation}

\begin{defn}\label{def:h0dg}
  The triangulated category \(\MF(\calZ,F)\) is defined as \(H^0(\MMF(\calZ,F))\).
  That \(\MF(\calZ,F)\) has the same objects as \(\MMF(\calZ,F)\) but the morphism space
  between \(\calF,\calF'\) is defined as \(H^0(\Hom^*(\calF,\calF'))\)
\end{defn}


In defining the category \(\MF(\calZ,F)\) we did not require smoothness of \(\calZ\).
In our setting only smooth varieties appear, so we assume smoothness of \(\calZ\) from now on. The smoothness assumption drastically simples the theory of matrix factorization, the general case is discussed in \cite{EfimovPositselski15},\cite{Orlov12}.

%

\subsubsection{}\label{sec:KosMFs}
Tensor product operation allows us to construct new matrix factorizations.
\begin{defn}
  Given two potentials \(F,F'\in \CC[\calZ]\) we define the  tensor product  bi-functor:
\[\otimes: \MF(\calZ,F)\times\MF(\calZ,F')\to \MF(\calZ,F+F')\]
as \((M,D)\otimes(M',D')=(M\otimes M',D\otimes 1+1\otimes D')\). Here we use the usual sign conventions for the tensor product of (curved) complexes
to ensure that the differential squares to \(F+F'\). 
\end{defn}

Similarly, one defines the exterior tensor product
\[\boxtimes: \MF(\calZ_1,F_1)\times \MF(\calZ_2,F_2)\to \MF(\calZ_1\times \calZ_2,F+F')\]

If $F$ has a presentation $F = f_1 g_1 +\cdots + f_m g_m$, $f_i,g_i\in\CC[\calZ]$, then one defines a Koszul matrix factorization
\[
\begin{bmatrix}f_1&g_1\\f_2&g_2\\ \vdots&\vdots\\f_m&g_m\end{bmatrix}
= \bigotimes_{i=1}^m \begin{tikzcd}(\CC[\calZ] \arrow[r,shift left=0.5ex,"f_i"] & \CC[
  \calZ] \arrow[l,shift left=0.5ex,"g_i"])\end{tikzcd}
\]
Equivalently, the Koszul matrix factorization can be presented with the help of odd variables $\theta_1,\ldots,\theta_m$: its module is $M = \CC[\calZ]\otimes \CC[\theta_1,\ldots,\theta_m]$, while its differential is $D = \sum_{i=1}^m f_i \theta_i + g_i \frac{\partial}{\partial \theta_i}$. Respectively, \(\ZZ_2\)-graded pieces of \(M\) are
\[M_i=\oplus_{|S|=i \mbox{ mod }  2} \CC[\calZ]\otimes \theta_S,\quad \theta_S=\prod_{s\in S}\theta_s.\]

If $f_1,\ldots,f_m$ form a regular sequence, then the Koszul matrix factorization is independent of the choice of $g_1,\ldots, g_m$ \cite[Lemma 2.2]{OblomkovRozansky16} up to an isomorphism, so we use an abbreviated notation \(\mathrm{K}^F(f_1,\dots,f_m)\in \MF(\calZ,F)\). For the independence we need to require \(\CC[\calZ]\) to have
finite homological dimension, it is true if  \(\calZ\) is smooth.

Let \(A,B\) be a  pair of rectangular matrices, \(A\) of size \(n\times m\) and \(B\) of size \(m\times n\),  with entries from \(\CC[\calZ]\).
In the main body of the paper we also use notation the following Koszul matrix  factorization 
\[[A,B]=[\vec{A},\vec{B^t}]\in \MF(\calZ,\Tr(AB)).\]
Here \(\vec{A}\), \(\vec{B^t}\) vectors of rank \(nm\) that consist of entries \(n\times m\) matrices \(A\), \(B\),
\((\vec{A})_{im+j}=A_{i,j}\), \((\vec{B^t})_{im+j}=B_{j,i}\).

\subsubsection{}\label{sec:quasi}
  In our work we also use categories of matrix factorizations \(\MF(\calZ',F)\) where \(\calZ'\subset \calZ\) is a quasi-affine subvariety of the affine
  variety  \(\calZ\). That is \(\calZ\setminus\calZ'\) is an affine closed subvariety. In our setting \(\calZ'\) is usually a some sort of stable locus.

  Since \(\calZ'\) is quasi-affine, we have an affine cover \(\mathcal{U}'\) of \(\calZ'\), \(\calZ'=\cup_{i\in I}\calZ'_i\), here \(\calZ'_J=\calZ'_{j_1}\cap \dots \cap \calZ'_{j_m}\)  are affine
  for any \(J=\{j_1,\dots,j_m\}\subset I\). For dg sheaf \(\calF_*\) on \(\calZ'\)
  we define \[C^q(\mathcal{U}',\calF_p)=\oplus_{J,|J|=q}\calF_p(\calZ'_J).\]
  Following \cite[Section 2.6.3]{LuntsShnurer16} and \cite[Section 2]{LinPomerleano13} we define dg category \(\MMF_{\check{C}mor}(\calZ',F)\) as follows.
  \begin{defn}\label{def:Cmor}
    The dg category \(\MMF_{\check{C}mor}(\calZ',F)\) has the same objects as \(\MMF(\calZ',F)\) but the space of morphisms between \((M,D)\) and
    \((M,D')\) is defined by
    \[\Hom^k((M,D),(M',D'))=\oplus_{k=p+q+p' } C^q(\mathcal{H}om(\widetilde{M_p},\widetilde{M'_{p'}})),\]
    where \(\widetilde{M_*},\widetilde{M'_*}\) are the sheaves associated with the modules \(M\), \(M'\).
    The differential on \(\Hom^*((M,D),(M',D'))\) is the differential of the total bicomplex: \(d_{tot}=d+d_{\check{C}}\) where
    \(d_{\check{C}}\) is the \v Chech differential and \(d\) is defined by \eqref{eq:Hom-dif}.
  \end{defn}

  Analogously to  definition~\ref{def:h0dg} we define the homotopy category
  \[\MF_{\check{C}mor}(\calZ',\calF)=H^0(\MMF_{\check{C}mor}(\calZ',\calF)).\]
  It is shown in \cite[Corollary 2.51]{LuntsShnurer16} that different affine covers
  result into the equivalent categories. In particular, if \(\calZ'\) is affine then
  the category \(\MF_{\check{C}mor}(\calZ',\calF)\) is equivalent to category
  \(\MF(\calZ,\calF)\). Thus from now on we drop subindex \(\check{C}mor\) since can 
  assume that all our categories are of the above type.

  \begin{remark}\label{rem:Cech}
    If apply the construction of \(\MF_{\check{C}mor}(\calZ',F)\) to \(\calZ\) quasi-projective then we obtain a derived category of matrix factorizations
    (see \cite[Proposition 2.50]{LuntsShnurer16} and discussion in the appendix).
  \end{remark}

\subsection{Equivariant matrix factorizations}
\label{sec:MFs}

In this section we remind the main definitions from the theory of equivariant matrix factorizations as in \cite{OblomkovRozansky16}. We also aim to clarify subtle points of the
construction and explain relations with the
previously developed theory of equivariant matrix factorizations \cite{BallardFaveroKatzarkov13}. We start with an outline of Chevalley-Eilenberg construction and define the category of equivariant matrix factorizations. We compare the equivariant matrix factorizations
with the strictly equivariant matrix factorizations.


\subsubsection{}\label{sec:CEdef}


Consider a group $H$ such that $\mathrm{Lie}(H)=\frh$, $\calZ$ is an affine variety with $H$-action and $\chi^2$-invariant function $F\in \CC[\calZ]^H_{\chi^2}$ where
\(\chi\in \hat{H}\) is a character.

   \begin{defn}
       We define category of strictly
 \(H\)-equivariant matrix factorization \[\MF(\calZ,H,\chi,F)=\{ (M,D)| D(h\cdot m)=h\cdot D(m), D^2=F, h\in H, m\in M\}\] as category consisting of pairs \((M,D)\) where \(M\) is a free  \(\ZZ_2\)-graded  \(\CC[\calZ]\)-module  with
 \(H\)-action and \(D\) is the \(H\)-equivariant curved odd (of \(\ZZ_2\) degree one) differential of degree \(\chi\). The morphisms and homotopies  in this category are assumed to be \(H\)-equivariant.
\end{defn}


Thus defined homotopy category is equivalent to the category of equivariant matrix factorizations was studied in \cite[Section 3]{BallardFaveroKatzarkov13}.
In our previous work we used notation \(\MF^{str}_H(\calZ,F)\) for such categories,
here we adopt notation that is used in \cite{BallardFaveroKatzarkov13}.

In our constructions we use spaces with action of non-reductive groups and we need to define a functorial push-forward along a regular embedding without appealing to the derived category setting. It appears 
 that such (non-derived) push-forward does not exists in the setting of the strictly equivariant matrix factorizations \cite{OblomkovRozansky16}. Thus we need to relax the equivariance
 constrain and define equivariant matrix factorizations with help of Chevalley-Eilenberg  complex (see section~\ref{sec:eqMFdef}).
 
\begin{defn} Chevalley-Eilenberg complex
 $\CE_\frh$ is the complex $(V_\bullet(\frh),d)$ with $V_p(\frh)=U(\frh)\otimes_\CC\Lambda^p \frh$ and differential $d_{ce}$:
 \def\dtheta{d}
 \begin{multline*} d_{ce}(u\otimes x_1\wedge\dots \wedge x_p)=\sum_{i=1}^p (-1)^{i+1} ux_i\otimes x_1\wedge\dots \wedge \hat{x}_i\wedge\dots\wedge x_p+\\
   \sum_{i<j} (-1)^{i+j} u\otimes [x_i,x_j]\wedge x_1\wedge\dots \wedge \hat{x}_i\wedge\dots\wedge \hat{x}_j\wedge\dots \wedge x_p,
   \end{multline*}
\end{defn}
   
\subsubsection{}\label{sec:eqMFdef} In this section we define triangualated category of weakly equivariant matrix factorizations. First we discuss the case when \(\calZ\) is affine.

For a given $\frh$-module $M$ the complex $\CE_\frh\odel M$ has terms $U(\frh)\otimes \Lambda^i(\frh)\otimes M$ with $\frh$-module structure
 $$ x\cdot(u\otimes \omega\otimes m)=x\cdot u\otimes \omega\otimes m,$$
 and the differential of the complex is   $d_{ce}=d_{ce}^1+d_{ce}^2$ where:
 \[
  d_{ce}^1=d_{ce}\otimes 1,\]
$$ d_{ce}^2(u\otimes x_1\wedge\dots \wedge x_p)=  \sum_{i=1}^p u\otimes x_1\wedge\dots \wedge \hat{x}_i\wedge\dots\wedge x_p\otimes x_i\cdot m
$$

A slight modification of the standard fact that $\CE_\frh$ is the resolution of the trivial module \cite[Theorem 7.2.2]{Weibel94} implies that $\CE_\frh\odel M$ is a free over \(U(\frh)\) resolution of the
$\frh$-module $M$.

\begin{defn}\label{def:weak-eq}
Let \(\calZ\) be a smooth variety with action of \(G\), \(H=R_u(G)\) is the unipotent radical of \(G\)  and \(G/H=\bar{G}\) is an algebraic torus.
  We define a dg category $\mathbf{MF}(\calZ,G^{\sim},F)$, \(F\in \CC[\calZ]^G_{\chi^2}\) whose objects are triples $(M,D,\partial)$ where \[M=M_0\oplus M_1,\quad M_i=\CC[\calZ]\otimes V_i,\quad V_i \in \Mod_{G},\quad
\partial\in \oplus_{i>j} \Hom_{\CC[\calZ]}(\Lambda^i\frh\otimes M, \Lambda^j\frh\otimes M)\] and $D$ is an odd endomorphism
$D\in \Hom_{\CC[\calZ]}(M,M)$ such that \(D\) is an odd (of \(\ZZ_2\)-degree one) \(G_{red}\)-equivariant of degree \(\chi\) and
$$D^2=F,\quad  D_{tot}^2=F,\quad D_{tot}=D+d_{ce}+\partial,$$
where the total differential $D_{tot}$ is an endomorphism of $\CE_\frh\odel M$, that commutes with the $U(\frh)$-action and of degree \(\chi\). The morphism space between \(\calF=(M,D,\partial)\)
and \(\calF'=(M',D',\partial')\) is defined as 
\[\Hom^*(\calF,\calF')=\oplus_{i\ge j} \Hom_{\CC[\calZ]}(\Lambda^i\frh\otimes M, \Lambda^j\frh\otimes M')\]
  with differential
  \[d(\Psi)= D'_{tot}\circ \Psi-\Psi\circ D_{tot}.\]
\end{defn}

In the last definition we define action of \(\Psi\) and \(\partial\) on \(\CE_\frh\odel M\)
by requiring
that this action respects \(U(\mathfrak{h})\)-action.

As non-equivariant case define the triangulated category \(\MF_H(\calZ,G^\sim,F)\) as homotopy category
\[\MF(\calZ,G^\sim,F)=H^0(\mathbf{MF}(\calZ,G^{\sim},F)).\]




   If \(\mathcal{U}'\)   is \(G\)-equivariant affine cover of \(\calZ'\),
   \(\calZ'=\cup_{i\in I}\calZ'_i\) then we define the dg category
   \(\MMF_{\check{C}mor}(\calZ',G^{\sim},F)\) by combining the constructions from
   definitions \ref{def:Cmor} and \ref{def:weak-eq}. The argument of  \cite[Corollary 2.51]{LuntsShnurer16} implies that different affine covers result into equivalent triangulated categories  \(\MF_{\check{C}mor}(\calZ',G^\sim,F)=H^0(\MMF_{\check{C}mor}(\calZ',G^\sim,F))\).
Thus we suppress the subindex \(\check{C}mor\) from our notations unless we want to emphasis the Cech aspect of the construction.

Let \(G=G_1\times G_2\) acts on the quasi-affine \(\calZ\) and \(F\in\CC[\calZ]_{\chi^2} \), \(\chi\in \hat{G}\). In this situation one can define
a triangulated category of partially strongly equivariant matrix  factorizations
\[\MF(\calZ,G_1\times G_2^\sim,F).\]
The objects of this category are two periodic complexes of locally-free sheaves that are strongly equivariant with respect to \(G_1\)  and
weakly equivariant with respect to \(G_2\).

Finally, let us relax the finite rank constraint in our definition and introduce the triangulated categories 
\[\MF^{\infty}_{\check{C}mor}(\calZ,G_1\times G_2^\sim,F),\quad \MF^{\infty}_{\check{C}mor}
(\calZ,G_1\times G_2^\sim,F)\]
that consist of the matrix factorizations of possibly infinite rank. 


\subsubsection{Koszul dual description}
\label{sec:Koszul-dual}

The space \(\Lambda^*\frh^\vee\) has a natural structure of dg algebra with the usual exterior algebra multiplication and
differential defined in terms of the structure constants of the Lie algebra. If \(\xi_i\), \(\xi_i^\vee\) are dual bases
of \(\frh\) and \(\frh^\vee\) then:
\[d(\xi^\vee_i)=\sum_{k,l} c^i_{kl}\xi_k^\vee\wedge \xi_l^\vee, \quad [\xi_k,\xi_l]=\sum_i c^i_{kl}\xi_i.\]

Moreover the Chevalley-Eilenberg complex \(\CE_\frh\) is a  free dg module over \(\Lambda^*\frh^\vee\), the generators \(\xi_i^\vee\) act by the contraction.
Thus given an affine space \(\calZ\) with \(H\)-action the space \(\CC[\calZ]\otimes\Lambda^*\frh\) acquires the structure \(\Lambda^*\frh^\vee\) dg module
since \(\CC[\calZ]\otimes \Lambda^*\frh=\CC[\calZ]\otimes_{U(\frh)}\CE_\frh\). The left action of \(\CC[\calZ]\) commutes with the action of \(\Lambda^*\frh^\vee\) thus
\(\CC[\calZ]\otimes \Lambda^*\frh\) is a free \(\CC[\calZ]\otimes \Lambda^*\frh^\vee\)-module of rank \(1\) and a free \(\CC[\calZ]\otimes \Lambda^*\frh^\vee\) module
is a direct sum of copies of \(\CC[\calZ]\otimes \Lambda^*\frh\).

\begin{defn}Let \(\calZ\) be a smooth variety with action of \(G\) and \(H=R_u(G)\),  \(G/H=\bar{G}\) are the radical and the reductive quotient of \(G\).
The dg category  \(\MMF(\calZ,G^{\dg},F)\), \(F\in \CC[\calZ]^G\)  has objects that are pairs  \(M,D\) where \(M\) is \(\ZZ_2\)-graded free dg \(\CC[\calZ]\otimes\Lambda^*\frh^\vee\)-module and \(D\) is a \(\ZZ_2\)-graded  \(\bar{G}\)-equivariant  dg endomorphism \(D\in \Hom_{\CC[\calZ]\otimes \Lambda^*\frh^\vee}(M,M)\) such that
\(D^2=F\). The space of morphisms between \(\calF=(M,D)\) and \(\calF'=(M',D')\) is defined as
\[\Hom(\calF,\calF)=\Hom_{\CC[\calZ]\otimes \Lambda^*\frh^\vee}(M,M)^{\bar{G}}\]
with differential defined by \ref{def:h0dg}.
\end{defn}

Respectively, we define a triangulated category \(\MF(\calZ,G^{\dg},F)=H^0(\MMF(\calZ,G^{\dg},F))\).
Let us point out that a differential  \(D\) of an element \((M,D)\in \MMF(\calZ,G^{dg},F)\), \(M=\bar{M}\otimes \Lambda^*\frh\) is of the form
\begin{equation}\label{eq:d_i}
  D=\sum_{I} D_I\xi_I^\vee, \quad \xi^\vee_I=\xi_{i_1}^\vee\wedge\dots\wedge\xi_{i_k}^\vee, \quad D_I\in \Hom_{\CC[\calZ]}(\bar{M},\bar{M})^{\bar{G}}.\end{equation}
In particular, \(D_\emptyset^2=F\) and \((\bar{M},D_\emptyset)\in \MMF(\calZ,F)\). Moreover, \(D_I\), \(|I|>0\) become the correction differentials under functor
\[{\mathrm{KSZ}}_{U(\frg)}=\CE_\frh\otimes_{\Lambda^*\frh^\vee}:\MMF(\calZ,G^{\dg},F)\to \MMF(\calZ,G^\sim,F).\]

Let us define the functor in the opposite direction by
\[{\mathrm{KSZ}}_{U(\frg)}^*(\calF)=\calF\otimes_{U(\frg)}\CE_\frg\]
\begin{proposition}\label{prop:KSZ-g}
  The functor \(\mathrm{KSZ}_{U(\frg)}\) is fully faithful and 
   the image of the functor is a subcategory that has as objects matrix factorizations
  \((M,D,\partial)\in \MF(\calZ,G^\sim,F)\)  such that \(D^{tot}=D+\partial+d_{ce}\) intertwines
  the \(\Lambda^*(\frh^\vee)\)-action on \(M\otimes\Lambda(\frh)\otimes U(\frh)\).
 \end{proposition}
\begin{proof}

  Given an element \((M,D)\in \MF(\calZ,F,G^{\dg})\) locally \(M\) is a direct sum
  of \(r\) copies of \(\CC[\calZ]\otimes \Lambda^*(\frg^\vee)\). Thus locally the tensor product in the construction of the functor becomes:
  \[\CC[\calZ]\otimes \Lambda^*(\frg)\otimes U(\frg)\oplus\dots\oplus\CC[\calZ]\otimes \Lambda^*(\frg)\otimes U(\frg).\]
  Here we used simplification \(\Lambda^*(\frg^\vee)\otimes_{\Lambda^*(\frg^\vee)} \Lambda^*(\frg)=\Lambda^*(\frg)\).

  Before this simplification the differentials \(D\otimes 1\) and \(1\otimes d_{ce}\) mutually compute. Hence \(D_{tot}=D+d_{ce}\)  satisfies equation \(D_{tot}^2=F\)
  and the triple \((M,D_\emptyset,\{D_{I}\})\) is well-defined element of \(\MF(\calZ,G^\sim,F)\). Since \(D\otimes 1\) and \(1\otimes d_{ce}\) commute we
  we have induced map of the morphism spaces \(\Hom(\calF,\calF')\to \Hom(\mathrm{KSZ}_{\frg}(\calF),\mathrm{KSZ}_{\frg}(\calF'))\)

  Moreover, the differential \(D_{tot}\) is compatible with the action of \(\Lambda^*(\frg^\vee)\). Thus we proved the first statement.

  On the other hand \(\CE_\frg\otimes_{U(\frg)}\CE_\frg\) is of infinite rank over \(\Lambda^*(\frg^\vee)\otimes \Lambda^*(\frg^\vee)\). As a vector space it is
  \(\mathbb{B}=\Lambda^*(\frh)\otimes U(\frh)\otimes \Lambda^*(\frh)\) and it PBW filtration \(F\)
  such that the associate graded space is \(\gr_F(\mathbb{B})=\Lambda^*(\frh)\otimes \CC[\frh^\vee]\otimes\Lambda^*(\frg)\).
  Moreover, the associated graded dg structure on this \(\Lambda^*(\frh^\vee)\otimes\Lambda^*(\frh^\vee)\) module is given by the differential:
 \[gr(D_B)=\sum_{i}\frac{\partial}{\partial \xi_i}\otimes \xi_i+1\otimes \xi_i\otimes \frac{\partial}{\partial \xi_i},\]
 where \(\xi_i\) is a basis of \(\frg\).

 By the usual Koszul duality argument,
 \((gr(\mathbb{B}), gr(D_B))\) is homotopic to the diagonal \(\Lambda^*(\frg^\vee)\)-\( \Lambda^*(\frg^\vee)\) bimodule \(\Lambda^*(\frg)\). Since the Chevalley-Eilenberg
 differential respects the filtration the homotopy lifts from the graded module
 \((\mathbb{B},D_B)\sim \Lambda^*(\frg)\). Thus we get \(\mathrm{KSZ}^*_{\frg}\circ\mathrm{KSZ}_{\frg}=1\)  and the statement follows.
  \end{proof}

  \begin{remark}Since \(\CE_\frh\otimes_{\Lambda^*\frh^\vee}\CE_\frh\) is a resolution of \(U(\frh)\otimes U(\frh)\)
    module \(U(\frh)\) in terms of free \(U(\frh)\otimes U(\frh)\)-modules, we see that
    \(\mathrm{KSZ}_{\frg}\circ\mathrm{KSZ}_{\frg}^*=1\). Thus we obtain an equivalence of categories \(\MF^\infty(\calZ,G^\sim,F)\simeq\MF^\infty(\calZ,G^{\dg},F)\) and have an inclusion of categories:
    \[\MF(\calZ,G^{dg},F)\subset \MF(\calZ,G^\sim,F)\subset \MF^\infty(\calZ,G^\sim,F)=
    \MF^\infty(\calZ,G^{dg},F).\]
\end{remark}

\subsection{Global matrix factorizations}
\label{sec:glob-matr-fact}

There are several variants of definition of matrix factorizations in the literature. Let us quickly outline the variants and explain the relation with the
category studied in this paper.

Let \(\calZ\) be any regular finite dimensional \(H\)-variety and \(F\in \CC[\calZ]^{G}_\chi\) here \(H\) is an algebraic group and
\(\chi\in \hat{H}\) is a character. Following  \cite{BallardFaveroKatzarkov13}, \cite{LinPomerleano13}, \cite{EfimovPositselski15}, \cite{LuntsShnurer16}, \cite{Hirano17} we introduce differential graded (dg) categories
\[\mathbf{Qcoh}(\calZ,H,F),\quad \mathbf{coh}(\calZ,H,F),\quad \mathbf{LFr}(\calZ,H,F),
  \quad \mathbf{lfr}(\calZ,H,F),\quad \mathbf{Inj}(\calZ,H,F),\]
of curved two-periodic \(H\)-equivariant complexes with differentials of degree \(\chi\) \eqref{eq:mfdf} where \(M_i\) are quasi-coherent,
coherent, locally free, locally free of finite rank and injective, respectively. The differential graded structure on the space of morphisms is defined by \eqref{eq:Hom-dif}.

If \(\calZ \) is an affine space the dg category \(\mathbf{lfr}(\calZ,H,F)\) is isomorphic to the dg category \(\MMF^{str}_H(\calZ,F)=\MMF(\calZ,H,F)\) discussed in the section~\ref{sec:categories} and studied in the early papers \cite{Eisenbud80}, \cite{Orlov04}.


The categories \(Z^0(\mathbf{Qcoh}(\calZ,H,F)), Z^0(\mathbf{coh}(\calZ,H,F)),
Z^0(\mathbf{LFr}(\calZ,H,F)), Z^0(\mathbf{lfr}(\calZ,H,F))\) are exact, see for example
\cite[Proposition 3.5]{Hirano17a}. Let us denote \(\mathrm{Acycl}(\calC)\) the smallest thick category that contains all totalizations of the exact sequences of the exact category \(\calC\). The authors of  \cite{BallardFaveroKatzarkov13},\cite{LinPomerleano13}, \cite{EfimovPositselski15}, \cite{LuntsShnurer16}, \cite{Hirano17} in following the construction in \cite{Orlov12}
define the derived categories of the above dg categories  as Verdier quotients:
\[\mathrm{D}(\calC)=H^0(\calC)/\mathrm{Acycl}(\calC).\]

As it is shown \cite{BallardFaveroKatzarkov13},\cite{LinPomerleano13}, \cite{EfimovPositselski15}, \cite{LuntsShnurer16} there are equivalences between these categories:
\[\mathrm{D}(\mathbf{coh}(\calZ,H,F))\simeq \mathrm{D}(\mathbf{lfr}(\calZ,H,F))\simeq \MF_{\check{C}mor}(\calZ,H,F)\]
\[H^0(\mathbf{Inj}_G(\calZ,F))\simeq D(\mathbf{Lfr}_G(\calZ,F))\simeq D(\mathbf{Qcoh}_G(\calZ,F))\simeq\MF_{\check{C}mor}^{\infty}(\calZ,H,F).\]
In the last statement the categories \(\MF_{\check{C}mor}(\calZ,H,F)\), \(\MF_{\check{C}mor}^\infty(\calZ,H,F)\) are equivariant versions of the categories from
the remark~\ref{rem:Cech}. As before we drop subindex \(\check{C}mor\) for brevity.
Non-equivariant version of the  equivalences are shown in \cite[Theorem 2.9, Proposition 2.50]{LuntsShnurer16}, the equivariant version of the equivalences is discussed in section 3 of \cite{BallardFaveroKatzarkov13} (see Corollary 3.13 and Proposition 3.14 there), the most general construction is discussed in \cite{EfimovPositselski15}.

Let \(f:\calZ\to\calY\) be an \(H\)-equivariant map.
Following the blue-prints from the theory of quasi-coherent sheaves the above mentioned authors \cite[Definition 3.36]{BallardFaveroKatzarkov13}, \cite[Section 3.5]{EfimovPositselski15}, \cite[Section 2.5.2]{LuntsShnurer16} define derived functors:
\[Rf_*: \MF^\infty(\calZ,H,f^*(W))\to\MF^{\infty}( \calY,H,W),\]
\[Lf^*:\MF(\calY,H,W)\to\MF( \calZ,H,f^*(W)),\quad f^*: \MF^{\infty}(\calY,H,W)\to\MF^{\infty}( \calZ,H,f^*(W)). \]

In listed paper the basic properties of these functors are proven.

\begin{proposition}\cite{EfimovPositselski15},\cite{BallardFaveroKatzarkov13},\cite{LuntsShnurer16}\label{prop:push-f-list}
  \begin{enumerate}
  \item There is an adjunction \(Lf^*,Rf_*\) of triangulated functors.
  \item If restriction of  \(f\) on the critical locus of \(W\) is proper then
    \(Rf_*\) sends  \(\MF(\calZ,H,f^*(W))\) to \(\MF( \calY,H,W)\)
  \item If \(f\) is proper and affine then \(Rf_*=f_*\). If \(f\) is flat then
    \(Lf^*=f^*\)
    \item \(Rf_*(\calF)\overset{L}{\otimes}\calE=Rf_*(\calF\overset{L}{\otimes}Lf^*(\calE))\)
  \end{enumerate}
  
\end{proposition}


 In section~\ref{sec:comp-push-forw} we discuss the push-forward and  pull-back functors in weakly equivariant setting. For two \(H\)-varieties \(\calZ,\calY\) with \(H\)-invariant  potentials \(F,W\) and  an equivariant map \(f\colon \calZ\to \calY\) such
that \(f^*(W)=F\), there is a  pull-back functor \(f^*\colon \MF(\calY,H,W)\to \MF(\calZ,H,F),\) since a pull-back of a  free module is free.
Moreover, if \(f\) is a smooth projection or a regular embedding then there is a well-defined
push-forward functor \[f_*\colon \MF(\calZ,H,F)\to \MF^\infty(\calY,H,W)\] which is a right adjoint to $f^*$, see \cite{OblomkovRozansky16} and corollary~\ref{cor:adj} below.

Most importantly, in section~\ref{sec:comp-push-forw} we define {\it tame} weakly equivariant matrix factorizations. Under some assumption of freeness of the group action,
proposition~\ref{prop:strict} shows equivalence  between the category tame weakly equivariant matrix factorizations and strictly equivariant matrix factorizations discussed in this sections. In section~\ref{sec:braid-realization} we remind the
braid realization from \cite{OblomkovRozansky16} and explain why {\it all} equivariant
matrix factorizations in \cite{OblomkovRozansky16} are tame.

In particular, let us also remark that the push-forward of a finite rank module along a smooth map \(f:\calZ\to \calY\) with  positive dimensional fibers has infinite rank. However, in all cases
considered in our paper we apply non-proper push-forward  in the situation when we  have properness over the critical locus and we can apply part three of proposition~\ref{prop:push-f-list} and proposition~\ref{prop:strict}.


\section{Weakly Equivariant matrix factorizations}
\label{sec:comp-push-forw}

\subsection{Push-forward: a construction}\label{sec:push-f-constr}
The main technical advantage of the weakly equivariant matrix factorizations over strictly equivariant matrix factorization is existence of an elementary construction for  the regular push-forward \ref{prop:push-f-star}:
\begin{equation}\label{eq:j}
  j_*: \MF(\calZ_0,H^\sim,F|_{\calZ'})\to \MF(\calZ,H^\sim,F),\end{equation}
    where \(\calZ_0\subset \calZ\) is defined by the \(\frh\)-invariant ideal whose generators form a regular sequence \cite{OblomkovRozansky16}.
    The push-forward functor
    satisfies the projection formula and the smooth base change, see proposition~\ref{prop:base-change} and proposition~\ref{prop:projection} below.

    The construction for \(j_*\) from \cite{OblomkovRozansky16} relies on an iterative algorithm. Here we explain the algorithm for the push-forward in dg setting. Our previous work relies on the formal properties of \(j_*\) which hold in the dg setting as explain below. In the proposition below we articulate preservation of tameness under the push-forward.

    \begin{defn}
  The subcategory \(\MF^{tame}(\calZ,G^{dg},F)\subset \MF(\calZ,G^{dg},F)\) has objects \((M,D,\partial)\) with the following property. For any \(I\)
  operator \(D_I\) as in \eqref{eq:d_i} respects some finite filtration of locally-free \(\CC[\calZ]\)-modules \(\mathfrak{F}_i(M)\subset M\),
  \(\mathfrak{F}_0(M)=M\), \(\mathfrak{F}_N(M)=0\) \(\mathfrak{F}_i(M)\subset\mathfrak{F}_{i-1}(M)\), \(D_I(\mathfrak{F}_i(M))\subset \mathfrak{F}_{i+1}(M)\) and
  \(\mathfrak{F}_{i}(M)/\mathfrak{F}_{i+1}(M)\) is locally-free.
\end{defn}


    To distinguish the push-forward functor from \cite{OblomkovRozansky16} from the push-forward from the previous section we use notation \(j_{\star}\). Later, see \ref{prop:star-star} we show that the push-forward coincides with the usual push-forward.
    \begin{proposition}\label{prop:push-f-star}
      Suppose \(j:\calZ_0\to \calZ\) is are regular \(H\)-equivariant embedding and
      \(F\in \CC[\calZ]^G\)
      Then there is a well-defined functor
      \[j_\star:\MF(\calZ_0,H^{dg},j^*(F))\to \MF(\calZ,H^{dg},F).\]
      Moreover, this functor preserves tameness:
      \[j_\star:\MF^{tame}(\calZ',H^{dg},j^*(F))\to \MF^{tame}(\calZ,H^{dg},F)\]
    \end{proposition}
    \begin{proof}
      In our paper \cite[Section 3]{OblomkovRozansky16} we construct functor \(j_*\) as in \ref{eq:j}.
      Construction of the functor \(j_\star\) almost verbatim repeats the construction from \cite{OblomkovRozansky16} but we provide details that are concerned with the tameness condition. In particular, the functor \(\mathrm{KSZ}_{U(\frg)}\) from \ref{prop:KSZ-g} intertwines  \(j_*\) and \(j_\star\).

      First, we discuss the case when \(H\) is reductive. Let \(\calF=(M,D)\in \MF(\calZ_0,H,j^*(F))\), \(M=\CC[\calZ_0]^n\) be a strongly \(H\)-equivariant matrix  factorization. First, we construct a pair \((\tilde{M},\tilde{D})\), \(\tilde{M}=\CC[\calZ]^n\) that lifts \((M,D)\), that is \(\tilde{D}|_{\calZ_0}=\calD\).

      A vector space of all lifts of a scalar multiple of \(\tilde{D}\) is a non-zero vector space and it is \(H\)-representation. Since \(\calD\) is \(H\)-invariant
      the last \(H\)-representation has a summand that is a trivial \(H\)-representation. Thus there is an \(H\)-invariant lift \(\tilde{D}\).

      The structure sheaf of \(\calZ_0\) has a free \(H\)-equivariant \(\CC[\calZ]\)-resolution
      \((C_\bullet,d^+)\). Since, embedding is regular we have
      \begin{equation}\label{eq:pos-Hom}
        \Hom^{<0}_{d^+}(C_\bullet,C_\bullet)=0.\end{equation}

      Next, we observe that \(\tilde{D}^2-j^*(F)\) vanishes on \(\calZ_0\) hence there is \(h^{(-1)}\in \Hom^{(-1)}(C_\bullet,C_\bullet)\otimes \Hom(\tilde{M},\tilde{M})\) such that \([d^+,h^{(-1)}]=\tilde{D}^2-j^*(F)\).  Since \(d^+\) and \(\tilde{D}^2-j^*(F)\) the same argument as before implies that we can choose \(h^{(-1)}\) to be \(H\)-equivariant.

      Let \(D^{(1)}=D+d^+-h^{(-1)}\). By the above discussion \((D^{(1)})^2-j^*(F)\in
      \Hom^{<-1}(C_\bullet,C_\bullet)^H\otimes\Hom(\tilde{M},\tilde{M})^H\). One can show
      that last expression commutes \(d^+\).

      We can use \ref{eq:pos-Hom} to construct step by step  correction \(h^{(-2i-1)}\in  \Hom^{-2i-1}(C_\bullet.C_\bullet)^H\otimes\Hom(\tilde{M},\tilde{M})^H \) such that
      \(D^{(-2j+1)}=D+d^++\sum_{k=0}^{j-1}h^{(-1-2k)}\) satisfies
      \[ (D^{(-2j+1)})^2-j^*(F)\in  \Hom^{<-2j}(C_\bullet,C_\bullet)^H\otimes\Hom(\tilde{M},\tilde{M})^H.\]
      The details of argument can be found in the proof of \cite[Lemma 3.1]{OblomkovRozansky16}. Since the complex \(C_\bullet\) is bounded, some large \(N\),
      \(D^{(-2N+1)}\) satisfies \((D^{(-2N+1)})^2=j^*(F)\).

      It is also shown in \cite[Lemma 3.2]{OblomkovRozansky16} that element \(\tilde{\calF}=(\tilde{M},D^{(-2N+1)})\) is unique up to isomorphism. Thus we can define
      \(j_\star(\calF)=\tilde{\calF}\).

      Now we consider the opposite case \(H\) is a unipotent group. This case was also treated in \cite[Lemma 3.5]{OblomkovRozansky16}. But the tameness and dg structure is not discussed there so provide this discussion here. The case of general \(H\) is a straight-forward combination of the reductive and unipotent case.

      As before we start with \(\calF=(M,D,\delta)\in \MF(\calZ_0,H^{dg},F)\) and \((C_\bullet,d^+)\) being \(H\)-equivariant resolution of the structure sheaf of \(\calZ_0\). By the previous
      argument there is \((\tilde{M},\tilde{D})\in \MF(\calZ,j^*(F))\), \(\tilde{D}=D+d^++d^{-}\), \(d^-=\sum_{k=0} h^{(-1-2k)}\) that extends
      \(\calF\) from \(\calZ_0\). However, we can not make \(d^-\) \(H\)-equivariant since \(H\) is not reductive.

      Here everywhere below we use short-hand notation \(\Hom^k=\Hom^k(C_\bullet,C_\bullet)\). Given \(A\in \Hom(\Lambda^\bullet\frh,\Lambda^\bullet\frh)\otimes \Hom(C_\bullet,C_\bullet)\) we also use notation \([A]_k\),  
for the part in \(\Hom(\Lambda^\bullet\frh,\Lambda^\bullet\frh)\otimes \Hom^k\).

      Since \((\delta+d_{ce})^2=[(\tilde{D}+\delta+d_{ce})^2-j^*(F)]_0\) vanishes on \(\calZ_0\), there is \(\tilde{h}^{(-1)}\in \Hom(\Lambda^\bullet\frh,\Lambda^{<\bullet})
      \otimes \Hom^{-1}\) such that \([\tilde{h}^{(-1)},d^+]=(\delta+d_{ce})^2\).
       Finally, let us observe that since \(d_{ce},\delta\) are polynomials of
      \(\xi_i^\vee\) as in \ref{prop:KSZ-g}, we choose \(\tilde{h}^{(-1)}\) to be of form
      \(\sum_i \phi_i\xi_i^\vee\), \(\phi_i\in \Hom^{-1}\).
      Let \(D^{(0)}=\tilde{D}+\delta+d_{ce}+\tilde{h}^{(-1)}\). By construction
      \( (D^{(0)})^2-F\in \Hom(\Lambda^\bullet\frh,\Lambda^\bullet\frh)\otimes \Hom^{<0}\).
      
      As explained, in \cite[Lemma 3.5]{OblomkovRozansky16} we can iterate the above process and obtain \(h^{(-k)}\in \Hom(\Lambda^\bullet\frh,\Lambda^\bullet\frh)\otimes \Hom^{-k}\) such that:
      \[ (D^{(k)})^2-j^*(F)\in \Hom(\Lambda^\bullet\frh,\Lambda^\bullet\frh)\otimes \Hom^{<-k},\quad D^{(k)}=D^{(k-1)}+h^{(-k)}.\]
      Moreover, by the above remark we can assume that \(h^{(-k)}\) is polynomial in
      \(\xi_i^\vee\) of degree \(k\).

      Thus for \(N\) greater that the length of the complex \((C_\bullet,d^+)\) we get
      \((D^{(N)})^2=j^*(F)\)  hence we can set \(j_\star(\calF)= (\tilde{M},\tilde{D},\tilde{\partial})=\) where \(\tilde{\partial}=\delta+\sum_k \tilde{h}^{(-k)}\).

      It is shown in \cite[Lemma 3.6]{OblomkovRozansky16} that the 
      different choices of \(\tilde{h}^{(k)}\) result in isomorphic matrix  factorizations. Also in \cite[Lemma 3.7]{OblomkovRozansky16} we show that the iterative process from above can be used to lift morphism between the objects. Thus we obtain a well-defined functor.

      To show the last statement of the proposition we observe that the coefficient
      \([\xi_I^\vee](\tilde{\partial})\) in  front of \(\xi_I^\vee\) in \(\tilde{\partial}\) is the sum of
      \[[\xi_I^\vee](\delta)\in \Hom(\Lambda^\bullet\frh,\Lambda^\bullet\frh)\otimes \Hom^0,\quad [\xi_I^\vee](\sum_k \tilde{h}^{(-k)})\in \Hom(\Lambda^\bullet\frh,\Lambda^\bullet\frh)\otimes \Hom^{<0}, \]
      Since \([\xi_I^\vee](\delta)\) is strictly upper-triangular \([\xi_I^\vee](\tilde{\partial})\) is strictly-upper triangular too.

    \end{proof}




 

\subsection{Properties of push-forward functors}
\label{sec:prop-push-forw}

In this section we show that the push-forward functors \(j_*\) and \(j_\star\) satisfy properties analogous to the properties listed in proposition \ref{prop:push-f-list}.
The key property is the fourth property in \ref{prop:push-f-list}.

\begin{proposition}\label{prop:projection}
  Let \(j:\calZ\to \calY\) be a regular \(G\)-equivariant embedding then
  \[j_*(\calF\otimes j^*(\calG))=j_*(\calF)\otimes\calG,\quad j_\star(\calF\otimes j^*(\calG))=j_\star(\calF)\otimes \calG\]
  for \(\calF\in \MF(\calZ,G^\sim,W),\calG\in\MF(\calY,G^\sim,W')\) and  \(\calF\in \MF(\calZ,G^{dg},W),\calG\in\MF(\calY,G^{dg},W')\).
  \end{proposition}
  \begin{proof}
    Let us discuss the second equation since the first equation has almost the same proof. Let \(\calF=(M,D_1,\partial_1)\) and \(\calG=(N,D_2,\partial_2)\).

    The pull-back functor \(j^*\) is defined by the restriction of the differentials   \[j^*(\calG)=(M|_\calZ,D_2|_\calZ,\partial_2|_\calZ).\] As explained in the proof of proportion~\ref{prop:push-f-star}, at the first step of construction
    of \(j_\star(\calF\otimes j^*(\calG))\) we need to choose a lift of \(\CC[\calZ]\)-modules
    \(N\otimes M|_\calZ \) and of the differentials \(D_1,\partial_1\),
    \(D_2|_\calZ,\partial_2|_\calZ\). Let \(\tilde{N}\), \(\tilde{D}_1,\tilde{\partial}_1\) be the lifts for \(\calF\). Clearly, \(\calG\) is
    a lift of \(\calG|_\calZ\).

    On the other hand to construct \(j_\star(\calF)\otimes \calG\) we need to pick a lift
    \(\tilde{\calF}\) of \(\calF\). We can choose the lift to be \((\tilde{N},\tilde{D}_1,\tilde{\partial}_1)\) from the previous paragraph.

    To complete construction of \(j_\star(\calF\otimes j^*(\calG))\) and of
    \(j_\star(\calF)\otimes \calG\) we need perform perturbative procedure on the curved complex \(\tilde{\calF}\otimes \calG\otimes C_\bullet\), here \(C_\bullet\) is a resolution of \(\CC[\calZ]\) in terms of locally free \(\CC[\calY]\)-modules.
    Since the outcome of the perturbative procedure is defined (up-to an isomorphism)
    by the initial data, the statement follows.
  \end{proof}

  By a standard argument we derive

  \begin{corollary}\label{cor:adj} Let \(j:\calZ\to \calY\) be a regular embedding.
    Then pairs functors \(j_*,j^*\), \(j_\star,j^*\) from the precious proposition
    are adjunction pairs.
  \end{corollary}

\subsection{Strictification}
\label{sec:strictification}

In the remark~\ref{rem:tame} (see below) suggests that in general we can not expect equivalence between the categories \(\MF(\calZ,G^{dg},F)\) and \(\MF(\calZ,G,F)\) even if the \(G\)-action is free.  However, the equivariant push-forward used for matrix factorizations  in our previous papers \cite{OblomkovRozansky16} preserves slightly smaller category
\(\MF^{tame}(\calZ,G^{dg},F)\) which we define above. The main result of this section is a construction of strictification for  the elements of
the tame subcategory \(\MF^{tame}(\calZ,G^{dg},F)\) under the assumption of free \(H=R_u(G)\)-action.

In this section we relate the tame matrix  factorizations to the matrix factorizations on the quotient. First notice that if the action of \(G\) on \(\calZ\) is free then
there is a natural inclusion functor \[\mathbf{I}: \MF(\calZ/G,F)\to \MF(\calZ,G^{dg},F).\] The image of functor consists of the equivariant
matrix  factorizations with vanishing correcting differentials. It turns out that that
the category of tame equivariant matrix  factorizations retracts to the image of \(\mathbf{I}\)

\begin{proposition}\label{prop:strict}
  Let us assume that \(G\) acts freely on \(\calZ\) and \(H^1(\calZ/G,\QQ)=0\). 
  Then there is an equivalence of categories:
  \[\mathbf{Q}: \MF^{tame}(\calZ,G^{dg},F)\to \MF(\calZ/G,F)\]
   that is \(\mathbf{Q}\circ \mathbf{I}\simeq \mathrm{Id}\) and there is an invertible natural
  transformation:
  \(\mathbf{I}\circ \mathbf{Q}\Rightarrow \mathrm{Id}.\)
\end{proposition}
\begin{proof}
  Given \(\calF,\calF'\in \MF(\calZ,G^{dg},F)\) the space of morphisms
  \(\Psi\in\Hom(\calF,\calF')\) is filtered by Chevalley-Eilenberg degree and
  we denote by \(gr_i(\Psi)\) the corresponding associated graded piece.

  To prove the statement it is enough to construct for each \[\calF=(M,D,\partial)\in
  \MF^{tame}(\calZ,G^{dg},F)\] an element \(\calF'=(M,D',0)\in \MF(\calZ,G^{dg},F)\) and
  an invertible \(\Psi_{\calF}\in \Hom(\calF,\calF')\) such that \(gr_0(\Psi)\) is invertible.
  Indeed, we can define \(\mathbf{Q}(\calF)=\calF'\) and for \(\phi\in \Hom(\calF_1,\calF_2)\), \(\mathbf{Q}(\phi)=\Psi_{\calF_2}^{-1}\circ \phi\circ \Psi_{\calF_1}\).  Respectively, the morphisms \(\Psi_{\calF}\) define the natural
  transformation  \(\mathbf{I}\circ \mathbf{Q}\Rightarrow \mathrm{Id}\) from the statement.

  First we construct the element \(\Psi_{\calF}\) in the case \(\calZ=\calZ'\times G\)
  and the group action is by the left multiplication on last factor. Since \(\frh=\Lie(H)\)
  is nilpotent, there is a non-zero element from the center of  \(\frh\). Let us assume that this element is \(\xi_1\), that is a first element of the basis of \(\frh\) used in \eqref{eq:d_i}.

  The \(H\) is nilpotent and we can assume \(H=\CC^N\) coordinates \(\xi_i^\vee\) and action of \(\xi_1\) on \(\CC[H]\) is the derivative \(\frac{\partial}{\partial \xi_1^\vee}\).
  Since each \(D_{I}\)  is strictly upper-triangular, there are  elements \(\phi_I\in \Hom_{\CC[\calZ]}(\bar{M},\bar{M})\), that is matrices with entries in \(\CC[\calZ]\), such that
  \begin{equation}\label{eq:cycle}
    \frac{\partial \phi_I }{\partial \xi_1^\vee}=D_{I\cup 1}\phi_I.\end{equation}

  Thus we define \(\Psi_{\calF}^{(1)}=\Id+\sum_I \phi_I\xi^\vee_I\). By construction
  \(\Psi_{\calF}^{(1)}\) is invertible and the total differential \(D^{(1)}_{tot}=(\Psi_{\calF}^{(1)})^{-1} D_{tot} \Psi_{\calF}^{(1)}\) does not contains correction differentials that contain \(\xi_1^\vee\). Hence the corresponding matrix  factorization \(\calF^{(1)}\) is naturally an object of category \(\MF(\calZ^{(1)},(G^{(1)})^{dg},F)\) where \(\calZ^{(1)}=\calZ'\times G^{(1)}\),
  \(G^{(1)}=G/H^{(1)}\), \(\mathop{Lie}(H^{(1)})=\langle\xi_1\rangle\).

  By iterating this construction we reduce to the case \(R_u(G)=\{1\}\) where the problem is trivial. To complete our proof we need to cover \(\calZ/G\) by affine
  charts  \(\calZ'_i\), \(i=1,\dots,M\) such that \(\calZ=\bigcup_{i} \calZ'_i\times G\).
  For each chart we have element \(\Psi_{\calF_i}\), \(\calF_i=\calF|_{\calZ'_i\times G}\) as constructed above.

  Over the intersections \(\calZ'_{ij}=\calZ'_i\cap \calZ'_j\) the
  elements may differ, that is\(\Psi_{ij}=\Psi_{\calF_i}\Psi_{\calF_j}^{-1}\) may be non-trivial. However, elements \(\psi_{ij}=\Psi_{ij}-1\) are constant nilpotent matrices by the uniqueness theorem for the differential equation \eqref{eq:cycle}.

  Thus \(\psi_{ij}\) form cocycle for cohomology group \(H^1(\calZ/G,\frh)\). Since
  \(H^1(\calZ/G,\QQ)=0\) and \(\frh\) is nilpotent, we have \(H^1(\calZ/G,\frh)=0\) and
  \(\Psi_{ij}=\Upsilon_i\Upsilon_j^{-1}\). By twisting \(\Psi_{\calF_i}\) with \(\Upsilon_i\)  one obtains a collection \(\Psi_{\calF,\calZ'_i}\)  which form a Cech
  presentation of the desired element \(\Psi_{\calF}\).
  
\end{proof}

\begin{remark} In the proposition
  the condition of freeness of the action of \(G\) can be relaxed to the condition of freeness of \(H=R_u(G)\)-action. Then get an equivalence of categories:
  \[\MF^{tame}(\calZ,G^{dg},F)\simeq \MF(\calZ/H,\bar{G},F).\]
\end{remark}

\begin{remark}\label{rem:tame} This remark was explained to us by Leonid Positselski.
  The tameness condition in the proposition is essential. To see that let us consider the case \(\calZ=U\) and \(G=U=\CC\) is a group of unipotent two by two matrices. The group \(G\) acts freely on the \(\calZ\) by the left multiplication and
  we would like to compare \(\MF(\calZ,U^{dg},0)\) with the category of
  two periodic complexes of vector spaces \(D^{per}(\pt)\). These categories do not appear to be equivalent, the former is a  two-periodic version of
  category of \(D\)-modules on \(\CC\).

  More concretely, if \(z\) is a coordinate along \(\calZ\) then a generator of \(\Lie(U)\) acts by \(\frac{d}{dz}\). A simplest
  object \(\MF(\calZ,U^{dg},0)\)  is a triple \(\calF=(\CC[z],\partial, 0)\) where \(\partial=f(z)\in \Hom(\CC[z],\CC[z])\). To find the intertwining operator \(\Psi_{\calF}\) from the previous proof we need to solve a differential equation
  \(\psi'=f\psi\). This differential equation  has  a non-zero polynomial solution if and only if \(f=0\).
  
\end{remark}

It is clear from our constructions that the strictification functor \(\mathbf{Q}\) intertwines
pull-back functors.
Moreover the strictification functor relates the push-forward and pull-back functors. 

\begin{proposition}\label{prop:star-star}
  Let \(j:\calZ\to \calY\) be a regular embedding and the conditions of proposition
  \ref{prop:strict} hold for \(\calZ\) and \(\calY\). Then functor \(\mathbf{Q}\) intertwines \(j_*\) with \(j_\star\).
\end{proposition}

\begin{proof}
  The functor \(\mathbf{Q}\) intertwines \(j^*:\MF^{tame}(\calZ,H^{dg},F|_\calZ)\to
  \MF^{tame}(\calY,H^{dg},F)\) with  \(j^*:\MF(\calZ/H^,F|_\calZ)\to
  \MF(\calY/H,F)\). On the other hand the first \(j^*\) is a left adjoint of \(j_\star\)
  and the second \(j^*\) is a left adjoint of \(j_*\). Hence the statement follows.
  \end{proof}

From now on we use notation \(j_*\) for \(j_\star\). In the cases we are most interested in the conditions of proposition~\ref{prop:strict} are satisfied so no confusion can possibly occur.

\subsection{Tensor products and quotients}    

    Suppose \(H=H_1\times H_2\), \(\frh=\frh_1\oplus \frh_2\) and \(H_i\) acts on \(\calZ_i\).
    In this situation we have a well-defined functor:
    \[\boxtimes: \MF(\calZ_1,H^\bullet_1,F_1)\times \MF(\calZ_2,H_2^\bullet,F_2)\to \MF(\calZ_1\times\calZ_2,H^\bullet,F_1+F_2),\]
    \[(M_1,D_1,\partial_1)\times (M_2,D_2,\partial_2)
    \mapsto (M_1\otimes M_2, D_1\otimes 1+1\otimes D_2, \partial_1 \otimes 1+1\otimes \partial_2).\]
  Here \(\bullet  \) can be either \(\sim\) or \(dg\). If \(\bullet=\emptyset\) then
  the correction differentials \(\partial_i\) vanish we get a standard definition of the
  exterior tensor product from \cite{BallardFaveroKatzarkov13}. Not all matrix
  factorizations from the category \(\MF(\calZ_1\times\calZ_2,H^\bullet,F_1+F_2)\)
  are of the above form.


    Finally, let us discuss the quotient map. The complex \(\CE_\frh\) is a resolution of the trivial \(\frh\)-module by free modules. Thus the correct derived
    version of taking \(\frh\)-invariant part of the matrix factorization \(\mathcal{F}=(M,D,\partial)\in\MFs(\calZ,H^\bullet,F)\), \(F\in\CC[\calZ]^\frh\),
    \(\bullet=\sim,dg\) is
\[\CE_\frh(\mathcal{F}):=(\CE_\frh(M),D+d_{ce}+\partial)\in\MFs(\calZ//H,F),\]
where \(\calZ//H:=\mathrm{Spec}(\CC[\calZ]^\frh )\) and use the general definition  for the complex of \(\frh\)-modules \(C_\bullet\):
\begin{equation}\label{eq:CE-def}
  \CE_\frh(C_\bullet):=\Hom_\frh(\CE_\frh(\CC[\calZ]),C_\bullet)=\mathrm{RHom}^*_{\frh}(\CC[\calZ],C_\bullet).
  \end{equation}
 The differential of the complex \(\CE_\frh(\CC[\calZ])\) inside
\(\Hom_\frh(\CE_\frh(\CC[\calZ]),-)\) commutes with the differential \(D+d_{ce}+\partial\) hence the latter differential descends  to
\(\CE_\frh(M)\).

If the potential \(F\) vanishes and \(\calF\) is strongly equivariant
then \(\CE_\frh(\calF)\) is the \(\CC[\calZ//H]\)-module of (derived) equivariant sections of \(\calF\). Similarly, in the case of non-vanishing \(F\) and strongly
equivariant matrix factorization \(\calF=(M,D)\) the \ChE matrix factorization
\(\CE_\frh(\calF)\) is homotopy equivalent to  a restriction of differential \(D\) to the \(\CC[\calZ//H]\)-module \(M^H\). We use notation
\[\calF^H\in \MF(\calZ//H,F)\]
for this restricted  matrix factorization. In particular, if the conditions of the proposition~\ref{prop:strict} are satisfied then
\[\CE_{\frh}(\calF)\simeq\mathbf{Q}(\calF).\]

The duality functor inverts the sign of the potential and dualizes  the underlying module:
\[\MF(\calZ,H^\bullet,F)\ni\calF\mapsto \calF^\vee\in\MF(\calZ,H^\bullet,-F),\quad (M,D,\partial)\to (M^\vee,D^\vee,\partial^\vee),\]
where \(M^\vee=M^*\) is the \(\frh\)-module that is dual to \(M\) and \(\bullet=\sim,dg\). The differentials are defined in terms of dual maps:
\[D^*_i: M_i^*\to M_{i+1}^*,\quad \partial^*\in \oplus_{i>j}\Hom(M_\bullet^*\otimes\Lambda^i\frh,M_\bullet^*\otimes \Lambda^j\frh).\]

Let us decompose the morphism \(\partial^*\) into isotypical components:
\[\partial^*=\sum_{i,j}\partial^{i,j *}_{k,l},\quad \partial_k^{i,j*}\in \Hom(M_k^*\otimes \Lambda^i\frh,M^*_l\otimes \Lambda^j\frh).\]
With these conventions we define
\[D^\vee_i=(-1)^i D^*,\quad \partial_{k,l}^{i,j\vee}=(-1)^{(k-l)k}(-1)^{(i-j+1)i}\partial_{k,l}^{i,j*}.\]
The signs pattern in the above formulas guarantees that the total differential \(D^\vee+\partial^\vee+d_{ce}\) squares to \(-F\). Also the space 
of derived homomorphisms
\begin{equation}\label{eq:ext-def}
\calE xt(\calF,\calG)=\CE_{\frh}(\calF^\vee\otimes\calG)\end{equation}
is a two-periodic complex of \(\CC[\calZ//H]\)-modules and the even part of its homology is the space of morphisms
\(\Hom(\calF,\calG)\). Let us also remark that in the case of trivial \(H\)  and \(\calZ\) is quasi-affine there is an isomorphism between \(\calE xt(\calF,\calG)\) and
\(\calE xt(\calG,\calF)\), because there is an isomorphism between a free module and its dual. 
This symmetry does not hold in the case when \(H\) is non-trivial.

Finally, let us point out that the space of morphisms between \(\calF,\calG\) is computed by computing homology:
\[\Hom^*(\calF,\calG)=H^*(\calE xt(\calF,\calG))\]


In the equivariant setting (i.e \(f\) is \(H\)-equivariant) the push-forward \(f_*\) commutes with
Chevalley-Eilenberg    functor.

\begin{proposition}
  Let \(f: \calZ\to\calY\) be an \(H\)-equivariant map that is either regular embedding or flat affine then
  \begin{equation} \label{eq:pushf-ce}
    \CE_\frh(f_*(\calF))=f_*\CE_{\frh}(\calF).
  \end{equation}
for any \(\calF\in \MF^{\infty}_H(\calZ,f^*(W))\).  
\end{proposition}
\begin{proof}
  By the conditions of the proposition the push-forward \(f_*\) and pull-back \(f^*\)
  are exact when restricted to an affine chart \(\calY_J\), \(\calZ_J=f^{-1}(\calY_J)\),
  \(\calZ=\bigcup_i \calZ_i\), \(\calY=\bigcup_i\calY_i\).
Hence \(f_*\) is the left adjoint of the pull-back \(f^*\), \(f^*(\calO_{\calY_J})=\calO_{\calZ_J}\):
\begin{equation*}
  \CE_\frh(f_*(\calF))=\mathrm{RHom}_{\frh}(\CC[\calY_J],f_*(\calF))=
  f_*\mathrm{RHom}_\frh(f^*(\CC[\calY_J]),\calF)=
  f_*\CE_{\frh}(\calF).
\end{equation*}
Since space of morphism is computed with the Cech complex, the local computation above
implies the statement.
\end{proof}

 

\subsection{Base change}
\label{sec:base-change}

Just as in the case of coherent sheaves we have the smooth base change isomorphism.
In more details, suppose we have affine manifolds \(\calZ,\calZ',\calY,\calY'\), \(\calZ'=\calZ\times_{\calY'}\calY\) and
the corresponding potentials \(F, F',W,W'\) that
fit into the commuting diagrams of \(H\)-equivariant maps of spaces:
\[\begin{tikzcd}
    \calZ'\ar[r,"g'"]\ar[d,"f'"]&\calZ\ar[d,"f"]\\
    \calY'\ar[r,"g"] &\calY
  \end{tikzcd},\quad \quad
\begin{tikzcd}
    \MF(\calZ',H^\bullet,F')\ar[d,"f'_*"]&\MF(\calZ,H^\bullet,F)\ar[l,"g^{\prime *}"]\ar[d,"f_*"]\\
   \MF( \calY',H^\bullet,W') &\ar[l,"g^*"]\MF(\calY,H^\bullet,W)
  \end{tikzcd},
\]
where \(\bullet=\sim, dg\).

\begin{proposition}\label{prop:base-change}
Let's assume that the map \(f\) are either smooth affine projections or regular embeddings.
Then if \(g\) is a flat map then we have a natural transformation
that identifies the functors:
\[g^*\circ f_* \simeq f'_*\circ g^{\prime*}.\]
\end{proposition}
\begin{proof}
  Let \(\calF=(M,D,\partial)\in \MF(\calZ,H^\bullet,F)\). If \(f\) is a smooth affine projection then \(f_*\) is an exact functor which is defined by the scalar restriction from \(\CC[\calZ]\) to \(\CC[\calY]\). Respectively, \(g^*\) is a functor of tensor
  product \(\otimes_{\CC[\calZ]}\CC[\calY']\).

  On the other hand \(g^{\prime*}\) is a functor of the tensor product \(\otimes_{\CC[\calZ]}\CC[\calZ']\). The push-forward functor \(f'_*\) is a functor of
  restriction of scalars from \(\CC[\calZ']\) to \(\CC[\calY']\). Since the restriction
  turns \(\otimes_{\CC[\calZ]}\CC[\calZ']\) into \(\otimes_{\CC[\calZ]}\CC[\calY']\) the statement follows.

  In the case when \(f\) is a regular embedding we argue similarly to the projection formula proof, see proposition~\ref{prop:base-change}. Indeed, to construct \(f_*(\calF)\) we need to choose  a lift \(\tilde{\calF}\) of \(\calF\) to \(\calY\)
  that is compatible with surjective homomorphism \(\CC[\calY]\to \CC[\calZ]\).

  Let \(C_\bullet\) be a resolution of the structure sheaf of \(f(\calZ)\) in terms of
  locally-free \(\CC[\calY]\)-modules. To complete construction of \(f_*(\calF)\)
  we need to run the perturbation computation outlined in the proof of proposition~\ref{prop:push-f-star}. The starting point of the perturbation is
  \(\tilde{\calF}\otimes C_\bullet\).

  Finally, \(g^*\circ f_*(\calF))\) is defined as
  \(f_*(\calF)\otimes_{\CC[\calY]}\CC[\calY']\). A key observation is that flatness of \(g\) implies that if we run the perturbation procedure of \(\tilde{\calF}\otimes C_\bullet\otimes_{\CC[\calY]}\CC[\calY']\) then we get the matrix factorization \(g^*\circ f_*(\calF)\).

  On other hand \(g^{\prime*}(\calF)\) is defined as \(\calF\otimes_{\CC[\calZ]}\otimes \CC[\calZ']\). Thus to construct \(f'_*\circ g^{\prime*}(\calF)\) we need choose a
  lift \(\widehat{\calF}\) of \(\calF\otimes_{\CC[\calZ]}\otimes\CC[\calZ']\). So we pick
  \(\tilde{\calF}\otimes_{\CC[\calY]}\CC[\calY']\) as a lift.

  Thus to compute
  \(f'_*\circ g^{\prime *}(\calF)\) we need to perform the perturbation procedure from proposition~\ref{prop:push-f-star} with the starting object \(\tilde{\calF}\otimes_{\CC[\calY]}\CC[\calY']\otimes_{\CC[\calY']}C'_\bullet\). Here \(C'_\bullet\) is a resolution of the structure sheaf of \(f'(\calZ')\) in terms of \(\CC[\calY']\) locally-free sheaves.
  Finally, let us observe that the complexes \(\CC[\calY']\otimes_{\CC[\calY]}C_\bullet\)
  and \(C'_\bullet\) are homotopy equivalent. The uniqueness of the outcome of the perturbation procedure implies the statement.
  
\end{proof}

\subsection{Knorrer periodicity} The equivariant version of Knorrer periodicity \cite{Knorrer}, \cite{OblomkovRozansky16} is used in the arguments of the current paper.
In the equivariant setting the equivalence is studied in \cite{Hirano17a}.

We
include the statement for convince  of reader. Let \(\calX\) be a quasi-projective
manifold with \(G\)-action where \(G\) is reductive. Let \(E\) be a \(G\)-equivariant
vector bundle and \(s\in \Gamma(X,E^\vee)\) a regular section. If \(\chi\in \hat{G}\) is
a character of \(G\) then \(E(\chi)\) is the \(\chi\)-equivariant twist of \(E\). In this setting \(s\) defines a function \(Q_s\in \CC[E]_\chi^G\).

Let \(\calZ\subset E(\chi)\) be an zero locus of \(s\) and \(i: E(\chi)|_{\calZ}\to E(\chi)\) is the natural inclusion.
Respectively, \(p: E(\chi)|_{\calZ}\to \calZ\) is the natural projection.

\begin{proposition}\label{prop:Kno}\cite[Theorem 1.2]{Hirano17a}
  Let \(W:\calX\to \mathbb{A}^1\) be \(\chi\)-semi-invariant function, such that
  \(W|_{\calZ}\to \mathbb{A}^1\) is flat. Then there is an equivalence:
  \[\MF(\calZ,G,\chi,W|_\calZ)\to \MF(E(\chi),\chi,p^*W+Q_s).\]
\end{proposition}


\section{Chern functor}
\label{sec:chern-functor}

\subsection{Main categories: equivariant structures}
\label{sec:eq-strs}
In this subsection we introduce the main examples of the equivariant matrix factorizations that are used in our work. We concentrate on the details of
the equivariant structures and of the categories. Our categories have natural monoidal structure and it is defined in the section \ref{sec:properties}
where we discuss the properties of the Chern functor.

Denote $G = \GLn$, $B\subset G$ being its Borel subgroup (of upper-triangular matrices).
Also \(T\subset B\) is the diagonal torus \((\mathbb{C}^*)^n\) and \(B=T\ltimes U\).
Here \(U\) is the group of the unipotent upper-triangular matrices and \(\mathrm{Lie}(U)=\frn\subset \frg=\gln\).

Recall that the flag variety $\Fl$ is a quotient: $\Fl = G/B$ and similarly its cotangent bundle $\TsFl$ is a quotient: \[\TsFl = (G \times\frn)/B.\]
A bilinear pairing \((X,Y)\mapsto \Tr(X Y)\) provides a \(\frg\)-invariant isomorphism between \(\frg\) and \(\frg^*\).
By combining this isomorphism with the \(G\)-action moment map we obtain
the map $\TsFl\rightarrow \frg$, $\mu(g,Y) = \Adv{g} Y$, where $(g,Y)\in G\times\frn$.

\begin{defn}
The main (unframed) monoidal category in~\cite{OblomkovRozansky16} is the category of equivariant matrix factorizations
\begin{equation}
\label{eq:stfrmf1}
\MF
=\MF^{tame}(\scX,G\times B^{dg}\times B^{dg}\times \Tqt,q^0t,W),
\end{equation}
%
%
where
\[\scX:=\frg\times (G  \times \frn)\times (G \times\frn),\quad W(X,g_1,Y_1,g_2,Y_2)=\Tr(X(\Ad_{g_1}Y_1-\Ad_{g_2}Y_2))\]
and the action of an element $(h,\grbo,\grbt)\in \GBt$ on $\scX$ is
\begin{equation}\label{eq:actionBB}
(h;\grbo,\grbt)\cdot(X;g_1,Y_1;g_2,Y_2) = (\Adv{h}X;hg_1\grbo^{-1},\Adv{\grbo}Y_1;
hg_2\grbt^{-1},\Adv{\grbt}Y_2).
\end{equation}
%
%
\end{defn}

Since \(\scX/B^2=\frg\times \TsFl\times\TsFl\)
by proposition~\ref{prop:strict}  the category \(\MF\) is equivalent to
\(\MF\left(\frg\times  \TsFl\times\TsFl,\GL_n\times\Tqt,tq^0,\mu_1-\mu_2\right)\) from the introduction.

The elements of \(\MF\) are strongly  equivariant with respect to \(\mathbb{T}_{q,t}=\CC^*_q\times\CC^*_t\) action on the ambient space
\begin{equation}\label{eq:qt-act}
  (\lambda,\mu)\cdot(X;g_1,Y_1;g_2,Y_2)=(\lambda^2 X;g_1, \lambda^{-2}\mu^2Y_1; g_2,\lambda^{-2}\mu^2Y_2 ).\end{equation}
The action of \(\mathbb{T}_{q,t}\) translates into \(q,t\)-grading by the rule:
\[ \mbox{for a function } f \mbox{ such that } (\lambda,\mu)\cdot f= \lambda^{a}\mu^b f,\quad \deg(f)=q^at^b. \]

One can observe that the potential \(W\) has weight \(t^2\) with respect to the action of \(\mathbb{T}_{q,t}\). Thus we use the following definition of
\(\mathbb{T}_{q,t}\)-equivariance of a matrix factorization \((M,D,\partial)\):
\begin{equation}\label{eq:Tqt-equiv}
   \mbox{the differentials } D, \partial \mbox{ are of degree } q^0t. 
 \end{equation}

 \begin{remark}The standard homological algebra convention is that the differentials preserve the equivariant degree. It is different from our convention since
   we assume that the differentials raise the \(t\)-degree by \(1\): the usual convention can not be used in our case since we work with the complexes curved
   by the potential of \(t\)-degree \(2\).
In particular, since \(qt\) weights of functions are even, to satisfy the condition \eqref{eq:Tqt-equiv} we need to shift standard \(qt\)-grading by, possibly odd, \(t\)-degree.
\end{remark}

Furthermore, the elements of the category \(\MF\) are strongly equivariant with respect to the tori \(\mathbb{T}_{q,t}\), \(T^2\subset B^2\) and \(G\) and weakly equivariant with respect to
\(U^2\). That is for every \(\calF\in \MF \) the assumption
\begin{equation}\label{eq:assum}
  \calF \mbox{ is strongly } \mathbb{T}_{q,t}\times G\times T^2 \mbox{-equivariant and  is weakly } U^2\mbox{-equivariant}
  \end{equation}

  The first part of assumption the (\ref{eq:assum}) holds for every category of matrix factorizations that we use in this paper. The matrix factorizations used in
  our paper are on the spaces with the natural action of the group \(\mathbb{T}_{q,t}\times G\times B^l\) where \(l\) can be \(0,1,2\). We assume weak
  equivariance for the action of \(U^l\) and the strong equivariance of all reductive
  quotients.

\begin{defn}  The Drinfeld center category is defined as:
\[
\MF_{\dr}:=\MF(\scc,G\times \Tqt,tq^0,W_{\dr}),\quad
\scc=\frg\times G\times \frg,\quad W_{\dr}(Z,g,X)=\Tr(X(\Ad_gZ-Z)),
\]
the group $G$ acting on components of $\scc$ by conjugation:
\[h\cdot (Z,g,X)=(\Ad_h(Z),hgh^{-1},\Ad_h(X))\]
\end{defn}

  In particular, the first part of the assumption (\ref{eq:assum}) is true for \(\MF_{\dr}\).
The group \(B\) does not act on \(\scc\) thus the category consists of strongly equivariant matrix factorizations.  Also the action of \(\mathbb{T}_{q,t}\)
is given by
\begin{equation}\label{eq:qt-act-Z}
  (\lambda,\mu)\cdot (Z,g,X)=(\lambda^{-2}\mu^2Z,g,\lambda^2X), \quad \deg(\lambda)=q,\quad \deg(\mu)=t.
\end{equation}
We assume the grading convention \eqref{eq:Tqt-equiv} holds for the matrix factorizations from \(\MF_{\dr}\).


\subsection{Braid realization}
\label{sec:braid-realization}

In our previous papers \cite{OblomkovRozansky16}, \cite{OblomkovRozansky17} we worked with a bigger category, \(\MF\subset \MF_\sim\)
\[\MF_\sim:=\MF(\scX,G\times B^\sim\times B^\sim\times \Tqt,q^0t,W).\]

In particular, we constructed in \cite{OblomkovRozansky16} a monoidal functor, see definition~\ref{def:conv-aff} for the precise formulas for convolution:
\[\Br_n^{aff}\to \MF_\sim,\quad \beta\mapsto \calC_\beta\]

It could be seen from the constructions in the last mentioned paper that the elements
\(\calC_\beta\) are isomorphic to the objects in the smaller category

\begin{proposition}
  For any \(\beta\in \Br_n^{aff}\) the object \(\calC_\beta\) is isomorphic to an object in
  \(\MF\).
\end{proposition}
\begin{proof}
  In \cite{OblomkovRozansky16} we construct an element \(\calC_\beta\) as
  a  convolution \(\calC_{\epsilon_1}^{(i_1)}\star\calC_{\epsilon_2}^{(i_2)}\star\dots\star\calC_{\epsilon_l}^{(i_l)}\) where \(\beta=\sigma_{i_1}^{\epsilon_1}\cdot\sigma_{i_2}^{\epsilon_2}\cdot\dots\cdot\sigma_{i_l}^{\epsilon_l}\) is a presentation of \(\beta\) in terms of elementary generators of the braid group. 

  First, we observe that in the case \(n=2\), \(\calC_{\pm}^{(1)}\) are strongly equivariant matrix factorizations, see exact description in \cite[section 9.3]{OblomkovRozansky16}.  Thus, two-strand generators are tame.

  For a general \(n\) the generators \(\calC_{\pm}^{(i)}\) are obtained by application of
  the induction functors to \(\calC_{\pm}^{(1)}\), see \cite[section 6]{OblomkovRozansky16}. However, the induction functors are compositions of pull-backs along some smooth projections and push-forwards along some regular embeddings. Thus tameness is preserved and \(\calC_{\pm}^{(i)}\in \MF\).

  Finally, we notice the convolution operation, see \eqref{eq:conv-v}, is a composition  of some smooth pull-back and smooth push-forward. Thus the convolution preserves tameness and the proposition follows from proposition~\ref{prop:push-f-star}. 
\end{proof}

The stable version of the above category \(\MF_\sim\) was used in \cite{OblomkovRozansky16} to construct a HOMFLYPT
link homology.

\begin{defn} The stable version  of category \(\MF_\sim\) is defined by:
  \[\MF_\sim^{st}:=\MF((\scX\times V_G)^{st},G\times B^\sim\times B^\sim\times \Tqt,q^0t,W),\]
where \(V_G\) is \(\CC^n\) that is a tautological representation of \(G\) and
\(j_{st}:(\scX\times V_G)^{st}\to (\scX\times V_G)\) is defined by the condition \eqref{eq:stab-X}. 
\end{defn}

In particular, we construct \cite{OblomkovRozansky16} a monoidal functor \(\Br_n\to \MF_\sim^{st}\). The same argument as in the previous proposition shows the elements in the image of the previous
functor are isomorphic to some elements of the stable version of category \(\MF\):

\begin{defn}
  The stable version of category \(\MF\) is defined by:
\[\MF^{st}:=\MF^{tame}((\scX\times V_G)^{st},G\times B^{dg}\times B^{dg}\times \Tqt,q^0t,W),\]
where \(V_G\) as in the previous definition.
\end{defn}

The proposition~\ref{prop:strict} implies the equivalences of categories:
\[\MF\simeq \MF(\gl_n\times \mathrm{T}^*\Fl\times \mathrm{T}^*\Fl,\GL_n\times \Tqt,tq^0,W),\]
\[\MF^{st}\simeq \MF((\gl_n\times \CC^n\times \mathrm{T}^*\Fl\times \mathrm{T}^*\Fl)^{st},\GL_n\times \Tqt,tq^0,W).\]

Since the \(B^2\)-action on \((G\times \frn)^2\) is free we also have equivalences of categories:
\begin{equation}\label{eq:to-strict}
\MF\simeq \MF(\gl_n\times (G\times \frn)^2,\GL_n\times B^2\times \Tqt,tq^0,W),\end{equation}
\begin{equation}\label{eq:to-strict-st}
  \MF^{st}\simeq \MF((\gl_n\times \CC^n\times (G\times \frn)^2)^{st},\GL_n\times B^2\times\Tqt,tq^0,W).\end{equation}

It is shown in \cite{OblomkovRozansky17}, the pull-back \(j^*_{st}\) categorifies
the restriction homomorphism \(\Br_n^{aff}\to \Br_n\). Using proposition~\ref{prop:strict} we derive a version of main result of \cite{OblomkovRozansky17} that uses the global matrix  factorization studied previously in the literature (see section \ref{sec:MFs}):

\begin{corollary}
  There are monoidal functors \(\Phi^{aff}:\Br_n^{aff}\to\MF\), \(\Phi:\Br_n\to \MF^{st}\) such that
  \[j^*_{st}\circ \Phi^{aff}=\Phi\circ \mathrm{for},\]
  where \(\for: \Br_n^{aff}\to \Br_n\) is a forgetful functor.
\end{corollary}

Finally, let us state global matrix factorization version of the construction of the knot homology from \cite{OblomkovRozansky16}. Let \((\mathrm{T}^*)^{\oplus 2}\Fl\) be a double tangent bundle of \(\Fl\). Respectively, \(j_{e}: (\mathrm{T}^*)^{\oplus 2}\Fl\to \mathrm{T}^*\Fl\times \mathrm{T}^*\Fl\) is the diagonal embedding. For a given
\(\calC\in \MF^{st}\) we define:
\[\mathbb{L}(\calC)=j^*(\calC)\in \MF((\gl_n\times \CC^n\times (\mathrm{T}^*)^{\oplus 2}\Fl)^{st},\GL_n\times \Tqt,tq^0,\mu_1-\mu_2)=\MF_{cyc}\]

The potential in the definition of the category \(\MF_{cyc}\) is linear along the cotangent direction and we can apply the Kn\"orrer periodicity to this category and then we obtain the category \(\MF(\widetilde{\FHilb}_{y=0}^{\rfree},B\times \Tqt,tq^0,0)\) that is used in the section~\ref{sec:knot}.

\begin{corollary}
  For any \(\beta\in \Br_n\)  the hyper-cohomology
  \[\mathrm{H}^\bullet(\beta)=\mathbb{H}(\mathbb{L}(\Phi(\beta))\otimes \Lambda^*(V_G^\vee)^G)\]
  is a triply-graded vector space that categorifies the HOMFLYPT polynomial of the closure of \(\beta\).
\end{corollary}

\subsection{Construction of the Chern functor: motivation}
\label{sec:construction}

The discussion in this
section means to provide a geometric intuition behind our construction and the technical
details are postponed till the next subsection.
The Chern character functor \[\CH\colon \MF\rightarrow \MF_{\Dr}\] is a Fourier-Mukai transform going through the category of $(\GBd)$-equivariant matrix factorizations $\MFGBdT(\scZCHd; \WCHxd)$, where
\[
  \scZCHd = \frg\times (G\times\frnD) \times G,\qquad
  \WCHxd(X;g,Y;h) = \Tr (
  X (\Ad_{gh}Y - \Ad_{h}Y)),
\]
$\frnD=\frn$ is (identified with) the diagonal in the factor $\frn\times\frn$ of the space $\scX$, the group $\Bd=B$ is (identified with) the diagonal  in $B^2$ acting on that space and the action of $\GBd$ on $\scX_\Delta$ is
\[
  (h,\grb)\cdot(X;g_1,Y;g) = (\Adv{h}X;hg_1\grb^{-1},\Adv{\grb}Y; \Adv{h}g)
\]
%

We define two equivariant maps
\[
  \begin{tikzcd}
    & \scZCHd
    \ar[dl,"\hat{f}_{\Delta}"']
    \ar[dr,"\hat{\pi}_{\dr}"]
    \\
    \scX && \scc
  \end{tikzcd}
\]
explicitly as
\begin{equation}
  \label{eq:pdr}
  \hat{f}_{\Delta}(X; g,Y;h) = (X; gh,Y;h,Y),
  \qquad
  \hat{\pi}_{\dr}(X;g,Y;h) = (\Adv{h}Y,g,X).
\end{equation}

The maps are consistent with the potentials, that is \(\hat{f}_\Delta^*(W)=W_\Delta=\hat{\pi}^*_{\dr}(W_{\dr})\).
Both maps $\hat{f}_\Delta$ and $\hat{\pi}_{\dr}$ have a clear geometric meaning.
The map $\hat{f}_\Delta$ identifies $\scZCHd$ with a subvariety of $\scX$ corresponding to the diagonal $\frnD\subset\frn\times \frn$, which is invariant with respect to the action of the diagonal subgroup $\Bd\subset B^2$. As for $\hat{\pi}_{\dr}$, note that $\scZCHd/\Bd = \frg\times\TsFl\times G$ and the $\Bd$-equivariant map $\hat{\pi}_{\dr}$ descends on the quotient to $\frg\times\TsFl\times G\rightarrow \frg\times \frg\times G$ acting as the Springer resolution.

Now the Chern functor and its adjoint co-Chern functor are compositions of a pull-back and a push-forward:
\[
  \CH = \hat{\pi}_{\dr*}\circ\hat{f}_\Delta^*,\qquad\HC = \hat{f}_{\Delta*}\circ\hat{\pi}^*_{\dr}.
\]

From the computational perspective it is convenient to split $\hat{\pi}_{\dr}$ into a composition of two maps
\[
  \begin{tikzcd}
    \scZCHd
    \ar[r,hook,"\emi"']
    \ar[rr,bend left,"\hat{\pi}_{\dr}"]
    &
    \scZ_{\CH}\ar[r,"\pDr"']
    &
    \scc
  \end{tikzcd}
\]
where $\scZ_{\CH} = \frg\times G\times\frg\times G\times \frb$ and the embedding $\emi$ is generated by the natural inclusion $\frnD\hookrightarrow \frb$ combined with the adjoint action \(\emi(X,g,Y,h)=(\Ad_h Y,g,X,h,Y)\). The map \(\pi_{\dr}\)
is a projection on the first three factors.

Let us also point out that a formula for co-Chern functor in this section requires many clarifications and correction:
it is not immediately clear what is an adjoint functor to the restriction of \(B^2\)-equivariant structure to
the \(B_\Delta\)-equivariant structure. We spell out the omitted technical details in the next subsection where we also introduce maps \(j^0,f_\Delta\) in the
commuting diagram:
\begin{equation}\label{eq:dia-maps}
\begin{tikzcd}
    \scZCHd\ar[d,"\hat{f}_\Delta"]\ar[r,hook]\ar[rr,bend left, hook,"\emi"]&\scZ_{\CH}^0\ar[r,hook,"j^0"]\ar[ld,"f_\Delta"]&\scZ_{\CH}\ar[d,"\pi_{\dr}"]\\
    \scX& &\scc
  \end{tikzcd}
\end{equation}



\subsection{The Chern functor: details}
\label{sec:chern-funct-deta}

As we explained in the previous subsection, we need two auxiliary spaces in order to define the Chern functor: 
\[\scz_{\CH}=\frg\times G\times\frg\times G\times\frn,\quad \scZ_{\CH}=\frg\times G\times\frg\times G\times\frb\]
The action of \(G\times B\) on these spaces is
\[(k,b)\cdot (Z,g,X,h,Y)=
  (\Ad_{k}(Z),\Ad_k(g),\Ad_k(X),khb,\Ad_{b^{-1}}(Y)).\]
The spaces \(\scC\) and \(\scX\) are endowed with the standard \(G\times B^2\)-equivariant structure, the action of  \(B^2\) on \(\scC\) is trivial.
The following maps
\[\pi_{\Dr}:\mathscr{Z}_{\CH}\rightarrow \mathscr{C},\quad f_\Delta:\scZ_{\CH}^0\rightarrow \mathscr{X},\quad j^0:\scZ^0_{\CH}\to \scZ_{\CH}.\]
\[\pi_{\Dr}(Z,g,X,h,Y)=(Z,g,X),\quad f_\Delta(Z,g,X,h,Y)=(X,gh,Y,h,Y)\]
are  fully equivariant if we restrict the \(B^2\)-equivariant structure on
\(\scX\) to the \(B\)-equivariant structure via the diagonal embedding
\(\Delta:B\rightarrow B^2\). We use notation
\[\mathrm{Res}^{B^2}_{B}: \MF\to\MF(\gl_n\times (G\times \frn)^2,\GL_n\times B_\Delta\times \Tqt,tq^0,W),\] for the corresponding restriction functor.
Note that $j^0$ is the inclusion map.

\begin{remark}
  A restriction  functor between  categories of weakly equivariant matrix factorizations needs to be carefully defined since the correction differentials
  do not naturally restrict to correction differentials of a subgroup (see earlier version of this manuscript for the details of construction). On other hand
  a construction of the restriction
  functors for strictly equivariant matrix factorizations  verbatim repeats the corresponding construction for the complexes of equivariant sheaves.
  Thus we use equivalence \eqref{eq:to-strict} in our construction of \(\mathrm{Res}^{B^2}_B\).
\end{remark}

Using cyclic invariance of trace we derive:
\[\pi_{\Dr}^*(W_{\Dr})-f^*_\Delta(W)=\Tr((\Ad_g^{-1}X-X)(\Ad_hY-Z)).\]
Hence  by discussion in section~\ref{sec:KosMFs} there is a well-defined  kernel of the Fourier-Mukai transform which is a Koszul matrix factorization
\[\mathrm{K}_{\CH}:=[X-\Ad_{g^{-1}}X,\Ad_hY-Z]\in \MF(\scZ_{\CH},\GL_n\times B_\Delta\times \Tqt,tq^0,\pi_{\Dr}^*(W_{\Dr})-f^*_\Delta(W)).\]

\begin{remark} The subvariety \(\calX_\Delta\) can be embedded inside \(\mathscr{Z}_{\CH}\) by setting \(\Ad_hY=Z\), see diagram \refeq{eq:dia-maps}.
  Respectively, the restriction of \(f_\Delta\) on \(\calX_\Delta\) is \(\hat{f}_\Delta\).
\end{remark}
  
\begin{defn}
For \(\calC\in \MF\) we define value of  Chern functor \(\CH(\calC)\in \MF_{\dr}\) by:
\begin{equation}\label{eq:CH}
  \CH(\calC):=\pi_{\Dr*}(\CE_{\frn}(\mathrm{K}_{\CH}\otimes (j^0_*\circ f^*_\Delta(\mathrm{Res}^{B^2}_B(\calC))))^{T}).
\end{equation}
\end{defn}

\begin{remark}
  To lighten up notations we omit \(\mathrm{Res}^{B^2}_{B}\) in our computations and formulas.
\end{remark}

Here and everywhere below we use notation \((-)^T\) for \(T\)-invariants.
Since by our assumption \eqref{eq:assum} our matrix factorizations are graded
with respect to \(T\)-action, the functor \((-)^T\)  is just an extraction of the \(T\)-degree
\(0\) part of the matrix factorization. 

To define functor \(\HC\) we need to introduce a dual matrix factorization:
\[\mathrm{K}_{\HC}:=\mathrm{K}_{\CH}^\vee\in\MF(\scZ_{\CH},\GL_n\times B_\Delta\times \Tqt,tq^0,-\pi_{\Dr}^*(W_{\Dr})+f^*_\Delta(W)),\]
as well as the induction functor that is the left adjoint to the restriction functor \(\mathrm{Res}^{B^2}_B\). Unfortunately, there is no readily available
left adjoint functor. In the next few paragraphs we construct a substitute for such functor.






The product \(\scZ_{\CH}\times B\) has a \(B\times B\)-equivariant structure: for $(p,g)\in \scZ_{\CH}\times B $ we define
\[
  (h_1,h_2)\cdot (p, g) = (h_1\cdot p,h_1^{-1} g h_2)
\]
%
Then the following map is \(B^2\)-equivariant:
\[\tilde{f}_{\Delta}:\scZ_{\CH}^0\times B\rightarrow \scX\times B,\]
\begin{equation*}\tilde{f}_\Delta(Z,g,X,h,Y,b)=(X,gh,Y,hb,\Ad_bY,b).\end{equation*}
The map \(\tilde{f}_\Delta\) is a composition of the projection along the first factor of \(\scZ_{\CH}\) and
the embedding inside \(\scX\times B\). The embedding is defined by the formula
\[\Ad_bY_1=Y_2,\]
so it is a regular embedding. Moreover, we have
\[j^{0*}(\mathrm{K}_{\HC}\otimes \tilde{\pi}_{\Dr}^*(\calD))\in\MF_{G\times B^2}(\scZ_{\CH}\times B,\tilde{f}_\Delta^*(W)),\quad \mathrm{K}_{\HC}=\mathrm{K}_{\CH}^\vee,\]
where \(\tilde{\pi}_{\Dr}:\scZ_{\HC}\times B\rightarrow\scC\) is a natural extension of map
\(\pi_{\Dr}\) by the projection along \(B\). Thus
we have a well-defined matrix factorization \[\tilde{f}_{\Delta *}\circ j^{0*}(\mathrm{K}_{\HC}\otimes \pi_{\Dr}^*(\calD))\in \MF_{G\times B^2}(\scX\times B,\pi_B^*(W)),\] where \(\pi_B\) is the projection along the last factor.

\begin{defn}
For \(\calD\in \MF_{\dr}\) we define value of  co-Chern functor \(\HC(\calD)\in \MF\) by:
\begin{equation}\label{eq:HC}\HC(\calD):=\pi_{B*}(\tilde{f}_{\Delta*}\circ j^{0*}(\mathrm{K}_{\HC}\otimes \tilde{\pi}_{\Dr}^*(\calD))).\end{equation}
\end{defn}

The functors  used in the construction of the functors \(\HC\) and \(\CH\)
are \(G\)-equivariant. On the other hand by construction of the functor \(\mathcal{E}\mathrm{xt}\) \eqref{eq:ext-def} the two periodic complexes
\begin{equation}\label{eq:exts}
  \mathcal{E}\mathrm{xt}(\calD,\CH(\calC)),\quad \mathcal{E}\mathrm{xt}(\HC(\calD),\calC)\end{equation}
are complexes of \(\CC[\calC]^G\) and \(\CC[\scX]^{G\times B^2}\)-modules, respectively.
Both spaces \(\calC\) and \(\scX\) have \(\frg\) as their first factor. The corresponding projections on \(\frg\)  are  \(G\)-equivariant  and thus the
complexes \eqref{eq:exts} are naturally the complexes of \(\CC[\frg]^G\)-modules.
The complexes \eqref{eq:exts} have the double grading that comes from the \(\mathbb{T}_{q,t}\)-action \eqref{eq:Tqt-equiv}, \eqref{eq:qt-act}, \eqref{eq:qt-act-Z}.

\begin{proposition}\label{prop:adj} The functor \(\HC\) is left adjoint to  \(\CH\), that is, for
  any \(\calC\in\MF\) and \(\calD\in\MF_{\Dr}\) we have
  an isomorphism of doubly-graded two-periodic complexes
  of \(\CC[\frg]^{G}\)-modules:
  \[\mathcal{E}\mathrm{xt}(\calD,\CH(\calC))=\mathcal{E}\mathrm{xt}(\HC(\calD),\calC).\]

\end{proposition}
\begin{proof}
  By expanding the definition of \(\mathcal{E}\mathrm{xt}\) from \eqref{eq:ext-def}
  taking into account the assumption \eqref{eq:assum} we derive that
  the complexes from the main statement are:
  \[\calD^\vee\otimes \pi_{\Dr*}(\CE_{\frn}(\mathrm{K}_{\CH}\otimes j^{0}_*\circ f_{\Delta}^*(\calC))^T) \quad\text{and}\quad
    \CE_{\frn^2}(\pi_{B*}\circ\tilde{f}_{\Delta*}\circ j^{0*}
    (\mathrm{K}_{\HC}\otimes\tilde{\pi}_{\Dr}^*(\calD))^\vee\otimes \calC)^{T^2}.\]
  Thus we have to compare these two objects of \(\DperG(\frg)\).
  We  simplify the second expression.  Observe that the \ChE\ functor commutes with the
  push-forward \eqref{eq:pushf-ce}.

  Since the first factor of \(B\) in \(B^2\) acts freely on \(\scZ_{\CH}\times B\), the functor
  \(\CE_{\frn^2}(\dots)^{T^2}\) is equivalent to the composition of the restriction to the
  unity inside \(B\) and the functor \(\CE_{\frn}(\dots)^T\) with respect to the diagonal action of \(B\). In other words:
  \[\CE_{\frn^2}(\mathrm{K}_{\HC}\otimes\tilde{\pi}_{\Dr}^*(\calF))^{T^2}=\CE_{\frn}(\mathrm{K}_{\HC}\otimes\pi_{\Dr}^*(\calF))^T,\]
  for any \(\calF\in \MF_{\Dr}\).
  When restricted to the unit inside \(B\), the map \(\tilde{f}_{\Delta}\) becomes \(f_{\Delta}\). 
  So now we have to compare
  \[\calD^\vee\otimes \pi_{\Dr*}(\CE_{\frn}(\mathrm{K}_{\CH}\otimes j^{0}_*\circ f_{\Delta}^*(\calC))^T)\quad\text{and}\quad f_{\Delta*}(\CE_{\frn}(
    j^{0*} (\mathrm{K}_{\CH})\otimes j^{0*}\circ {\pi}_{\Dr}^*(\calD^\vee))\otimes \calC)^{T}),\]
  since \(\pi_{\Dr}^*(\calD)^\vee=\pi_{\Dr}^*(\calD^\vee)\) and
  \(\mathrm{K}_{\HC}^\vee=\mathrm{K}_{\CH}.\)
  The pull-back functor \(f^*_{\Delta}\) is adjoint to the corresponding push-forward functors, hence we need to compare homology of complexes
  \[\calD^\vee\otimes \pi_{\Dr*}(\CE_{\frn}(\mathrm{K}_{\HC}\otimes j^{0}_*\circ f_{\Delta}^*(\calC))^T)\quad\text{and}\quad \CE_{\frn}(
    j^{0*}(\mathrm{K}_{\CH})\otimes j^{0*}\circ {\pi}_{\Dr}^*(\calD^\vee)\otimes f^*_{\Delta}(\calC))^{T}.\]

  Since the map \(\pi_{\Dr}\) is \(\frn\)-equivariant and \(\calD^\vee\) has a trivial \(\frn\)-structure we need to match the homology of
  the complexes:
  \[\CE_{\frn}(\pi^*_{\Dr}(\calD^\vee)\otimes \mathrm{K}_{\HC}\otimes j^{0}_*\circ f_{\Delta}^*(\calC))^T)\quad\text{and}\quad \CE_{\frn}(
    j^{0*}(\mathrm{K}_{\CH})\otimes j^{0*}\circ {\pi}_{\Dr}^*(\calD^\vee)\otimes f^*_{\Delta}(\calC))^{T}.\]
  By the projection formula from proposition~\ref{prop:projection}  the first complex is
  \[\CE_{\frn}\bigg( j^0_*\bigg( j^{0*}\circ\pi^*_{\Dr}(\calD^\vee)\otimes j^{0*}(\mathrm{K}_{\HC})\otimes  f_{\Delta}^*(\calC)\bigg)\bigg)^T.\]

  As a final step   we use that \(j^{0}_*\) commutes with the \ChE functor \eqref{eq:pushf-ce}.

\end{proof}

\begin{remark}
  Thus \(\CH\) is a right adjoint of \(\HC\). On the other hand  \(\CH\) is not a left adjoint of \(\HC\) because the duality functor does not commute with the functor
  \(\pi_{\Dr*}(\CE_{\frn}(-))\).
\end{remark}
\subsection{Properties}
\label{sec:properties}

Recall that the category \(\MF\) has a convolution algebra structure defined with the auxiliary convolution space
\[\scX_{cnv}:=\frg\times (G\times\frn)\times (G\times\frn)\times (G \times\frn).\]
This space has a unique \(B^3\)-equivariant structure such that the  projection maps
\begin{equation}\label{eq:pi-ij}
  \pi_{ij}(X,g_1,Y_1,g_2,Y_2,g_3,Y_3)=(X,g_i,Y_i,g_j,Y_j),\end{equation}
commute with the action of each $B$.

\begin{defn}\label{def:conv-aff}
Given two elements \(\calF,\calG\in \MF\)  we define their convolution by 
\begin{equation}\label{eq:conv}\calF\star\calG:=\pi_{13*}(\CE_{\frn^{(2)}}(\pi_{12}^*(\calF)\otimes\pi_{23}^*(\calG))^{T^{(2)}}),
\end{equation}
\end{defn}

In the definition  and everywhere below  \(\frn^{(2)}\), \(T^{(2)}\) stand for the Lie algebra of the
unipotent radical and the reductive quotient of the second factor in \(B^3\).






Similarly, we define a monoidal structure on \(\MF_\dr\):
\begin{defn}\label{def:unfr-conv} Given \(\calF,\calG\in \MF_\dr\) we define their convolution by:
\begin{equation}\label{eq:convDr}
  \calF\star\calG:=\pi_{3*}(\pi_1^*(\calF)\otimes\pi_2^*(\calG))
\end{equation}
where the maps \(\pi_i\) are the maps \(\frg\times G^2\times \frg\rightarrow\scC\) are
\[\pi_1(Z,g_1,g_2,X)=(Z,g_1,X),\quad \pi_2(Z,g_1,g_2,X)=(\Ad_{g_1}Z,g_2,X),\]
\[\pi_3(Z,g_1,g_2,X)=(Z,g_2g_1,X).\]
\end{defn}

The functor \(\HC\) respects the convolution product.
\begin{proposition}\label{prop:mono}
  The functor \(\HC\) is monoidal.
\end{proposition}

Consider the embedding of the nilpotent cone $j^0\colon\caln\rightarrow\frg$.
Let us define
\[\MF_{\Dr}^0:=\MF(\scC^0,G\times B^2\times \Tqt,q^0t,W_{\dr}), \quad \scC^0=\frg\times G\times\caln,\]
where \(W_\dr(Z,g,X)=\Tr(X(\Ad_gZ-Z)).\)

By restricting the maps $\pi_i$ in~\eqref{eq:convDr} to the nilpotent locus we obtain a definition of the
monoidal structure on the category $\MF_{\dr}^0$ such that the pull-back functor:
$$j^{0*}:\MF_{\dr}\to\MF_{\dr}^0
$$
is a monoidal functor.
Define an analog of the functor $\HC$ for the nilpotent version of our category:
$$\HC^0:\MF_{\dr}^0\to \MF,\quad \HC^0(\calD):=\pi_{B*}(\tilde{f}_{\Delta*}(\mathrm{K}_{\HC}\otimes \pi_{\Dr}^*(\calD))).$$
Since the maps $j^0$ and $\pi_{\dr}$ commute,
the previous functor is the composition of the new one with the pull-back:
$$\HC=\HC^0\circ j^{0*}.$$

The nilpotent cone \(\mathcal{N}\) is singular and we have to exercise some care when we work with the category \(\MF_{\dr}^0\). Luckily, in the proof below we apply pull-back
functors along smooth or regular map to the elements of \(\MF_{\dr}^0\). 

\begin{proof}[Proof of proposition~\ref{prop:mono}]
  By the previous remark, it is enough to show that the functor $\HC^0$ is monoidal. To simplify notations we use \(\mathrm{K}\) for the Koszul matrix factorization
  \(\mathrm{K}_{\HC}\).
  Our proof relies on the base change along the rectangle of maps:
  \[
    \begin{tikzcd}
      \scX& \scX_{cnv}\arrow[l,"\pi_{13}"]\arrow[r,"i_{cnv}"]&\scX\times\scX\\
      \scZ^0_{\CH}\times B\arrow[u,dotted,"\tilde{f}_{\Delta}\circ\pi_{B}"]\arrow[d,dotted,"\pi_{\Dr}\times 1"]& \mathscr{Y}\arrow[r,dotted,"\hat{i}_{cnv}"]\arrow[u,dotted,"\hat{f}"]\arrow[l,dotted,"\pi_Y"]&\scZ^0_{\CH}\times\scZ^0_{\CH}\times B\arrow[u,"\tilde{f}_{\Delta}\times\tilde{f}_{\Delta}\circ \pi_B"]\arrow[d,"\pi_{\Dr}\times\pi_{\Dr}\times \pi_B"]\\
      \mathscr{C}^0\times  B&\scC^0_{cnv}\times G\times B\arrow[l,dotted,"\pi_3\times \pi_G"]\arrow[r,dotted,"j_{cnv}"]\arrow[u,"i_C",dotted]&\scC^0\times\scC^0
    \end{tikzcd}
  \]
  where
  \[\mathscr{Y}=\frg^3\times G^3\times\frn\times B\] and the dotted maps will be explained below.
  First of all, we explain why pushing along the solid arrows results in \(\HC^0(\calD')\star\HC^0(\calD'')\). Indeed,  expand
  the expression for the convolution:
  \[\pi_{13*}(\CE_{\frn^{(2)}}(i^*_{cnv}(\tilde{f}_{\Delta*}\circ\pi_{B*}(\mathrm{K}\otimes \pi^*_{\Dr}(\calD'))\boxtimes \tilde{f}_{\Delta*}\circ\pi_{B*}(\mathrm{K}\otimes \pi^*_{\Dr}(\calD'')) ))^{T^{(2)}}),\]
  where \(i_{cnv}\colon \scX_{cnv}\rightarrow \scX\times\scX\) is the natural inclusion:
  \[i_{cnv}(X,g_1,Y_1,g_2,Y_2,g_3,Y_3)=(X,g_1,Y_1,g_2,Y_2)\times (X,g_2,Y_2,g_3,Y_3).\]
  Next we notice that \(\frn^{(2)}\) and \(T^{(2)}\) act freely on \(\scZ^0_{\CH}\times B\times \scZ^0_{\CH}\times B\) and since
  this space is a domain for \(\mathrm{K}\otimes\pi_{\Dr}^*(\calD')\boxtimes\mathrm{K}\otimes \pi_{\Dr}^*(\calD'')\),
  we replace the functor \(\CE_{\frn^{(2)}}(\dots)^{T^{(2)}}\) with the restriction functor \((\dots)_{b'=1}\) where
  \(b'\) is an element of the first copy of \(B\) in \(\scZ^0_{\CH}\times B\times \scZ^0_{\CH}\times B\). Thus the convolution is given by the pull-backs push-forwards
  along the solid arrows of the above diagram:

  \[\HC^0(\calD')\star\HC^0(\calD'')=\pi_{13*}(i^*_{cnv}(\tilde{f}_{\Delta*}(\mathrm{K}\otimes\pi_{\Dr}^*(\calD'))\boxtimes
    \tilde{f}_{\Delta*}\circ\pi_{B*}(\mathrm{K}\otimes \pi^*_{\Dr}(\calD''))))).\]

  Define the maps \(\hat{i}_{cnv}\) and \(\hat{f}\):
  \begin{equation}\label{eq:i-cnv}
  \hat{i}_{cnv}(Z_1,Z_2,X,g_1,g_2,g_3,Y,b)=(Z_1,g_1g_2^{-1},X,g_2,Y,Z_2,g_2bg_3^{-1},X,g_3b^{-1},Y),\end{equation}
  \[\hat{f}(Z_1,Z_2,X,g_1,g_2,g_3,Y,b)=(X,g_1,Y,g_2,Y,g_3,\Ad_b Y).\]
  Thus using the base change in the upper left corner of the diagram we obtain
  \[\HC^0(\calD')\star\HC^0(\calD'')=\pi_{13*}\circ\hat{f}_{*}\circ \hat{i}^*_{cnv}(\mathrm{K}_1\boxtimes \mathrm{K}_2\otimes \pi_{\Dr}^*\times\pi_{\Dr}^*\times\pi_B^*(\calD'\boxtimes\calD'')),\]
  where \(\mathrm{K}_i\) is the pull-back of the kernel \(\mathrm{K}\) along the projection on the \(i\)-th copy of
  \(\scZ^0\) in the product \(\scZ^0\times \scZ^0\times B\).

  Next define the maps \(i_C\) and \(j_{cnv}\):
  \[i_C(Z,X,g',g'',Y,b,h)=(Z,\Ad_{g'g''}Z,X,g'h,h,(g'')^{-1}hb,Y,b),\]
  \[j_{cnv}(Z,X,g',g'',Y,b,h)=(X,Z,g',X,\Ad_{g'g''}Z,g'').\]
  Now   because of the explicit formula for the  Koszul matrix factorization \(\mathrm{K}_1\)  and
  the construction of the push-forward \cite{OblomkovRozansky16} we conclude that
  \[\hat{i}^*_{cnv}(\mathrm{K}_1\boxtimes \mathrm{K}_2\otimes \pi_{\Dr}^*\times\pi_{\Dr}^*\times\pi_B^*(\calD'\boxtimes\calD''))=
    i_{C*}(\mathrm{K}\otimes j^*_{cnv}(\calD'\boxtimes \calD'')),\]
  where \(\mathrm{K}=i_{C}^*(\mathrm{K}_2)\). In more details, let use coordinates on \(\mathscr{Y}\)  as in \eqref{eq:i-cnv} to describe
  the tensor product \(\mathrm{K}_1\boxtimes \mathrm{K}_2\):
 \begin{align*}
[\Ad^{-1}_{g_1g_2^{-1}}X-X,\Ad_{g_2}Y-Z_1]&\otimes[\Ad^{-1}_{g_2bg_3^{-1}}X-X,\Ad_{g_3b^{-1}}Y-Z_2]=\\
   [\Ad^{-1}_{g_1}X-\Ad^{-1}_{g_2}X,Y-\Ad_{g_2}^{-1}Z_1]&\otimes [\Ad_{g_2}^{-1}X-\Ad_{g_3b^{-1}}^{-1}X,Y-\Ad_{g_3b^{-1}}Z_2]=\\
[\Ad_{g_1}^{-1}X-\Ad^{-1}_{g_2}X,\Ad_{g_3b^{-1}}Z_2-\Ad_{g_2}^{-1}Z_1]&\otimes [\Ad_{g_3b^{-1}}^{-1}X-\Ad_{g_1}^{-1}X,Y-\Ad_{g_3b^{-1}}Z_2]
\end{align*} 

The first equality is given by a change of basis in each Koszul matrix  factorization.
The second follows from the row operation transformation of the matrix factorizations.

Thus the second term in the last tensor product is \(\mathrm{K}\). The first factor is the matrix factorization that defines push-forward \(i_{C*}\) since the image of the
embedding \(i_{C}\) is defined by \(\Ad_{g_3b^{-1}}Z_2-\Ad_{g_2}^{-1}Z_1\), see iterative construction of push-forward in the proof of proposition~\ref{prop:push-f-star}.

  There is a unique map \(\pi_Y\) that makes our diagram commute. The commutativity of
  the diagram implies  the formula
  \[
    \HC^0(\calD')\star\HC^0(\calD'') =
    \tilde{f}_{\Delta*}\circ \pi_{B*}\circ \pi_{Y*}\circ i_{C*}(\mathrm{K}\otimes j^*_{cnv}(\calD'\boxtimes \calD'')).\]
  Applying the base change to the lower-left corner in order to obtain another formula
  \[
    \HC^0(\calD')\star\HC^0(\calD'') =
    f_{\Delta*}\circ \pi_{B*}(\mathrm{K}\otimes \pi^*_{\Dr}\times 1\circ \pi_{3*}\times \pi_{G*}\circ j_{cnv}^*(\calD'\boxtimes \calD'')),\]
  where \(\mathrm{K}\) is \(\mathrm{K}_{\HC}\). To complete proof we observe that
  \[\pi_{3*}\times \pi_{G*}\circ j_{cnv}^*(\calD'\boxtimes \calD'')=\pi_{B}^*(\calD'\star\calD'').\]
\end{proof}

\begin{remark}\label{rem:EG}
   Techniques of the previous proof can be used  to show the
   "projection" formula, as it was suggested to the authors by the anonymous referee:
   \[\CH(\calC\star \HC(\calD))=\CH(\calC)\star \calD,\]
   where \(\calD\in \MF_{\Dr}\) and \(\calC\in \MF\). A proof will appear in a subsequent publication.
\end{remark}

Let us  recall from \cite{OblomkovRozansky16} that the unit in \((\MF,\star)\) is defined by
\begin{equation}\label{eq:unit-MF}
  \calC_{\parallel}=j_{\parallel*}(\calO),\end{equation}
where
\(j_{\parallel}:\frg\times G\times\scX(B) \rightarrow\frg\times G\times\scX\), \(\scX(B)=\frg\times B\times G\times \frn\)
is defined by
\[j_{\parallel}(Z,g,X,b,k,Y)=(Z,g,X,k,Y,kb,\Ad_bY).\]

\begin{proposition} \label{prop:central}For any \(\calD\in \MF_{\dr}\) and a pair of Koszul matrix factorizations  \(\calc^+,\calc^-\in \MF\) such that \(\calc^+\star\calc^-\sim \calc_\parallel \),
  there is a homotopy equivalence:
  \begin{equation}
    \label{eq:dblpr}
    \calC^+\star \HC(\calD)\star \calc^-\sim\HC(\calD).
  \end{equation}
\end{proposition}
\begin{proof}
  The space \(\hat{\scX}=\frg\times (G\times \frn)^4\times B\) has natural \(B^4\)-equivariant projections
  \(\hat{\pi}_{ij}:\hat{\scX}\to \scX \) as well as an embedding \[\tilde{f}_{\Delta}\colon\scZ_{\CH}\times (G\times \frn)^2\times B\to
    \hat{\scX}\]
  defined as
  \[\tilde{f}_\Delta(Z,g,X,h,Y,g_1,Y_1,g_4,Y_4,b)=(X,g_1,Y_1,gh,Y,hb,\Ad_bY,g_4,Y_4,b).\]
  The double product in the statement is equal to \(\hat{\pi}_{14*}(\CE_{\frn^{(2)}\times\frn^{(3)}}(\calC')^{T^{(2)}\times T^{(3)}})\) where \(\calC'\) is the complex on
  \(\hat{\scX}\):
  \[\calc':=\hat{\pi}_{12}^*(\calc^+)\otimes \hat{\pi}_{34}^*(\calc^-)\otimes \tilde{f}_{\Delta *}\bigl(j^{0*}\circ\pi_{\dr}^*(\calD)\otimes \mathrm{K}_{\CH}\bigr). \]
  Since the push-forward \(\tilde{f}_{\Delta*}\) is adjoint to the pull-back \(\tilde{f}_{\Delta}^*\) with respect to the
  pairing \(\CE_{\frn^{(2)}\times\frn^{(3)}}(\cdot\otimes\cdot)^{T^{(2)}\times T^{(3)}}\), the \lhs of~\eqref{eq:dblpr} has an expression
  \begin{equation}\label{eq:pi14}
    \calC^+\star \HC(\calD)\star \calc^- =
    \pi_{14*}(\CE_{\frn^{(2)}\times\frn^{(3)}}(\calc'')^{T^{(2)}\times T^{(3)}}),
  \end{equation}
  where \(\calc''\) is the matrix factorization
  on \(\scz\times (G\times \frn)^2\times B\):
  \[\calc'':=\pi^*_{1\bullet}(\calc^+)\otimes \pi^*_{\bullet 4}(\calc^-)\otimes j^{0*}\circ\pi^*_{\dr}(\calD)\otimes \mathrm{K}_\CH,\]
  and the maps in the last formula are:
  \[\pi_{1\bullet}(Z,g,X,h,Y_\bullet,g_1,Y_1,g_4,Y_4,b)=(X,g_1,Y_1,gh,Y_\bullet),\]
  \[\pi_{\bullet 4}(Z,g,X,h,Y_\bullet,g_1,Y_1,g_4,Y_4,b)=(X,hb,\Ad_bY_\bullet,g_4,Y_4),\]
  \[\pi_{14}(Z,g,X,h,Y_\bullet,g_1,Y_1,g_4,Y_4,b)=(X,g_1,Y_1,g_4,Y_4).\]
  These maps satisfy relations: \(\hat{\pi}_{14}\circ\tilde{f}_\Delta =\pi_{14}\), \(\hat{\pi}_{12}\circ\tilde{f}_\Delta=\pi_{1\bullet}\),
  \( \hat{\pi}_{34}\circ\tilde{f}_\Delta=\pi_{\bullet 4}\).

  Just as in the proof of proposition~\ref{prop:adj} we observe that the left \(B\)-action on \(B\) is free, hence the functor
  \(\CE_{\frn^{(2)}\times\frn^{(3)}}(\dots)^{T^{(2)}\times T^{(3)}}\) is equivalent to the composition of the restriction to the unity in \(B\) and \(\CE_{\frn}(\dots)^T\) with respect to the diagonal action of \(B\). In other words,~\eqref{eq:pi14} is equal to
  \( \pi_{14*}(\CE_{\frn}(\calc''|_{b=1})^{T})\), here and everywhere below we use the same notations \(\pi_{cd}\) for the restriction of maps
  \(\pi_{cd}\) on the sub locus \(b=1\).

  Since \(\mathrm{K}_{\CH}=[X-\Ad^{-1}_gX, \Ad_hY-Z]\), we can use row transformations of Koszul matrix factorizations in order to replace \(X\) with \(\Ad_gX\) in \(\pi^*_{1\bullet}(\calc^+)\). Now we combine this computation with the fact that the matrix
  factorization  \(\pi^*_{1\bullet}(\calc^+)\) is strongly
  \(G\)-equivariant,
  to establish the isomorphism
  \[
    \calc''|_{b=1} \cong
    \calc''':=\pi'^{*}_{1\bullet}(\calc^+)\otimes \pi^*_{\bullet 4}(\calc^-)\otimes j^{0*}\circ\pi^*_{\dr}(\calD)\otimes \mathrm{K}'_\CH,\]
  where \(\mathrm{K}'_{\CH}=[X-\Ad^{-1}_gX,*]\) and the map \(\pi'_{1\bullet}\) is
  \[\pi'_{1\bullet}(Z,g,X,h,Y_\bullet,g_1,Y_1,g_4,Y_4,b)=(X,g^{-1}g_1,Y_1,h,Y_\bullet).\]

  To compute \(*\) in \(\mathrm{K}'_{\CH}\), observe that this matrix factorization has the potential:
  \begin{multline*}
    \Tr(X(\Ad_{g^{-1}g_1}Y_1-\Ad_hY_\bullet))+\Tr(X(\Ad_hY_\bullet-\Ad_{g_4}Y_4))+\Tr(X(Z-\Ad_gZ))\\-\Tr(X(\Ad_{g_1}Y_1-\Ad_{g_4}Y_4))=
    \Tr((\Ad_gX-X)(\Ad_{g_1}Y_1-Z))\\=\Tr((X-\Ad_{g}^{-1}X)(\Ad_{g^{-1}g_1}Y_1-\Ad^{-1}_gZ))
  \end{multline*}
  Thus we conclude that \(*=-\Ad_{g^{-1}g_1}Y_1+\Ad_{g}^{-1}Z\) and in \(\calc'''\) only the first two factors depend on the variable \(h,Y_\bullet\) and have
  a non-trivial $B$-action. Let \(\pi\) be the projection along
  the coordinates \(h,Y_\bullet\). Since \(\pi_{*}(\CE_{\frn}(\pi'^{*}_{1\bullet}(\calc^+)\otimes \pi^*_{\bullet 4}(\calc^-))^T)=\calc^+\star\calc^-\sim
  \calc_{\parallel}\)
  we have the homotopy between the convolution in~\eqref{eq:dblpr}
  and \(\tilde{\pi}_{14*}(\tilde{\calc})\), where
  \(\tilde{\calc}\) is the matrix factorization on the space \(\frg\times G\times \scX\):
  \[\tilde{\calc}:=\calc_{\parallel}\otimes  j^{0*}\circ\pi^*_{\dr}(\calD)\otimes \mathrm{K}'_\CH, \]
  and \(\tilde{\pi}_{14}\) is determined by \(\pi_{14}=\tilde{\pi}_{14}\circ \pi.\)


  Now let us recall the formula \eqref{eq:unit-MF} for \(\calC_\parallel\).  Thus the convolution in~\eqref{eq:dblpr}
  is equal to \(\pi_{B*}(j_{\parallel*}(j^{0*}\circ\pi^*_{\dr}(\calD)\otimes j_{\parallel}^*(\mathrm{K}'_{\CH})))\), where \(\pi_B\) is the projection along \(B\).
  By identifying \(\frg\times G\times \scX(B)\) with \(\scZ^0_{\HC}\times B\) we match \(j_{\parallel}\) with \(\tilde{f}_\Delta\).
  Since \(j^*_{\parallel}(\mathrm{K}'_{\CH})=\mathrm{K}_{\CH}\) the statement follows.
\end{proof}

In \cite{OblomkovRozansky17} we constructed an monoidal functor from the affine braid group \[\Phi^{\aff}\colon\Br_{\aff}\to \MF.\]
Under this homomorphism the generators of the group become Koszul matrix factorizations, thus the previous proposition
implies

\begin{corollary} For any \(\beta\in \Br_{\aff}\) and \(\calD\in \MF_{\dr}\) we have:
  \[\Phi^{\aff}(\beta)\star \HC(\calD)\sim \HC(\calD)\star \Phi^{\aff}(\beta).\]
\end{corollary}
\begin{proof}
  Let us present \(\beta\) as product of elementary braids \(\beta=\sigma_{i_1}^{\epsilon_1}\cdot\sigma_{i_\ell}^{\epsilon_\ell}\)
  The statement is equivalent to
  \[\Phi^{\aff}(\sigma_{i_1}^{\epsilon_1})\star\dots\star\Phi^{\aff}(\sigma_{i_\ell}^{\epsilon_\ell})\star \HC(\calD)\star \Phi^{\aff}(\sigma_{i_\ell}^{-\epsilon_1})\star\dots\star\Phi^{\aff}(\sigma_{i_1}^{\epsilon_1})\sim \calC_\parallel.\]
  Since the matrix factorization \(\Phi^{\aff}(\sigma_i^\epsilon)\) is a Koszul matrix
  factorization \cite{OblomkovRozansky16}, we apply the previous proposition \(\ell\)
  times to prove the statement.
\end{proof}

\begin{remark}
  The statement of proposition~\ref{prop:central} holds without an assumption that \(\calC^+,\calC^-\) are Koszul matrix factorizations. To show that we need to use
  systematically theorem 2.2 from \cite{khovanovRozansky07c}.
\end{remark}

Since \(W_{\dr}(Z,g,X)=\Tr((g-1)[Zg^{-1},X])\), we have a Koszul matrix factorization
\begin{equation}\label{eq:unit}
  \calO_{\dr}=[g-1,[X,R]]\in\MF_{\Dr},\quad R=Zg^{-1}.
\end{equation}
It is a unit (see proposition~\ref{prop:unit}) in the monoidal category \(\MF_{\Dr}\).
The proof of theorem \ref{thm:main} requires a computation of the co-Chern of the matrix factorization \(\calO_{\dr}\).

\begin{proposition} \label{prop:unit} The element \(\calO_{\dr}\in \MF_{\Dr}\) is the convolution unit and
  \[\HC(\calO_{\dr})=\calC_{\parallel}.\]
\end{proposition}
\begin{proof}
  Recall that \(\calC_{\parallel}=j_{\parallel*}(\calO)\), where \(j_{\parallel}\colon \scX(B):=\frg\times B\times G\times\frn\rightarrow\scX\)
  is defined by
  \[j_{\parallel}(X,b,h,Y)=(X,h,Y,hb,\Ad_bY).\]
  There is a unique \(B^2\)-equivariant structure on the space \(\scX(B)\) that makes the map
  \(j_\parallel\) equivariant.
  On the other hand, \(\scX(B)\) embeds naturally into \(\scZ_{\HC}\times B\):
  \begin{equation}
    i_{\CH}\colon \scX(B)\to \scZ_{\CH}\times B,\quad   i_{\CH}(X,b,h,Y)=(\Ad_{h}Y, 1,X,h,Y,b).
  \end{equation}
  Moreover, \(\tilde{f}_\Delta\circ i_{\CH}=
  j_\parallel\) and the subvariety \(i_{\CH}(\scX(B))\)
  is defined by  equations
  \begin{equation}\label{eq:XB}
    Z=\Ad_{h}Y,\quad g=1.
  \end{equation}

  To complete the proof, observe that the  Koszul matrix factorization  \(\pi_{\Dr}^*(\calO_{\dr})\otimes \mathrm{K}_{\CH}\) is the
  Koszul matrix factorization  for the ideal with the generators \eqref{eq:XB}. The equations in \eqref{eq:XB} form a regular sequence, hence
   by lemma 3.6 in \cite{OblomkovRozansky16} (also see proposition~\ref{prop:push-f-star}) such Koszul matrix factorization is unique  and it coincides with the push-forward
  \(i_{\CH *}(\calO)\) by  definition.


\end{proof}

\section{Coherent sheaves and Drinfeld center}
\label{sec:coh-shv}


In this section we introduce the framed and stable enhancements of the matrix factorizations from \(\MF\) and \(\MF_{\Dr}\):
\[\MF^{\rfr},\quad\MF^{\rfs},\quad \MF^{\rst},\quad\MF^{\rfr}_{\Dr},
  \quad\MF^{\rfs}_{\Dr}.\]

The stable category \(\MF^{\rst}\) appeared in~\cite{OblomkovRozansky16}, its convolution structure was used to represent the braid group \(\Br_n \) and  construct the link homology. The homology in \cite{OblomkovRozansky16} are triply graded and categorify
the HOMFLYPT polynomial. Recently \cite{OblomkovRozansky20}, the authors have shown that the homology coincide with
the Khovanov-Rozansky  homology \cite{KhovanovRozansky08a}, as it was conjectured in
\cite{GorskyNegutRasmussen16}.
We recall this construction from \cite{OblomkovRozansky16} in section~\ref{sec:knot}..

The framed stable categories \(\MF^{\rfr}\), \(\MF^{\rfr}_{\Dr}\) appear naturally in the context of study of quiver varieties, framing is a standard
feature of Nakajima's quiver varieties. The framed stable category \(\MF^{\rfs}\) plays somewhat technical role in our constructions,
this category is essentially equivalent to the previously defined category \(\MF^{\rst}\), see remark~\ref{rem:singular}. The
framed stable category \(\MF_{\Dr}^{\rfs}\) on other hand is equivalent to the category of two-periodic complexes of coherent sheaves
on \(\Hilb_n(\CC^2)\). Thus this stable framed category is essential for main result of the paper.



In more details, category  \(\MF^{\rfs}_{\Dr}\) is related to the dg category \(\Dper (\Hilb_n(\CC^2))\) by the localization functor defined later:
\[ \Loc^{\rfs}\colon \Dper(\Hilb_n(\CC^2))\rightarrow \MF^{\rfs}_{\Dr},\]
where \(\Dper(\Hilb_n(\CC^2))=\MF(\Hilb_n(\CC^2, \Tqt,tq^0))\) is the derived category of two-periodic complexes of \(\Tqt\)-equivariant coherent sheaves.

Moreover, we  show that the functor \(\Loc^{\rfs}\) is an equivalence. Thus we have a well-defined functor
\[\CH^{\rfs}_{\rloc}:=\left(\Loc^{\rfs}\right)^{-1}\circ \CH^{fs}.\]

For various framing construction we use the free rank \(n\)  bundles with various \(G\times B^2\) equivariant structures. Let us fix notations
\[\CC^n=V_{H}\ni v,\quad h\cdot v=hv,\]
\[\CC^n=V_{H}^*\ni v,\quad h\cdot v=vh^{-1},\]
where \(H\) can be \(G, B^{(1)}, B^{(2)}, B^{(3)}\) which are the factors of \(G\times B^3\). Respectively, the other factors of \(G\times B^3\) act trivially on bundle.

We realize the space \(V_H\) as  vector space of columns \(\Hom(\CC,\CC^n)\) and
\(V^*_H\) as vector space of rows \(\Hom(\CC^n,\CC)\). Thus the products
\(v w\), \(w v\), \(v\in V_H\), \(w\in V^*_H\) are well-defined.

\subsection{Framed categories}
\label{sec:framed-categories}

Define framed versions of the spaces \(\scX\) and \(\scC\)
\[\scX^{\rfr}:=\scX\times V^*_G\times V_{B^{(1)}}\times V_{B^{(2)}},\quad \scC^{\rfr}:=\scC\times V_G^*\times V_G,\]
and their potentials
\[W^{\rfr}(X,g_1,Y_1,g_2,Y_2,w,v_1,v_2):=W(X,g_1,Y_1,g_2,Y_2)+\Tr(w(g_1v_1-g_2v_2)),\]
\[W^{\rfr}_{\Dr}(X,g,Z,w,v)=W_{\Dr}(X,g,Z)+\Tr(w(v-gv)).\]

\begin{defn}\label{def:fr-cats}
We define the framed categories  as
\[ \MF^{\rfr}:=\MF(\scX^{\rfr},G\times B^2\times\mathbb{T}_{q,t},q^0t,W^{\rfr}),\quad
  \MF^{\rfr}_{\Dr}:=\MF(\calC^{\rfr},G\times\mathbb{T}_{q,t},q^0t,W^{\rfr}_{\dr}).
\]

The framed categories have a natural monoidal structure similar to the unframed case.
The convolution spaces are
\[\scX_{\rcnv}^{\rfr}:=\scX_{\rcnv}\times V^*_G\times V_{B^{(1)}}\times V_{B^{(2)}}\times V_{B^{(3)}},\quad \scC_{\rcnv}:=\scC_{\rcnv}\times V^*_G\times V_G\]
and the extended maps are
\[\pi_{ij}(X,g_1,Y_1,g_2,Y_2,g_3,Y_3,w,v_1,v_2,v_3)=(X,g_i,Y_i,g_j,Y_j,w,v_i,v_j),\]
\[\pi_1(Z,g_1,g_2,X,w,v)=(Z,g_1,X,w,v),\quad \pi_2(Z,g_1,g_2,X,w,v)=(Z,g_2,\Ad_{g_1}X,w,g_1v),\]
\[ \pi_3(Z,X,g_1,g_2,w,v)=(Z,g_1g_2,X,w,v).\]

The convolution product is defined by the formulas \eqref{eq:conv} and \eqref{eq:convDr}.
\end{defn}

\begin{defn}
Let us  define
\[\MF^v:=\MF(\scX^v,G\times B^2\times\mathbb{T}_{q,t},q^0t,W),\quad \scX^v=\scX\times V_G.\]
here we use the natural projection \(\scX^v\to\scX\) to define \(W \) on \(\scX^v\).
\end{defn}

This category plays an auxiliary role but it naturally related to the stable
categories discussed in the next section.

\begin{proposition}
  The category \(\MF^{\rfr}\) is equivalent to the category \(\MF^v\).
\end{proposition}
\begin{proof}
  The equivalence \(i_{\rfr}\colon\MF^v\to\MF^{\rfr}\)
is defined as
\[i_{\rfr}(\calF):=\pi^*_{V^*V}(\calF)\otimes \mathrm{K}^{\rfr},\quad  \pi_{V^*V}: \scX^{\rfr}\to \scX\]
where $\pi_{V^*V}$ is a projection along the factors \(V_{G}^*\) and \(V_{B^{(2)}}\):
\[\pi_{V^*V}(X,g_1,Y_1,g_2,Y_2,w,v_1,v_2)=(X,g_1,Y_1,g_2,Y_2,g_1v_1)\]
while the Koszul matrix factorization on \(\scX^{\rfr}\)  on  is
\[\mathrm{K}^{\rfr}:=[w,g_1v_1-g_2v_2]\in \MF_{G\times B^2}^{\mathbb{T}_{q,t}}(\scX, \Tr(w(g_1v_1-g_2v_2))),\]

The inverse functor is defined in terms of embedding:
\[i_{w=0}: \scX\times V_{B^{(1)}}\times V_{B^{(2)}}\to \scX^{\rfr},\quad i_{w=0}(X,g_1,Y_1,g_2,Y_2,v_1,v_2)=(X,g_1,Y_1,g_2,Y_2,v_1,v_2,0)\]
\[i^{-1}_{\rfr}=\pi_{V*}\circ i_{w=0}^*,\]
here \(\pi_{V}\) is the projection along \(V_{B^{(2)}}\).

\end{proof}

The equivalence between \(\MF^{\rfr}\) and \(\MF^v\) is an example of Kn\"orrer periodicity equivalence \ref{prop:Kno}. In particular, one can see that \(i_{\rfr}=i_{w=0,*}\circ \pi^*_{V}\). The category \(\MF^v\)  has a monoidal structure defined by
\begin{equation}\label{eq:conv-v}
  \calF\star\calG:=(\pi_{13}\times 1)_{*}(\CE_{\frn^{(2)}}(\pi^*_{12}\times 1^*(\calF)\otimes \pi^*_{23}\times 1^*(\calG))^{T^{(2)}}),\end{equation}
here \(\pi_{ij}\) are the maps \eqref{eq:pi-ij}.

\begin{proposition}
  The functor \(i_{\rfr}\)  is a  monoidal equivalence.
\end{proposition}
\begin{proof}
  First we observe that with use of row transformation of Koszul matrix factorizations (see for example \cite[section 2.3]{OblomkovRozansky16}) one can obtain the equality of Koszul matrix factorizations
  \begin{multline*}
    \pi_{12}^*(\mathrm{K}^{\rfr})\otimes\pi_{23}^*(\mathrm{K}^{\rfr})=
    \pi_{12}^*([w,g_1v_1-g_2v_2])\otimes \pi^*_{23}([w,g_2v_2-g_3v_3])\\=[0,g_2v_2]\otimes
    [w,g_1v_1-g_3v_3]=[0,g_2v_2]\otimes \pi^*_{13}(\mathrm{K}^{\rfr})\end{multline*}
  Thus \(\pi_{12}^*(i^{\rfr}(\calF))\otimes\pi_{23}^*(i^{\rfr}(\calG)) 
  \) is homotopy equivalent to the restriction of \(\pi_{12}^*(\calF)\otimes \pi_{23}^*(\calG)\otimes \pi_{13}^*(\mathrm{K}^{\rfr})\)
  to the locus \(\scX_{\rcnv}^{\rfr}|_{v_2=0}\). Thus the statement follows from the projection formula for \(\pi_{13}\).
\end{proof}

The monoidal functor from above sends the unit \(\calC_{\parallel}\)
to the unit element \(\calC^{\rfr}_{\parallel}\) in the category \(\MF^{\rfr}\).
Let us spell out the construction of the last unit. Indeed we have
\(\calC_{\parallel}^{\rfr}=j_{\parallel*}^{\rfr}(\calO)\), where \[j_{\parallel}^{\rfr}\colon \scX^{\rfr}(B):=\frg\times B\times G\times\frn\times V_G^*\times V_G\rightarrow\scX\times V_G^*\times V_G\]
is defined by
\[j_{\parallel}(X,b,h,Y,w,v)=(X,h,Y,hb,\Ad_bY,h^{-1}(v),b^{-1}h^{-1}(v)).\]
There is a unique \(B^2\)-equivariant structure on the space \(\scX^{\rfr}(B)\) that makes the map
\(j_\parallel^{\rfr}\) equivariant.

\begin{remark}\label{rem:singular}  It seems to be natural to define \(\MF_{\Dr}^v\) as category of
  matrix factorizations on \(\scC_{\Dr}^v\subset \scC_{\Dr}\times V_G\)  defined by
  the condition \(v=gv\) and the potential defined by restricting
  \(W_{\Dr}\). So one can expect an intimate relation between \(\MF_{\Dr}^v\) and
  \(\MF_{\Dr}^{\rfr}\). However, we choose to avoid working with this category since
  \(\scC_{\Dr}^v\) is singular and we need to use some complicated tools to work with
  this category.
\end{remark}

\subsection{Framed Chern and co-Chern functors}
\label{sec:framed-chern-co}

The framed categories are connected by the framed versions of the functors \(\HC\) and \(\CH\). The framed version of the
space \(\scZ_{\CH}\) and maps \(\pi_{\Dr}\), \(f_\Delta\) are defined as
\[\scZ_{\CH}^{\rfr}:=\scZ_{\CH}\times V^*_G\times V_G, \quad \pi_{\Dr}\colon\scZ_{\CH}^{\rfr}\rightarrow\scC^{\rfr},\quad f_\Delta\colon \scZ_{\CH}^{\rfr}\rightarrow \scX^{\rfr} \]
\[\pi_{\Dr}(Z,g,X,h,Y,w,v)=(Z,g,X,w,v),\] \[f_\Delta(Z,g,X,h,Y,w,v)=(X,gh,Y,h,Y,w,h^{-1}g^{-1}(v),h^{-1}(v)).\]
We define the framed version of the Chern character functor \(\CH^{\rfr}\) by the formula \eqref{eq:CH}.
The framed version \(\HC^{\rfr}\) is defined by~\eqref{eq:HC}, where this time
\[\tilde{f}_\Delta\colon\scZ^{\rfr}_{\CH}\times B\rightarrow \scX^{\rfr}\times B,\]\[ \tilde{f}_\Delta(Z,g,X,h,Y,w,v,b)=(X,gh,Y,hb,\Ad_bY,h^{-1}(v),b^{-1}h^{-1}g^{-1}(v),w,b).\]

Similarly, we define \(\scZ^{0,\rfr}_{\CH}=\scZ^0_{\CH}\times V_G^*\times V_G\) and \(j^0:\scZ^{0,\rfr}_{\CH}\to \scZ^0_{\CH}\) is a natural embedding.

\begin{proposition}\label{prop:fsHC} Functors from the diagram below
  \begin{equation*}
    \begin{tikzcd}
      \MF^{\rfr}\arrow[rr,bend left,"\CH^{\rfr}"]&&\MF_{\Dr}^{\rfr}\arrow[ll,bend left,"\HC^{\rfr}",pos=0.435]
    \end{tikzcd},
  \end{equation*}
  have the following properties
  \begin{itemize}
  \item The functor \(\CH^{\rfr}\) is a right adjoint of \(\HC^{\rfr}\),
  \item The functor \(\HC^{\rfr}\) is monoidal.
  \item    For any \(\calD\in \MF_{\dr}^{\rfr}\) and a pair of Koszul matrix factorizations  \(\calc^+,\calc^-\in \MF^{\rfr}\) such that \(\calc^+\star\calc^-\sim \calc_\parallel^{\rfr} \),
    there is a homotopy equivalence:
    \begin{equation*}
      \calC^+\star \HC^{\rfr}(\calD)\star \calc^-\sim\HC^{\rfr}(\calD).
    \end{equation*}
  \end{itemize}
\end{proposition}
\begin{proof}
  Proof of the first statement  word by word repeats the proof of proposition \ref{prop:adj}. For second statement we can use the argument from proof of
  proposition~\ref{prop:mono}. We just need to use a framed version of the commuting diagram of maps of spaces:
  \[
    \begin{tikzcd}
      \scX& \scX_{cnv}^{\rfr}\arrow[l,"\pi_{13}"]\arrow[r,"i_{cnv}"]&\scX^{\rfr}\times\scX^{\rfr}\\
      \scZ^{0,\rfr}_{\CH}\times B\arrow[u,dotted,"\tilde{f}_{\Delta}\circ\pi_{B}"]\arrow[d,dotted,"\pi_{\Dr}\times 1"]& \mathscr{Y}^{\rfr}\arrow[r,dotted,"\hat{i}_{cnv}"]\arrow[u,dotted,"\hat{f}"]\arrow[l,dotted,"\pi_Y"]&\scZ^{0,\rfr}_{\CH}\times\scZ^{0,\rfr}_{\CH}\times B\arrow[u,"\tilde{f}_{\Delta}\times\tilde{f}_{\Delta}\circ \pi_B"]\arrow[d,"\pi_{\Dr}\times\pi_{\Dr}\times \pi_B"]\\
      \mathscr{C}^{0,\rfr}\times  B&\scC^{0,\rfr}_{cnv}\times G\times B\arrow[l,dotted,"\pi_3\times \pi_G"]\arrow[r,dotted,"j_{cnv}"]\arrow[u,"i_C",dotted]&\scC^{0,\rfr}\times\scC^{0,\rfr}
    \end{tikzcd}
  \]
  where \(\mathscr{Y}^{\rfr}=\mathscr{Y}\times V^*_G\times V_G\) and framed
  versions of the corresponding maps that start at \(\mathscr{Y}^{\rfr}\) are:
  \[\hat{i}_{cnv}(Z_1,Z_2,X,g_1,g_2,g_3,Y,b,w,v)=(Z_1,g_1g_2^{-1},X,g_2,Y,w,v,Z_2,g_2bg_3^{-1},X,g_3b^{-1},Y,w,v),\]
  \[\hat{f}(Z_1,Z_2,X,g_1,g_2,g_3,Y,b,w,v)=(X,g_1,Y,g_2,Y,g_3,\Ad_b Y,w,g_1^{-1}(v),g_2^{-1}(v),bg_3^{-1}(v)).\]
  Respectively, the framing alteration of the other map is more straight forward:
  \[i_C(Z,X,g',g'',Y,w,v,b,h)=(Z,\Ad_{g'g''}Z,X,g'h,h,(g'')^{-1}hb,Y,b,w,v),\]
  \[j_{cnv}(Z,X,g',g'',Y,w,v,b,h)=(X,Z,g',w,v,X,\Ad_{g'g''}Z,g'',w,v).\]
  With the last modifications of the proof of proposition~\ref{prop:mono} yields
  the second statement of the current  proposition.

  For a proof of the third statement we use the argument of proposition~\ref{prop:central}. We need to work with the framed versions of
  the spaces and maps. In particular, the framed space
  \[\hat{\scX}^{\rfr}=\frg\times (G\times \frn)^4\times B\times V^*_G\times V_{B^{(1)}}\times V_{B^{(2)}}\times V_{B^{(3)}}\times V_{B^{(4)}}\] has natural \(B^4\)-equivariant projections
  \(\hat{\pi}_{ij}:\hat{\scX}^{\rfr}\to \scX^{\rfr} \) as well as an embedding \(\tilde{f}_{\Delta}\colon\scZ_{\CH}^{\rfr}\times (G\times \frn)^2\times B\)
  defined as
  \begin{multline*}\tilde{f}_\Delta(Z,g,X,h,Y,w,v,g_1,Y_1,g_4,Y_4,b)\\=(X,g_1,Y_1,gh,Y,hb,\Ad_bY,g_4,Y_4,b,w,g_1^{-1}(v),h^{-1}g^{-1}(v),b^{-1}h^{-1}(v),g_4(v)).
  \end{multline*}
  The other maps used in the proof are defined by    the  relations: \(\hat{\pi}_{14}\circ\tilde{f}_\Delta =\pi_{14}\), \(\hat{\pi}_{12}\circ\tilde{f}_\Delta=\pi_{1\bullet}\),
  \( \hat{\pi}_{34}\circ\tilde{f}_\Delta=\pi_{\bullet 4}\).

  With these modifications proof of proposition~\ref{prop:central} yields the
  third statement of the current proposition.
\end{proof}

\subsection{Stable and framed stable categories}
\label{sec:stable-framed-stable}

The stable versions of the categories are defined on the stable pieces of the corresponding varieties. In particular, we have
embeddings:
\[j_{\rst}\colon\scX^{\rst}\hookrightarrow \scX\times V_G=\scX^v,\quad 
  \quad j_{\rfs}\colon \scX^{\rfs}\hookrightarrow \scX^{\rfr},\quad \scC^{\rfs}\hookrightarrow \scC^{\rfr}. \]
Denote by \(v\) the coordinate along the vector space \(V\). The open stability condition for \(\scX^{\rst}\) and \(\scX^{\rfs}\) is defined by the following requirements:
\begin{equation}\label{eq:stab-X}
  \CC\langle X,\Ad_{g_1}^{-1}Y_1\rangle\, v=V_G,\quad \CC\langle X,\Ad_{g_1}^{-1}Y_1\rangle\, g^{-1}_1(v_1)=V_G, 
\end{equation}
where \(v_1\) denotes coordinate along the factor \(V_{B^{(1)}}\) in \(\scX^{\rfr}\).
The open stability condition for \(\scC^{\rfs}\) is
\[ \CC\langle Z,X\rangle\, v=V_G.\]


The stable piece  \(\scX^{\rst}_{\rcnv}\subset\scX_{\rcnv}\times V\) of the convolution space is an open subset
which is an intersection
\[\scX^{\rst}_{\rcnv}:=(\pi_{12}\times 1)^{-1}(\scX^{\rst})\cap (\pi_{23}\times 1)^{-1}(\scX^{\rst}).\]

\begin{defn} We define stable categories 
  \[\MF^{\rst}:=\MF(\scX^{\rst},G\times B^2\times \Tqt,q^0t,W),\quad
   \MF^{\rfs}_{\Dr}:=\MF(\scC^{\rfs},G\times \Tqt,q^0t,W^{\rfr}),\] \[\MF^{\rfs}:=\MF(\scX^{\rfs},G\times B^2\times \Tqt,q^0t,W^{\rfr})\]
We define the convolution  product on \(\MF^{\rst}\)  by the  formula  \eqref{eq:conv-v}.
 The monoidal structure is defined on categories
\( \MF^{\rfs}_{\Dr}, \MF^{\rfs}\)
by restricting the corresponding maps to the stable pieces of the convolution spaces.
\end{defn}
The same argument as in Theorem 1.2.3 of \cite{OblomkovRozansky17} implies the following:
\begin{proposition}
  The pull-back maps \(j_{\rfs}^*\), \(j_{\rst}^*\circ\pi^*_V\), where \(\pi_V\)  is the projection along \(V\), are monoidal functors.
\end{proposition}
\begin{proof}
  The statement that \(j^*_{\rst}\circ\pi_V^*\) is a monoidal functor is shown in
  Theorem 1.2.3 of \cite{OblomkovRozansky16}. The same method yields the statement
  for \(j_{\rfs}^*\), we remind the details below. The key observation is the
  shrinking lemma from \cite{OblomkovRozansky16} (see lemma 12.3 from \cite{OblomkovRozansky16}). The lemma states that  the pull back to an open Zarisky subset is an equivalence of the categories
  of matrix factorizations as long as the open set contains all critical points
  of the potential.

  That is the category of matrix factorizations only depends on the formal neighborhood  of the critical locus.  Thus we
  need to study  the stable locus of the intersection of two critical loci \(Crit(\pi_{12}^*(W^{\rfr}))\) and \(Crit(\pi_{23}^*(W^{\rfr}))\) inside \(\scX^{\rfr}\).

  But on this intersection \(\Ad_{g_1}^{-1}Y=\Ad_{g_2}^{-1}Y=\Ad_{g_3}^{-1}Y\). Hence
  the stable condition can be imposed by requiring \[\mathbb{C}\langle X,\Ad_{g_3}^{-1}Y\rangle v=V_G.\] This open condition is constant along the fibers
  of projection \(\pi_{13}\) hence this condition does not interfere with the convolution and the statement follows.

\end{proof}

\subsection{Framed stable Chern and co-Chern functor}
\label{sec:framed-stabel-chern}

We define the framed stable version of the space \(\scZ_{\CH}\) as an open locus inside \(\scZ_{\CH}\times V\) defined as an intersection:
\[\scZ_{\CH}^{\rfs}:=\pi_{\Dr}^{-1}(\scC^{\rfs})\cap f_{\Delta}^{-1}(\scX^{\rfs}),\]
where the corresponding maps \(f_{\Delta}\), \(\pi_{\Dr}\) are defined by restricting the framed maps to the locus
where the covectors vanish. Thus we can define the functors \(\CH^{\rfs}\) and \(\HC^{\rfs}\) by the formulas \eqref{eq:CH}
and \eqref{eq:HC}. Moreover, the functors constructed in this section fit into the diagram

\begin{equation}\label{eq:HC-CH-dia}
  \begin{tikzcd}
    \MF^{\rfs}\arrow[rrr,bend left,"\CH^{\rfs}"]&\arrow[l,"j_{\rfs}^*"']\MF^{\rfr}\arrow[r,shift left=1,"\CH^{\rfr}"]&\arrow[l,shift left=1,"\HC^{\rfr}"]\MF_{\Dr}^{\rfr}\arrow[r,"j_{\rfs}^*"] &\MF_{\Dr}^{\rfs}\arrow[lll,bend left,"\HC^{\rfs}"]
  \end{tikzcd},
\end{equation}

\begin{proposition}
  Functors \(\CH^{\rfs}\) and \(\HC^{\rfs}\) 
  have the following properties
  \begin{itemize}
  \item The functor \(\CH^{\rfs}\) is a right adjoint of \(\HC^{\rfs}\),
  \item The functor \(\HC^{\rfs}\) is monoidal.
  \item The diagram \eqref{eq:HC-CH-dia} is commutative.
  \end{itemize}  
\end{proposition}
\begin{proof}
  The first and second statements are shown by repeating the argument from the proof
  of \ref{prop:fsHC}. The third statement follow from the base change since
  an open embedding is a smooth map. In more details, let us argue for
  statement about \(\CH^{\bullet}\). Indeed, for an element \(\calC\in
  \MF^{\rfr}\) we have
  \begin{equation*}
    j^*_{\rfs}\circ \CH^{\rfr}(\calC)=j^*_{\rfs}\circ \pi_{\Dr*}(\CE_{\frn}(\mathrm{K}_{\CH}\otimes (j^0_*\circ f^*_\Delta(\calC)))^{T})= \pi_{\Dr*}\Bigg( j^*_{\rfs}\bigg(\CE_{\frn}(\mathrm{K}_{\CH}\otimes (j^0_*\circ f^*_\Delta(\calC)))^{T}\bigg)\Bigg).
  \end{equation*}
  Next we observe that stability condition on the critical locus of the potential on \(\scZ^{\rfr}_{\CH}\) can be imposed
  requiring stability of the image of the push-forward along \(\pi_{\dr}\).
  Hence, the map \(j_{\rfs}\) preserves the critical part of   fibers of the projection 
  \(\pi_{\dr}\). On the other hand the functor
  \(\CE_{\frn}(-)^T\) is constant along the fibers of \(\pi_{\dr}\). Thus last expression is equal:
  \begin{equation*}
    \pi_{\Dr*}\Bigg( \CE_{\frn}\bigg(j^*_{\rfs}\bigg(\mathrm{K}_{\CH}\otimes (j^0_*\circ f^*_\Delta(\calC)\bigg)\bigg)^{T}\Bigg)=\pi_{\Dr*}\Bigg( \CE_{\frn}\bigg(j^*_{\rfs}(\mathrm{K}_{\CH})\otimes (j_{\rfs}^*\circ j^0_*\circ f^*_\Delta(\calC)\bigg)^{T}\Bigg).
  \end{equation*}
  Finally, we use the base change we get
  \[j^*_{\rfs}\circ j_*^0\circ f_\Delta^*=j^0_*\circ j_{\rfs}^*\circ f^*_\Delta=
    j^0_*\circ f^*_\Delta\circ j_{\rfs}^*.\]
  Combining with the previous computation we obtain
  \[j^*_{\rfs}\circ \CH^{\rfr}(\calC)=\CH^{\rfs}\circ j^*_{\rfs}(\calC).\]
  Similar, argument implies
  \[j^*_{\rfs}\circ \HC^{\rfr}(\calD)=\HC^{\rfs}\circ j^*_{\rfs}(\calD), \quad \calD\in \MF^{\rfr}.\]
\end{proof}

\subsection{Linearized  categories}
\label{sec:line-kosz-dual}


In the case of \(\bullet=\rfs\) we show that the linearized category is actually equivalent to the
original category, thus we complete the construction of the functor
\(\CH^{\rfs}_{\rloc}\)  introduced at the beginning of the section.

The variety \(\scC=\frg\times G\times\frg\) has coordinates \((Z,g,X)\).
A coordinate substitution $Z = R g$, \(R\in\frg\)
on our main variety \(\scC\) makes the potential tri-linear:
\[\underline{W}_{\Dr}(R,g,X)=\Tr(X[g,R])=W_{\Dr}(Rg,g,X).\]
The framed potential also linearizes:
\[\underline{W}_{\Dr}^{\rfr}(R,g,X,v,w)=\underline{W}_{\Dr}(R,g,X)+\Tr(w(v-gv)).\]
Thus we introduce linearized categories:
\[\underline{\MF}^\bullet_{\Dr}:=\MF_{G}(\underline{\scC}^{\bullet},\underline{W}^{\bullet}_{\rlin}),\]
where \(\scC^\bullet\) is the open set in \(\underline{\scC}^\bullet\) (the group \(G=\GL_n\)
an affine open subset of \(\frg=\gl_n\))
and \(\underline{W}^{\rfs}=\underline{W}^{\rfr}\). That is
\[\underline{\scC}=\frg^3,\quad \underline{\scC}^{\rfr}=\frg^3\times V_G\times V_G^*.\]

Since \(j_{G}\colon\scC^\bullet\hookrightarrow\underline{\scC}^\bullet\) is an open embedding, the pull-back functor \(j_G^*\)
is a localization functor and we denote
\[\mathrm{loc}^\bullet\colon\underline{\MF}^\bullet_\Dr\rightarrow\MF^\bullet_\Dr\]
for this functor.

\begin{remark} The categories \(\underline{\MF}^\bullet_\Dr\) do not have a monoidal structure that respects the localization functor.
\end{remark}

\begin{proposition}\label{prop:loc-iso}
  The functor \(\mathrm{loc}^{\rfs}\) is an equivalence of categories.
\end{proposition}
\begin{proof}
  The compliment to \(j_G(\scC)^{\rfs}\) inside \(\underline{\scC}^{\rfs}\) is defined
  by the equation \(\det(g)=0\).  
  Then according to the lemma 12.3 from \cite{OblomkovRozansky16},
  the pull back to a Zarisky open subset is an equivalence of the categories
  of matrix factorizations as long as the open set contains all critical points
  of the potential.
  Thus it is sufficient to show that the zero locus of \(\det(g)=0\) does not intersect the critical locus of the potential.

  The critical locus \(Z_{crit}\) is given by the system of equations:
  \[[X,R]-wv=0,\quad [g,X]=[g,R]=0,\quad wg=w,\quad gv=v.\]
  These equations appear in the description of the moduli space \(\mathscr{M}^{\rfr}\) is the
  section~\ref{sec:local-funct}. Moreover, the stability condition for
  \(\underline{\scC}^{\rfs}\) is equivalent to the stability condition for \(\mathscr{M}^{\rfs}\). In other words
  \[Z_{crit}\subset \mathscr{M}^{\rfs}\times \frg.\]

  On other hand proposition \ref{prop:Nak} and  the last three equations for \(Z_{crit}\) imply that \(g= Id\).
   Slightly rephrasing, we can argue that the stability condition in section~\ref{sec:stable-framed-stable} says that \(v\) is a
cyclic vector for \(X\) and \(Z = Rg.\) If g commutes with \(X\) and \(R\) then it
also commutes with \(Z\), so \(gv = v\) implies that \(g = Id.\)

  \end{proof}

  \begin{remark}
    The stability condition of \(\MF^{\rfs}\) transported by \(\CH^{fs}\) defines a point  \[(Z,g,X,v,w)\in\scC^{\rfr}=\frg\times G\times \frg\times V_G\times V_G^*\] to be stable if \(\CC\langle Z,X\rangle v=V_G\). On the other  hand by the critical locus  argument of the previous proposition the last open condition is equivalent
    to a condition \(\CC\langle R,X\rangle v=V_G\).
  \end{remark}
  
\subsection{Koszul duality}
\label{sec:koszul-duality}

Recall a general setup of the Koszul duality. We are interested in the version which provides a matrix factorization model for the differential graded category
of the derived complete intersection. The Koszul duality in this context was discussed in \cite{ArkhipovKanstrup15a} where the results
of \cite{Orlov12} and \cite{Isik10} are combined. In this note we present slightly more streamlined
construction of the duality functor.

Derived algebraic geometry is explained in many places, here we present a derived algebraic geometry treatment of derived complete intersections
in the most elementary setting sufficient for our needs.
%

Initial data for an affine derived complete  intersection is a collection of elements
\(f_1,\dots,f_m\in\CC[\mathbf{X}]\). In this section we assume that \(X\) is smooth, affine and of finite dimension. It determines the differential graded algebra
\[\calR=(\CC[\mathbf{X}]\otimes \Lambda^* U,D),\quad D=\sum_{i=1}^mf_i\frac{\partial}{\partial \theta_i},\]
where \(\theta_i\) from a basis of \(U=\CC^m\) and the variables \(\theta_i\) anti-commute \(\theta_i\theta_j=-\theta_j\theta_i\).
Thus algebra \(\calR\) is super-commutative.

Respectively, we define dg category of coherent sheaves on
\(\Spec(\calR)\) as
\[\Coh(\Spec(\calR))=\frac{\{\text{bounded complexes of finitely generated  $\calR$  dg modules}\}}{\{\text{quasi-isomorphisms}\}}.\]

We forfeit the homological grading in our setting and we assume that the complexes
are two periodic. In our proofs we use usual complexes to simplify
notations. Two periodic folding  turns usual complexes to two-periodic: \[(\oplus_{i\in\ZZ}M_i, D_i:M_i\to M_{i+1}) \Longrightarrow (\mathbb{M}_0\oplus \mathbb{M}_1, \mathbb{D}_i: \mathbb{M}_i\to \mathbb{M}_{i+1}),\]
\[  \mathbb{M}_0=\oplus_{i\in 2\ZZ}M_i,\quad\mathbb{M}_1 \oplus_{i\in 1+2\ZZ}M_i,
\quad \mathbb{D}_0=\sum_{i\in \ZZ}D_{2i},\quad \mathbb{D}_0=\sum_{i\in \ZZ}D_{2i+1}.\]

Thus defined dg category generalizes the derived category of two-periodic complexes
\(\mathrm{D}^{\rper}(\mathbf{Z})\), in the following sense. If the intersection \(\mathbf{Z}\)
of \(f_i=0\), \(i=1,\dots,m\) is transverse and \(\mathbf{Z}\) is smooth then ( by applying two-periodic folding to 
\cite{Raskin14}):
\[\Coh(\Spec(\calR))=\mathrm{D}^{\rper}(\mathbf{Z}).\]

Suppose that the ambient space \(\mathbf{X}\) carries an action of a group \(G\) and
that the ideal \(f_1,\dots,f_m\) is preserved by the action:
\(g(f_i)=\sum_{j}a(g)_{ij}f_j\), \(a_{ij}\in \CC[X]\), \(g\in G\). The matrix \((a(g)^{-1})^*\) defines an action of  \(g \in G\) on \(U^*\otimes\CC[\mathbf{X}]\) and the  corresponding  \(G\)-action preserves the differential \(D\). In this setting
we have a well-defined \(G\)-equivariant dg category
\(\Coh(\Spec(\calR),G)\).

Consider a potential on \(\mathbf{X}\times U\):
\begin{equation}\label{eq:Potential}
  W=\sum_{i=1}^m f_i(x)z_i,
\end{equation}
where \(z_i\) is a basis of \(U^*\) dual to the basis \(\theta_i\).  For the Koszul matrix factorization:
\[
  \MF(\mathbf{X}\times U,W)\ni\mathbb{B}=(\calR\otimes \CC[U], D_{\mathbb{B}}),\quad D_{\mathbb{B}}=\sum_{i=1}^m z_i\theta_i+f_i\frac{\partial}{\partial \theta_i}.\]
We have \(D_{\mathbb{B}}^2=W\) because of \eqref{eq:Potential}.
For a \((M,D_M)\)  dg module over \(\calR\),  the tensor product
\[\mathrm{KSZ}_U(M):=M\otimes_{\CC[X]\otimes\Lambda^*(U)} \mathbb{B}\] 
is an object of \(\MF(\mathbf{X}\times U,W)\) with the differential \(D=D_M\otimes 1+ 1\otimes D_{\mathbb{B}}\). In the last formula we use usual tensor product, not the derived one,
because \(\mathbb{B}\) is free over \(\CC[\mathbf{X}]\otimes\Lambda^*(U)\).
The map \(\mathrm{KSZ}_U\) extends to a functor between triangulated categories:
\[\mathrm{KSZ}_U\colon \Coh(\Spec(\calR))\rightarrow \MF(\mathbf{X}\times U,W).\]

The functor in the other direction is based on the dual matrix factorization:
\[\MF(\mathbf{X}\times U,-W)\ni B^*=(\calR\otimes \CC[U],D_B^*),\quad D^*_{\mathbb{B}}=-\sum_{i=1}^m z_i\theta_i+f_i\frac{\partial}{\partial \theta_i},\]
\[\mathrm{KSZ}_U^*\colon\MF(\mathbf{X}\times U,W)\rightarrow \Coh(\Spec(\calR)),\quad \mathrm{KSZ}^*_U(\calF):=
  \Hom_{\calR}(\calF\otimes_{\CC[\mathbf{X}\times U]}\mathrm{B}^*,\calR).\]

\begin{theorem}\label{thm:KSZ} Suppose \(X\) is smooth and quasi-affine.
  Then  the functors \(\mathrm{KSZ}_U\) and \(\mathrm{KSZ}_U^*\) are mutually inverse:
  \[\mathrm{KSZ}_U\circ\mathrm{KSZ}_U^*\simeq \mathbf{D},\quad \mathrm{KSZ}_U^*\circ\mathrm{KSZ}_U\simeq \mathbf{D},\]
  where \(\mathbf{D}\) are respective duality functors. In particular, for a dg \(\calR\)-module \(M\) we define \(\mathbf{D}(M)=\Hom_\calR(M^\bullet,\calR)\) where
  \(M^\bullet\) is a free resolution. 
\end{theorem}
\begin{proof}
   Let us indicate the main observation behind the argument.
  To simplify notations we assume that
  \(m=1\) and \(\mathbf{X}\) is affine. Respectively, we set \(\CC[\mathbf{X}][z]\) to be a ring of regular functions on
  \(\mathbf{X}\times \CC_z\) and the underlying super-commutative ring of \(\calR\) is \(\CC[\mathbf{X}][\theta]\).
  First we prove that \(\mathrm{KSZ}^*_U\circ \mathrm{KSZ}_U= \mathbf{D}\).
  A dg \(\calR\)-module $M$
  has a free resolution \(M^\bullet\) since
  \(X\) is an affine space. Thus the composition \(\mathrm{KSZ}^*_U\circ \mathrm{KSZ}_U(M)\) is
  equal to
  \begin{equation}
    \label{eq:BB1}
    \Hom_\calR(B^*\otimes_{\CC[\mathbf{X}][z]}B,\calR)\otimes_{\calR}\check{M}^\bullet=
    (\Hom_{\calR}(\CC[\mathbf{X}][z,\theta^{(1)},\theta^{(2)}],\calR),\check{D}_{\mathbb{BB}})\otimes_{\calR} \check{M}^\bullet,\end{equation}
   where \(\check{M}^\bullet=\mathbf{D}(M^\bullet)\) and the differential \(\check{D}_{\mathbb{BB}}\) is dual to the differential
  \[D_{\mathbb{BB}}=z\theta_-+f(x)\frac{\partial}{\partial\theta_+},\quad \theta_\pm=\theta^{(1)}\pm\theta^{(2)} \]
  acting on \(\mathbb{B}\otimes_{\CC[X][z]}\mathbb{B}=\CC[X][z,\theta^{(1)},\theta^{(2)}]\).

  Since \(\theta_-\) and \(\frac{\partial}{\partial \theta_+}\) commute, the differential \(D_{\mathbb{BB}}\) is a sum of two differentials: \(D_{\mathbb{BB}}^{(1)}=z\theta_-\) and \(D_{\mathbb{BB}}^{(2)}=f(x)\frac{\partial}{\partial \theta_+}\). Next we analyze the complex for the differentials that is dual to
  \(D_{\mathbb{BB}}^{(1)}\).
 Indeed,  there is a homotopy equivalence of the
  complexes of \(\CC[\theta_-]\)-modules
  \begin{equation}\label{eq:ztheta}
    (\CC[z,\theta_-],z\theta_-)\sim [
 \CC[\theta_-]\xrightarrow{\theta_-}\CC[\theta_-]\xrightarrow{\theta_-}\CC[\theta]\xrightarrow{\theta_-}\dots].\end{equation}
The duality functor \(\Hom_{\CC[\theta_-]}(\bullet,\CC[\theta_-])\) inverts the arrow, and we obtain a resolution of
the simple \(\CC[\theta_-]\)-module:
\begin{equation}\label{eq:contract}
  [\CC[\theta_-]\xleftarrow{\theta_-}\CC[\theta_-]\xleftarrow{\theta_-}\CC[\theta_-]\xleftarrow{\theta_-}\dots]\sim \CC.
\end{equation}
Hence the complex \eqref{eq:BB1} is quasi-isomorphic to \(\check{M}^\bullet\) and the statement follows.

The analysis of the composition in other direction \(\mathrm{KSZ}_U\circ\mathrm{KSZ}^*_U\) is even easier, because it is reduced to the
study of the dg module
\begin{equation}\label{eq:D'}
  (\mathbb{B}\otimes_{\CC[X][\theta]}\mathbb{B}^*,D'_{\mathbb{BB}}),\quad D'_{\mathbb{BB}}=\theta(z^{(1)}+z^{(2)}), \end{equation}
which is quasi-isomorphic to the dg module \(\CC[X][z^{(1)},z^{(2)}]/(z^{(1)}+z^{(2)})\) with the trivial differential. Hence the statement follows.


The action on the
of morphisms between the objects is defined in a standard way:
\begin{equation}\label{eq:KSZ-mor}
  \mathrm{KSZ}_U(\phi)=\phi\otimes 1,\quad \mathrm{KSZ}_U^*(\Psi)=\Psi^\vee\otimes 1,\end{equation}
where \(\Psi^\vee\in \Hom(\calG^\vee,\calF^\vee)\) is dual to \(\Psi\in \Hom(\calF,\calF^\vee)\).

The homotopy relation \eqref{eq:contract} only affects the second tensor factor in the
LHS of the formulas \eqref{eq:KSZ-mor}.
Thus the construction applied to the space of morphism
yields
\[\mathrm{KSZ}_U\circ\mathrm{KSZ}_U^*(\Psi)=\Psi^\vee,\quad \mathrm{KSZ}_U^*\circ\mathrm{KSZ}_U(\phi)=\phi^\vee.\]
Thus  \(\mathrm{KSZ}_U\) and \(\mathrm{KSZ}_U^*\) are fully faithful functors that is bijective on isomorphism classes of objects thus these functors are equivalences.

So far we discussed the case \(m=1\) and \(X\) affine. In the case \(m>1\) we need work with the tensor product of complexes \(K_i=(\CC[z_i,\theta_{i,-}],z_i\theta_{i,-})\), \(i=1,\dots,m\),
\(\theta_{i,-}=\theta_i^{(1)}-\theta_i^{(1)}\) instead of one complex \eqref{eq:ztheta}. For each \(K_i\) we have the contraction~\eqref{eq:contract} of the dual complex
\(K_i^\vee \) and the rest of argument goes through. Similarly, for \(m>1\) we need to modify \eqref{eq:D'} by working with \(D'_{BB}=\sum_i\theta_i(z_i^{(1)}+z_i^{(2)})\).

The case of \(\mathbf{X}\) being is quasi-affine is completely parallel to the affine case since we can use apply the affine case proof to each term of the Cech resolution, see
section~\ref{sec:quasi}.
\end{proof}

If we have a group \(G\) acting on \(\mathbf{X}\) and acting linearly on \(U\) and on \(W\in \CC[\mathbf{X}\times U]^G_\chi\), \(\chi\) is a character of \(G\) then the above functors extend to the isomorphism between the equivariant
version of the categories:
\begin{equation}\label{eq:KoszG}
  \Coh(\Spec(\calR),G,\chi)\simeq \MF(\mathbf{X}\times U,G,\chi,W).
  \end{equation}

In the last formula we work with the strongly equivariant matrix factorizations, thus all differentials and homotopies in the (curved) complexes are  \(G\)-equivariant.
Hence no technical difficulties related to Chevalley-Eilenberg correction differential arise and  argument from the  above proves  \eqref{eq:KoszG} after a slight adjustment of  notations. Let us also point out that below we only need the Koszul duality for the strongly equivariant matrix factorizations since
the elements of  \(\MF_{\Dr}^\bullet\) are strongly equivariant with respect to \(G=\GL(n)\) according to our main assumption \eqref{eq:assum}.

\begin{remark}\label{rem:KSZ} Suppose that the scheme \(\mathbf{Z}\subset \mathbf{X}\) defined by \(f_1,\dots,f_m\) is a smooth complete intersection.
   Then
  the functor \(\mathrm{KSZ}_U\) has a description in terms of push-forward functor. In more details,
  \[\mathrm{KSZ}_U=i_{f=0,*}\circ \pi_U^*: D^{\rper}(\mathbf{Z})=\MF(\mathbf{Z},0)\to \MF(\mathbf{X}\times U,W)\] where \(i_f\) and
  \(\pi_U\) are natural inclusion and projection
   \[i_{f=0}: \mathbf{Z}\times U\to \mathbf{X}\times U,\quad \pi_U: \mathbf{Z}\times U\to \mathbf{Z}.\]
   \end{remark}

\subsection{Koszul functors}
\label{sec:local-funct}

In our setting \(U=\frg\), while \(\mathbf{Z}^\bullet=\mathscr{M}^\bullet\subset \mathbf{X}^\bullet\) is a version of the commuting variety:
\begin{equation}\label{eq:CV}
\mathscr{M}\subset \frg^2=\mathbf{X},
\quad \mathscr{M}^{\rfr}, \mathscr{M}^{\rfs}\subset \frg^2\times V_G\times V^*_G=\mathbf{X}^\rfr. 
\end{equation}
The closed conditions describing the spaces \(\mathscr{M}\) are
\begin{equation}\label{eq:unfr-eqs}
[Z,X]=0, \quad Z,X\in \frg,\end{equation}
for the spaces \(\mathscr{M}^{\rfr},\mathscr{M}^{\rfs}\) are
\begin{equation}\label{eq:framed-eqs}
[Z,X]-vw=0, \quad Z,X\in \frg,\quad v\in V_G, w\in V^*_G,\end{equation}
the open condition for the stable spaces is:
\(\CC[Z,X]v=V_G.\)

Let us also recall a construction of \(\Hilb_n\). The variety \(\Hilb_n\)
is \(G\) a quotient of \(\widetilde{\Hilb}_n\subset \frg^2\times V_G\) where
\[(Z,X,v)\in\widetilde{\Hilb}_n\mbox{ iff } \CC[Z,X]v=V_G.\]

The relation of \(\Hilb_n\) with moduli space \(\mathscr{M}^{\rfs}\) is explained
in the book by Nakajima\cite{Nakajima99}. First we observe that there is
a natural \(G\)-equivariant embedding
\[\widetilde{\Hilb}\to \mathscr{M}^{\rfr},\quad (Z,X,v)\mapsto (Z,X,v,0).\]
 with \(G\)-action on \(\mathscr{M}^{\rfr}\),
\(g\cdot (Z,X,v,w)=(\Ad_gZ,\Ad_gX,gv,wg^{-1}).\)

\begin{proposition}\label{prop:Nak}\cite[Proposition 2.8, Theorem 1.9]{Nakajima99} For any \(n\) we have
  \begin{enumerate}
  \item \(\mathscr{M}^{\rfs}=\widetilde{\Hilb}_n\) as subvariety of \(\mathscr{M}^{\rfr}\).
    \item  For any \((Z,X,v)\in \widetilde{\Hilb}_n\) its stabilizer is \( \Id\in G\).
  \end{enumerate}
\end{proposition}

We define the dg algebra \(\calR^\bullet\) as the algebra \(\CC[X^\bullet]\otimes\Lambda^* \frg\) with the differential
\begin{equation}\label{eq:DG}
  D^\bullet=\sum_{ij}f_{ij}^\bullet\frac{\partial}{\partial \theta_{ij}},
  \end{equation}
where \(\theta_{ij}\) is a basis of matrix units in \(\frg\) and \(f^\bullet_{ij}\) is the \(ij\)-entry of the corresponding
version of the defining equations \eqref{eq:unfr-eqs}, \eqref{eq:framed-eqs} for the embedding \eqref{eq:CV}.

To simplify our notations we denote
\[\Coh^\bullet:=\Coh_{G}(\Spec{\calR^\bullet}).\]

Finally, let us notice that in the case \(\bullet=\rfs\) by proposition~\ref{prop:Nak} the equations \(f^{\rfs}_{ij}\) define a smooth complete intersection.
Hence, we have an equivalence of categories
\begin{equation}\label{eq:Cohfs-iso}
     \MF^{\rfs}_{\Dr}\simeq \underline{\MF}^{\rfs}_{\Dr}\simeq\Coh^{\rfs}\simeq \mathrm{D}^{\rper}(\Hilb_n(\CC^2),{\mathbb{T}_{q,t}},q^0t).
\end{equation}

Indeed, since the potential \(\underline{W}^\bullet\) is linear as a function of \(g\in\frg\) and the scaling torus \(\mathbb{T}_{q,t}\) does not act on
\(g\), we obtain a pair of mutually inverse functors:
\[
  \begin{tikzcd}
    \underline{\MF}^\bullet_{\Dr}\arrow[rr,bend left,"\mathrm{KSZ}_{\frg}^*"]&&\Coh^\bullet\arrow[ll,bend left,"\mathrm{KSZ}_\frg"]
  \end{tikzcd}
\]

The tensor product of dg modules gives categories \(\Coh^\bullet\) monoidal structure. It appears that there is no natural monoidal
structure on the linearized category \(\underline{\MF}^{\bullet}_{\Dr}\). Thus, it is not meaningful to discuss monomial structure of
the functor \(\mathrm{KSZ}_{\frg}\).  However, in the case \(\bullet=\rfs\) we can use a sequence of equivalences to define a bivariant operation  structure
\begin{equation} \label{eq:conv-Hilb}
  \star:   \mathrm{D}^{\rper}(\Hilb_n(\CC^2),{\mathbb{T}_{q,t}},q^0t)\times \mathrm{D}^{\rper}(\Hilb_n(\CC^2),{\mathbb{T}_{q,t}},q^0t)\to \mathrm{D}^{\rper}(\Hilb_n(\CC^2),{\mathbb{T}_{q,t}},q^0t)
\end{equation}
by intertwining the monoidal structure \(\star\) on      \(\MF^{\rfs}_{\Dr}\), , see definition~\ref{def:fr-cats}, by the isomorphism \eqref{eq:Cohfs-iso}.
Moreover, the composition \(\loc^\bullet\circ\mathrm{KSZ}_{\frg}\), that is used in \eqref{eq:Cohfs-iso}, has the following monoidal properties. Below we use that  the categories \(\MF^{\bullet}_{\dr}\) have monoidal structure \(\star\), as described in definition~\ref{def:fr-cats}.

\begin{proposition}\label{prop:compare-star}
  Let \(\tilde{M'},\tilde{M''}\in \MF(\mathbf{X}^\bullet,G,q^0t,0)\), \(\bullet=\emptyset,\rfr,\rfs\) as in \eqref{eq:CV}.
Respectively,   \(M',M''\in \mathrm{Coh}^\bullet\) are corresponding complexes of
  dg-modules.
Then  there is an isomorphism
  \[\loc^\bullet\circ \mathrm{KSZ}_{\frg}(M'\otimes M'')\simeq\loc^\bullet\circ \mathrm{KSZ}_{\frg}(M')\star \loc^\bullet\circ \mathrm{KSZ}_{\frg}(M'')\]
\end{proposition}
\begin{proof}
  In the proof we examine the case \(\bullet=\emptyset\), the other case are analogous.
  The Koszul matrix factorization that is used in our construction of  the Koszul duality \(\mathrm{KSZ}_{\frg}\) is
  \(\mathbb{B}=[g-1,[X,R]]\), in notations of section \ref{sec:koszul-duality}. On the other hand we have shown in proposition~\ref{prop:unit} that \(\mathbb{B}=\mathcal{O}_{\dr}\) (see \eqref{eq:unit}) is the
  unit in \(\MF_{\Dr}\).

  Let us use notation from the definition~\ref{def:unfr-conv}. In particular, the the coordinates on the convolution space \(\frg\times G^2\times \frg\) are \((Z,g_1,g_2,X)\). Respectively,
  the convolution of  \(\mathbb{B}\star \mathbb{B}\) is computed as push-forward along \(\pi_3\) of Koszul matrix factorization:
  \[\pi_{1}^*(\mathbb{B})\otimes \pi_{2}^*(\mathbb{B})=
    K(g_1-1)\otimes \pi_{3}^*(\mathbb{B}).\]
  The last equality a consequence of the row transformation of the Koszul matrix factorizations and \(K(g_1-1)\) is the Koszul complex
  of \(g_1-1\). In more details:
  \begin{multline*}
    \pi_{1}^*(\mathbb{B})\otimes \pi_{2}^*(\mathbb{B})=[g_1-1,[X,Zg_1^{-1}]]\otimes
    [g_2-1,[X,g_1Zg_1^{-1}g_2^{-1}]]\\=[g_1-1,[X,Zg_1^{-1}]]\otimes [g_3g_1^{-1}-1,[X,g_1Zg_3^{-1}]],
  \end{multline*}
  the last equation we used variable \(g_3=g_2g_1\). Using row operations, see for example \cite[section 2.3]{OblomkovRozansky16} we eliminate all instances of \(g_1\) and \(g_1^{-1}\) from the second Koszul matrix factorization with help of the first matrix factorization. The result is
  \[[g_1-1,R_1]\otimes [g_3-1,R_2]\]
  where \(R_2\) does not depend on \(g_1\).
  The last matrix factorization has potential \(\pi^*_3(W_\dr)\) and this potential does not depend on \(g_1\) hence \(R_1=0\). On the other hand \([g_3-1,R_2]\) is the matrix factorization with the potential \(\pi^*_3(W_\dr)\). Thus \(R_2=[X,Zg_3^{-1}]\) by the uniqueness lemma \cite[Lemma 3.6]{OblomkovRozansky16} since the entries of \(g_3-1\)
  form a regular sequence.
  
  On the other hand
  \(\mathrm{KSZ}_{\frg}(M')\star\mathrm{KSZ}_{\frg}(M'')\) is a push-forward \(\pi_{3*}\) applied to
  \[\pi_1^*(M'\otimes_{\CC[\mathbf{X}]\otimes\Lambda^\bullet \frg} \mathbb{B})\otimes \pi_2^*(M''\otimes_{\CC[\mathbf{X}]\otimes\Lambda^\bullet \frg} \mathbb{B}).\]
  By the previous remark we can rewrite the last expression as
  \[\tilde{\pi}_1^*(\tilde{M'})\otimes_{\mathcal{A}} \pi^*_3(\mathbb{B})\otimes \tilde{\pi}_2^*(\tilde{M''})\otimes_{\mathcal{A}} K(g_1-1).\]
  where \(\mathcal{A}=\pi_1^*(\CC[\mathbf{X}]\otimes\Lambda^\bullet \frg)=\pi_2^*(\CC[\mathbf{X}]\otimes\Lambda^\bullet \frg)\) and \(\tilde{\pi}_i:
  \frg\times G^2\times \frg\to \frg^2\) is a composition of \(\pi_i\) with a projection
  along the group factor
  \(p_G:\frg\times G\times \frg\to \frg^2\).

  Let us define map \(\tilde{\pi}_3\) similarly \(\tilde{\pi}_3=p_G\circ \pi_3\).
  The maps \(\tilde{\pi}_1\), \(\tilde{\pi}_2\) coincide  with the map \(\tilde{\pi}_3\)
  on the locus defined by \(Z=\Ad_{g_1}Z\).
  Thus the presence of the  Koszul complex
  \(K(g_1-1)\) allows us to construct a homotopy
    \[\tilde{\pi}_2^*(\tilde{M}'')\otimes_{\mathcal{A}} K(g_1-1)\sim \tilde{\pi}_3^*(\tilde{M}'')\otimes_{\mathcal{A}} K(g_1-1),
    \]
  Hence we arrive to the matrix factorization
  \[\pi_3^*(M'\otimes_{\CC[\calC]\otimes\Lambda^\bullet \frg} M''\otimes_{\CC[\mathbf{X}]\otimes\Lambda^\bullet \frg}\mathbb{B})\otimes K(g_1-1).\]
  Applying \(\pi_{3*}\) to the last expression we obtain the statement by the projection
formula.
\end{proof}

\begin{remark}
 It is predicted  by Kapustin and Rozansky \cite{KapustinRozansky10},   the monoidal structure on \(\Coh^\bullet\) induced by the Koszul duality, as in above proposition, is a deformation of the standard monoidal structure. We do not explore this fascinating structure in this paper. 
\end{remark}

\subsection{Localized Chern and co-Chern functors}
\label{sec:localized-chern-co}

Since \(\mathrm{loc}^{\rfs}\) is invertible, we can
use the Koszul duality functor in order to construct the functor
\[\CH^{\rfs}_{\rloc}:=\CH^{\rfs}\circ (\mathrm{loc}^{\rfs})^{-1}\circ \mathrm{KSZ}^*_\frg: \MF^{\rfs}\rightarrow \mathrm{D}^{\rper}(\Hilb_n(\CC^2),\mathbb{T}_{q,t},q^0t).\]

The localization functor does not seem to be invertible in  cases of  \(\bullet=\emptyset,\rfr\) however a construction of the
functor in the opposite direction does not require invertibility of the localization:
\[\HC^\bullet_{\rloc}:=\HC^\bullet\circ \mathrm{loc}^\bullet\circ \mathrm{KSZ}_\frg\colon\Coh^\bullet\rightarrow\MF^\bullet,\quad \bullet=\emptyset,\rfr,\rfs.\]

To complete the construction of functors \(\CH^{\rst}_{\rloc}\), that is discussed in the introduction we need the following

\begin{proposition}\label{prop:st-sq} The functor \(i_{\rfr}\) restricts to an isomorphism \(i_{\rfs}\) between the stable
  framed and stable categories:
  \[\begin{tikzcd}
      \MF^{\rst}\arrow[r,"\sim","i_{\rfs}"'] &\MF^{\rfs}\\
      \MF^{v}\arrow[r,"\sim", "i_{\rfr}"']\arrow[u,"j_{\rst}^*"]&\MF^{\rfr}\arrow[u,"j_{\rfs}^*"]
    \end{tikzcd}
  \]  
\end{proposition}
\begin{proof}
  The functors \(i_{\rfr}\)   and \(i_{\rfs}\) are defined as composition \(i_{w=0*}\circ \pi_{V}^*\). Both maps  have presentations \(i_{w=0}=\tilde{i}_{w=0}\times
  Id_{\calX^v}\), \(\pi_{V}=\tilde{\pi}\times Id_{\calX_v}\). On the other hand the stability condition on \(\calX^{\rfr}\) is defined in terms of the
  projection \(\calX^{\rfs}\to \calX^v\). Thus the maps in the diagram in the statement form a commutative diagram. Finally, the upper arrow is an isomorphism
  because \(\tilde{i}_{w=0*}\circ\tilde{\pi}_V^*\) is the Knorrer equivalence of categories \ref{prop:Kno}.
\end{proof}

Thus we can define the functors \(\CH^{\rst}_{\rloc}\), \(\HC^{\rst}\) by the commutative diagram:

\begin{equation}\label{eq:dia-stl}
\begin{tikzcd}
  \MF^{\rst}\arrow[r,"\sim","i_{\rfs}"']\arrow[rrrr,bend left=20,"\CH^{\rst}_{\rloc}"] &\MF^{\rfs}\arrow[r,shift left=1,"\CH^{\rfs}"]&\MF^{\rfs}_{\Dr}\arrow[l,shift left=1,"\HC^{\rfs}"]&\underline{\MF}^{\rfs}_{\Dr}\arrow[l,"\sim"',"\mathrm{loc}^{\rfs}"]&\mathrm{D}^{\rper}(\Hilb_n,\mathbb{T}_{q,t},q^0t)\arrow[l,"\sim"',"\mathrm{KSZ}_{\frg}"]\arrow[llll,bend left=20,"\HC^{\rst}_{\rloc}"]
  \end{tikzcd}
\end{equation}

As shown in \cite{OblomkovRozansky17}, that  there are monoidal functors from the  finite braid group
\[ 
  \Phi\colon \Br_n\to \MF^{\rst}.\]

Since \(\scX^v=\scX^v\times V_G\)  and the extra factor \(V_G\) does not participate in the convolution product \eqref{eq:conv-v} we have commuting diagram:
  \[\begin{tikzcd}
      \MF\arrow[r,"\pi_V^*"] &\MF^{v}\arrow[r,"\sim","i_{\rfr}"']&\MF^{\rfr}\\
      \Br_n^{\raff}\arrow[r,equal]\arrow[u,"\Phi^{\raff}"]&\Br^{\raff}_n\arrow[u,"\Phi^{\raff}"]\arrow[r,equal]&\Br_n^{\raff}\arrow[u,"\Phi^{\raff}"]
    \end{tikzcd}
  \]

In all cases the generators of the braid group are mapped to
Koszul matrix factorizations, hence the previous argument implies the following:
\begin{proposition} The functor \(\HC^\rst_{\rloc}\) is monoidal. Moreover,
its image is central:
  for any \(\calD\in \mathrm{D}^{per}(\Hilb)\)  and \(\beta\in \Br_n\)
  \[\HC^\rst_{\rloc}(\calD)\star \Phi(\beta)\sim \Phi(\beta)\star \HC^{\rst}_{\rloc}(\calD).\]
\end{proposition}
\begin{proof}
   Let's pick an affine lift \(\tilde{\beta}\in \Br^{\raff}_n\) of \(\beta\in \Br_n\). Now we can use proposition~\eqref{eq:HC-CH-dia} and proposition~\ref{prop:st-sq}:
 \begin{multline*}
 \Phi(\beta^{-1})\star\HC^\rst_{\rloc}(\calD)\star \Phi(\beta)=
   \Phi(\beta^{-1})\star\HC^\rfs(\calD')\star \Phi(\beta)\\=
   j_{\rfs}^*(\Phi^{\raff}(\tilde{\beta}^{-1})\star \HC^{\rfr}(\calD'')\star\Phi^{\raff}(\tilde{\beta}))=
 j_{\rfs}^*( \HC^{\rfr}(\calD''))=\HC^\rst_{\rloc}(\calD),\end{multline*}
here \(\calD'=\mathrm{loc}^{\rfs}\circ \mathrm{KSZ}_\frg(\calD)\) and \(\calD''\)
is such that \(j_{\rfs}^*(\calD'')=\calD'\).
\end{proof}

The last proposition and the previous diagram imply
\begin{corollary}\label{cor:main1}  Functors \(\CH^{\rst}\) and \(\HC^{\rst}\) 
have the following properties

  \begin{itemize}
\item the functor \(\CH^{\rst}_{\rloc}\) is a right adjoint of \(\HC^{\rst}_{\rloc}\),
\item the functor \(\HC^{\rst}_{\rloc}\) is monoidal,
\item the image of \(\HC^{\rst}_{\rloc}\) commutes with the elements \(\Phi(\beta)\), \(\beta\in \Br_n\).
\end{itemize}

\end{corollary}

\section{Knot invariants}
\label{sec:knot}

In this section we discuss the relation between the Chern character and the functor of the braid closure that was used in our previous papers \cite{OblomkovRozansky16}. There we constructed a link invariant \(\HHHb(\beta)\) of the closure of a braid
\(\beta\in \Br_n\). Our construction of the invariant is based on the homomorphism:
\[\Br_n\to \MF^{\rst},\quad \beta\mapsto \calC^{\rst}_\beta,\]
and the braid closure procedure
\[\mathbb{L}\colon \MF^{\rst}\to \MF(\widetilde{\FHilb}_{y=0}^{\rfree},B\times\Tqt,q^0t,0),\quad \calC\mapsto j^*_{e,G}(\calC)^G,\] 
where \(\MF(\widetilde{\FHilb}_{y=0}^{\rfree},\Tqt,q^0t,0)\) is the category of coherent sheaves on  \(\widetilde{\FHilb}^{\rfree}_{y=0}\), where
\[\widetilde{\FHilb}^{\rfree}_{y=0}=\{(X,Y,v)\in \frb\times\frn\times V_B\;|\;\CC\langle X,Y\rangle\, v=V_B\},\]
where subindex \(y=0\) is chosen to indicate that the matrix \(Y\) is strictly upper-trianglular.

The space \(\widetilde{\FHilb}^{\rfree}_{y=0}\times G\) embeds inside \(\scX\) by a \(G\times B^2\)-equivariant map \(j_{e,G}\) used in the
construction of \(\mathbb{L}\):
\begin{equation}\label{eq:jeG}
  j_{e,G}\colon\widetilde{\FHilb}^{\rfree}_{y=0}\times G\rightarrow \scX,\quad (X,Y,g)\mapsto (X,g,Y,g,Y).\end{equation}

It is shown in \cite{OblomkovRozansky16} that the triply-graded space
\begin{equation}\label{eq:HHgeo}
\HHH^\bullet(\beta):=\CE_{\frn}(\mathbb{L}(\calC^{\rst}_\beta)\otimes \Lambda^\bullet V)^{T},\end{equation}
is an isotopy invariant of the closure \(L(\beta)\) of the braid \(\beta\) (after an explicit
grading shift that is expressed in terms of writhe of the braid).

The quotient of the space \(\widetilde{\FHilb}^{\rfree}_{y=0}\) by \(B\) is the free flag Hilbert space \(\FHilb^{\rfree}_{y=0}\).
This space is smooth, an explicit smooth atlas is presented in \cite{OblomkovRozansky16}. The condition
of commutativity $[X,Y]=0$
defines an embedding of the flag Hilbert scheme
\(\FHilb_{y=0}\) inside \(\FHilb^{\rfree}_{y=0}\) .

The vector space \(V_B\) is a fiber of a trivial \(B\)-equivariant vector bundle on \(\widetilde{\FHilb}^{\rfree}_{y=0}\) which we denote the same symbol.
The vector bundle \(V_B\) becomes a vector bundle on the \(B\)-quotient \(\FHilb^{\rfree}_{y=0}\), we use notation \(\calB^\vee\)  for this vector bundle and
\(\calB\) for the dual bundle. Respectively, the trivial \(G\)-equivariant bundle with a fiber \(V_G\) over \(\mathscr{M}^{\rfs}\) descends to the vector
bundle on \(\Hilb\) which we denote by \(\calB^\vee\) and its dual by \(\calB^\vee\).

Given a finite-dimensional representation \(W_G\) of \(G\) we denote by \(W_B\) the restriction of this representation to \(B\). Given \(B^2\)-equivariant matrix
factorization \(\calF=(M,D,\partial)\) we denote define the tensor product with \(W_B\) by
\[\calF\otimes W_{B^{(1)}}=(M\otimes W_{B^{(1)}},D\otimes 1,\partial\otimes 1).
\]
Similarly, we define the tensor product \(\calG\otimes W_G\) for a \(G\)-equivariant matrix factorization.

The following is a simple consequence of the definition of the Chern functor:
\begin{proposition}
  For any \(m\) the functor \(\CH^{\rst}_{\rloc}\) intertwines the tensor powers of \(V^*\) and \(\calB\):
  \[\CH^{\rst}_{\rloc}(\calC\otimes (V^*_{B^{(1)}})^{\otimes m})=\CH^{\rst}_{\rloc}(\calC)\otimes \calB^{\otimes m},\quad \calC\in \MF^{\rfs}\]
  and this functor intertwines the action of the symmetric group \(S_m\). The analogous statement holds for the
  dual bundles: \(V\) and \(\calB^\vee\).

\end{proposition}

As corollary we obtain the second statement of theorem~\ref{thm:main}:
\begin{corollary}
  For any \(k\) and \(\beta\in \Br_n\) we have
    \[
    \CH^{\rst}_{\rloc}(\Phi(\beta\cdot FT^k))=
    \CH^{\rst}_{\rloc}(\Phi(\beta))\otimes \det(\calB)^k\]
  where \(FT\) is the full-twist braid.
\end{corollary}
\begin{proof}
  We show in \cite{OblomkovRozansky17} that \(\Phi(\beta\cdot FT^k)=\Phi(\beta)\otimes (\Lambda^n V^*_{B^{(1)}})^k\). Hence the previous proposition implies the statement.
\end{proof}

Also the analog of the proposition~\ref{prop:unit} holds in the stable category:
\begin{proposition}
  The element \(\calO\in \mathrm{D}^{\rper}(\Hilb_n,\mathbb{T}_{q,t},q^0t)\) is the convolution unit and
  \[\HC^{\rst}_{\rloc}(\calO)=\calC_{\parallel}^{\rst}.\]
\end{proposition}
\begin{proof}
  Let us define the units \(\calO^{\rfr}\) and \(\calO^{\rfs}\) in the categories
  \(\MF^{\rfr}\), \(\MF^{\rfs}\) by the formula analogous to \eqref{eq:unit}:
  \[  \calO^\bullet=[g-1,[X,Z]+wv]\in \MF_{\dr}^\bullet.
\]
The argument from the proof of proposition~\ref{prop:unit} can repeated word by
word in the framed situation to obtain:
\[\HC^{\rfr}(\calO^{\rfr})=\calC^{\rfr}_{\parallel}.\]
Next we observe that \(j^*_{\rfs}(\calC_{\parallel}^{\rfr})=\calC_{\parallel}^{\rfs}\)
hence we combine this observation with proposition~\eqref{eq:HC-CH-dia}:
\begin{equation}\label{eq:unit-fr}
  \calC_\parallel^{\rfs}=j^*_{\rfs}(\HC^{\rfr}(\calO^{\rfr}))=\HC^{\rfs}(j^*_{\rfs}(\calO^{\rfr}))=
  \HC^{\rfs}(\calO^{\rfs}).  
\end{equation}
The equivalence \(i_{\rfs}\) sends the unit \(\calC^{\rfs}_{\parallel}\) to the unit
\(\calC_\parallel^\st\). Thus we have:
\[\calC^{\rst}_{\parallel}=i_{\rfs}(\calC_\parallel^\rfs)=i_{\rfs}\circ \HC^{\rfs}(\calO^{\rfs}).\]
Thus according to the diagram~\eqref{eq:dia-stl} we only need to check
that \(\calO^{\rfs}=\mathrm{loc}^{\rfs}\circ\mathrm{KSZ}_{\frg}(\calO)=
j^*_{\rfs}(\mathrm{KSZ}_\frg(\calO))\). Finally, matrix factorization \(\mathrm{B}\) the construction of the  Koszul functor \(\mathrm{KSZ}_{\frg}\) in
theorem~\ref{thm:KSZ} is equal to the matrix factorization \eqref{eq:unit-fr}.
\end{proof}

Thus we can interpret the homology \(\HHH^\bullet(\beta)\) geometrically in terms of sheaves on the flag Hilbert schemes, see \cite{OblomkovRozansky16} for a detailed discussion.
One unpleasant aspect of the flag Hilbert scheme is that the flag Hilbert scheme \(\FHilb\) is very singular and it is
hard to use the standard tools like localization for computations on it. Below we show how one can fix this issue
with the Chern functor. 

\begin{theorem}\label{thm:HC-homology}
  For any \(\beta\in \Br_n\) there is an isomorphism
  \[\HHH^\bullet(\beta)\cong\mathcal{E}\mathrm{xt}(\calO, \CH^{\rst}_{\rloc}(\calC_\beta^{\rst})\otimes \Lambda^\bullet \calB).\]
\end{theorem}
\begin{proof}
  First, observe an isomorphism of complexes of
  \(\CC[\frg]^G\)-modules:
  \[\mathcal{E}\mathrm{xt}(\calC^{\rst}_{\parallel},\calC^{\rst}_\beta\otimes \Lambda^\bullet V^*)\simeq \HHH^\bullet(\beta),\]
  where \(\calC_\parallel^{\rst}=j^*_{\rst}\circ \pi_V^*(\calC_\parallel)\) is the unit
  in \(\MF^{\rst}\).
 Indeed, the matrix factorization \(\calC^{\rst}_{\parallel}\) is
  the Koszul matrix factorization for the defining ideal of the \(B\)-orbit of the image \(j_{e,G}(\widetilde{\Hilb}^{\rfree})\), and
  since the action of \(B\) on the orbit is free, the statement follows (a more detailed  argument can be found in the section 13 of~ \cite{OblomkovRozansky16}).


Finally, we use the adjointness of the functors \(\CH^{\rst}_{\rloc}\) and \(\HC^{\rst}_{\rloc}\):
  \[\mathcal{E}\mathrm{xt}(\calC^{\rst}_\parallel,\calC^{\rst}_\beta\otimes \Lambda^\bullet\calB)=\mathcal{E}\mathrm{xt}(\HC^{\rst}_{\rloc}(\calO),\calC^{\rst}_\beta\otimes\Lambda^\bullet V^*)=
  \mathcal{E}\mathrm{xt}(\calO,\CH_{\rloc}^{\rst}(\calC^{\rst}_\beta)\otimes \Lambda^\bullet \calB).\]
\end{proof}

\section{Computations}
\label{sec:computations}
In this section we apply our methods to the homology of the closure of a sufficiently positive JM braid.
Recall that the variety \(\scX^{\rst}\) has a collection of the line bundles
\(\calL_i\), \(i=1,\dots,n-1\) which are the structure sheaves of \(\scX^{\rst}\) with twisted \(B^2\)-equivariant
structure. More precisely, if \(\chi_i\) is the character of \(B\) that evaluates the \(i\)-th diagonal element, then
\(\calL_i=\calO_{\scX^{\rst}}\langle\chi_i,0\rangle\) where the characters in the brackets are the twisting characters of
\(B^2\). Denote
\[
\calL^{\vec{a}}
=\otimes_{i=1}^{n-1}\calL^{a_i}_i.\]
We showed in \cite{OblomkovRozansky17} that
$\calC^{\rst}_{\delta^{\vec{a}}} =  \calC_{\parallel}^{\rst} \otimes \calL^{\vec{a}}$, hence
%
%
our main theorem implies that the homology of the closure of \(\delta^{\vec{a}}\)
is computed
in terms of hypercohomology of \(\CH^{\rst}_{\rloc}(\calC_\parallel^{\rst}\otimes \calL^{\vec{a}})\). Below we prove some vanishing result that allows to use
localization to compute this hypercohomology by localization method.

\subsection{Very positive JM braids}
\label{sec:very-positive-jm}

In general the object \(\CH_{\rloc}^{\rst}(\calC^{\rst}_\beta)\), \(\beta\in \Br_n\) is an element of the category \(\mathrm{D}^{per}(\Hilb_n,\Tqt,q^0t)\).
There is no \(\Tqt\)-localization theory for objects of the last category. However, there is a natural folding functor:
\[\mathrm{Fold}:\mathrm{D}^b(\Hilb_n,q^0t^0)\to \mathrm{D}^{per}(\Hilb_n,q^0t)\]
which is a composition of the shift functor
\[\mathrm{D}^b(\Hilb_n,q^0t^0)\to \mathrm{D}^b(\Hilb_n,q^0t), (M_i,d)\to (t^{i}\cdot M_i,d)\]
and the functor that reduced the integral homological grading to two-periodic grading.

It is an interesting question for which \(\beta\in \Br_n\) the element \(\CH_{\rloc}^{\rst}(\calC^{\rst}_\beta)\) is from the smaller category
\(\mathrm{Fold}(\mathrm{D}^b(\Hilb_n,q^0t^0))\). For example, if \(\beta\) is a half-twist on three strands, the computations from
\cite{OblomkovRozansky16} suggests that  \(\CH_{\rloc}^{\rst}(\calC^{\rst}_\beta)\notin\mathrm{Fold}(\mathrm{D}^b(\Hilb_n,q^0t^0))\). However, we propose the following

\begin{conjecture}
  If \(\beta\) is such that \(L(\beta)\) is a link of planar singularity then
  \begin{enumerate}
  \item  \(\CH_{\rloc}^{\rst}(\calC^{\rst}_\beta)\) is isomorphic to 
    \(\mathrm{Fold}(\calF_\beta)\), where \(\calF_\beta\) is a \(\Tqt\)-equivariant sheaf (a complex with a support in a single homological degree).
    \item \(H^{>0}(\Hilb_n,\calF_\beta)=0\).
  \end{enumerate}
\end{conjecture}

Some support to the conjecture is provided in the paper \cite{GorskyKivinenOblomkov22}. Another statement in the direction of the conjecture is the following

\begin{theorem}\label{thm:very-pos}
  Suppose that \(\vec{a}\in \ZZ^{n-1}\) such that \(a_i\ge a_{j}-1\), \(i>j\). Then there is a \(\Tqt\)-equivariant sheaf
  \(\calF_{\delta^{\vec{a}}}\in \Coh(\Hilb_n,\Tqt,q^0t^0)\) such that
  \[\mathbb{H}(\CH_{\rloc}^{\rst}(\calC^{\rst}_{\delta^{\vec{a}+c\mathbf{1}}})\otimes \Lambda\calB)\simeq \mathrm{H}^0(\calF_{\delta^{\vec{a}}}\otimes\Lambda^\bullet\calB\otimes \det(\calB)^c),\]
  if \(c\) is sufficiently positive.
\end{theorem}

Our proof relies on a construction of Chern functor for dualizable link homology \cite{OblomkovRozansky19a},
that is discussed in the next section, and Broer's vanishing result \cite{Broer93}.
In forthcoming work we prove analogous result for the braids that are composition of Coxeter braids \cite{OblomkovRozansky17a} and JM braids.

\subsection{Another Chern functor}

To prove theorem \ref{thm:very-pos} we need to develop a version of the Chern functor for the dualizable homology from \cite{OblomkovRozansky19a}.
In this note we only treat the case of dualizable homology of closure of pure braids. An extension to general braids presents no technical difficulties and will be done
in a subsequent publication. Let us recall some details of \cite{OblomkovRozansky19a}.

Let us introduce a subspace \((G\times \frb)\times_{\frh} (G\times \frb)\subset
(G\times \frb)^2\) that consists of \((g_1,Y_1,g_2,Y_2)\) such that
the diagonal parts of \(Y_i\) are equal.
As explained in \cite{OblomkovRozansky19a}, the group of pure braids \(\PBr_n\subset \Br_n\) can be realized inside monoidal category:
\[\underline{\Phi}: \PBr_n\to (\underline{\MF}^{\st},\star),\]
\[\underline{\MF}^{\st}:=\MF((\frg\times (G\times \frb)\times_{\frh}(G\times \frb)\times V)^{\st},G\times B^2\times \Tqt,q^0t,W),\]\[ W(X,g_1,Y_1,g_2,Y_2,v)=\Tr(X(\Ad_{g_1}Y_1-\Ad_{g_2}Y_2))\]
where \(G\times B^2\times \Tqt\) is consistent with the action on \((\frg\times(G\times\frb)^2\times V)^{\st}\) and the convolution is defined by the
formula \eqref{eq:conv}.

Similarly to previously discussed case we define map \(j_{e,G}\) by the formula
\eqref{eq:jeG} and the functor \(\mathbb{L}\):
\[\mathbb{L}\colon \underline{\MF}^{\rst}\to \MF(\widetilde{\FHilb}^{\rfree},B\times\Tqt,q^0t,0),\quad \calC\mapsto j^*_{e,G}(\calC)^G,\] 
where \(\MF(\widetilde{\FHilb}^{\rfree},\Tqt,q^0t,0)\) is the category of coherent sheaves on the quotient \(\widetilde{\FHilb}^{\rfree}/B\), where
\[\widetilde{\FHilb}^{\rfree}=\{(X,Y,v)\in \frb\times\frb\times V_B\;|\;\CC\langle X,Y\rangle\, v=V_B\}.\]

The space \(\widetilde{\FHilb}^{\rfree}\) has a natural involution \(\mathfrak{F}\) that switches two copies of \(\frb\) and it is shown in \cite{OblomkovRozansky19a} that for any \(\beta\in \PBr_n\) we have an isomorphism
\begin{equation}
  \label{eq:F-sym}
 \mathbb{L}(\underline{\Phi}(\beta))=\mathfrak{F}^*(\mathbb{L}(\underline{\Phi}(\beta))).
\end{equation}

It is also shown  \cite{OblomkovRozansky19a} that the triply-graded space below is the link invariant of the closure \(L(\beta)\) 
\[\mathrm{HXY}^\bullet(\beta):=\CE_{\frn}(\mathbb{L}(\underline{\Phi}(\beta))\otimes \Lambda^\bullet V)^{T}.\]
 
The LHS of the last formula is element of category \(\mathrm{D}^{\per}(\CC_x^n\times\CC^n_y,q^0t)\) just as LHS of \eqref{eq:HHgeo} is element of \(\mathrm{D}^{\per}(\CC_x^n,q^0t)\).
It is shown in \cite{OblomkovRozansky19a} that there is a relation  these two types of homology
\[\mathrm{H}^\bullet(\beta)=\mathrm{HXY}^\bullet(\beta)\stackrel{L}{\otimes}_{\CC[\mathbf{x},\mathbf{y}]}\CC[\mathbf{x}],\]
where \(\CC[\mathbf{x},\mathbf{y}]=\CC[x_1,\dots,x_n,y_1,\dots,y_n]\) and
\(\CC[\mathbf{x}]=\CC[x_1,\dots,x_n]\).

In this section we construct another version of the Chern functor:
\[\underline{\CH}:\underline{\MF}^{\rst}\to \mathrm{Coh}(\frg\times\frg\times V, G\times \Tqt,q^0t),\]  
that allows us to compute our link homology. To explain the construction we need to introduce few auxiliary spaces.

Let us define a \(G\times B\times B\)-equivariant space
\[\underline{\scY}=\widetilde{\FHilb}^{\rfree}\times G\times G\] 
with \(G\times B^2\)-action given by:
\[(h,b_1,b_2)\cdot (X,Y,v,k,g)=(\Ad_{b_1}X,\Ad_{b_2}Y,b_1v,b_2kb_1^{-1},hgb_2^{-1}).\]
To define \(\underline{\CH}\) we need two maps:
\[p:\scY\to \widetilde{\FHilb}^{\rfree}, \quad (X,Y,v,k,g)\mapsto (X,Y,v),\]
\[\mu^{\oplus 2}:\scY\to \frg^2\times V,  \quad (X,Y,v,k,g)\mapsto (\Ad_{gk}X,\Ad_{g}Y,gkv)\]
 
Finally, we define a two periodic complex \(\mathrm{K}_\Delta=\mathrm{K}(k\in B)\)
on \(\underline{\scY}\). Then we have for \(\beta\in \PBr_n\)
\[\underline{\CH}(\beta)=\mu^{\oplus 2}_*(p^*(\mathbb{L}(\underline{\Phi}(\beta))\otimes \mathrm{K}_\Delta))^{B^2}.\]

\begin{proposition}\label{prop:HXY}
  For any \(\beta\in \PBr_n\) we have
  \[\mathbb{H}((\underline{\CH}(\beta)\otimes\Lambda^\bullet\calB)^G)=\mathrm{HXY}^\bullet(\beta)\]
\end{proposition}
\begin{proof}
  Indeed, \(\mathbb{H}((\underline{\CH}(\beta)\otimes\Lambda^i\calB)^G)=\mathbb{H}(p^*(\mathbb{L}(\underline{\Phi}(\beta))\otimes \mathrm{K}_\Delta\otimes\Lambda^i\calB)^{G\times B^2})\). Let us use notation \(B^{(i)}\) for the \(i\)-th factor in \(B^2\).
  Since the action of  \(G\times B^{(2)}\) on \(\und{\scY}\) is free,
  \((p^*(\calS)\otimes \mathrm{K}_\Delta\otimes\Lambda^i \calB)^{G\times B^{(2)}}\) is  homotopy equivalent to matrix factorization \(\calS\otimes \Lambda^i\calB\)
  on \(\widetilde{\FHilb}^{\rfree}\).  Hence \(\mathbb{H}((\calS\otimes\Lambda^i\calB)^B)=\mathbb{H}(p^*(\calS)\otimes \mathrm{K}_\Delta\otimes\Lambda^it\calB)^{G\times B^2})\) and the statement follows.
\end{proof}

The \(\GL_n\)-cover \(\widetilde{\Hilb}\) of \(\Hilb\) is a subvariety of \(\frg^2\times V\), that is we have an embedding \(i_{\Hilb}:\widetilde{\Hilb}\to \frg^2\times V\).  Respectively, \(\widetilde{\Hilb}_{y=0}\) is  a subvariety
defined by the nilpotency of \(Y\) in \((X,Y,v)\in\widetilde{\Hilb}\).
The subvariety  \(\widetilde{\Hilb}_{y=0}\) is a complete intersection and let us denote
by \(\mathrm{K}(\widetilde{\Hilb}_{y=0})\) the corresponding Koszul complex.

\begin{conjecture}
  For any \(\beta\in \PBr_n\) we have
  \[i_{\Hilb*}(\CH^{\st}_{\loc}(\beta))=
   \underline{\CH}(\beta) \otimes \mathrm{K}(\widetilde{\Hilb}_{y=0})
  \]
\end{conjecture}

If the complex \(\mathbb{L}(\underline{\Phi(\beta)})=\calS_\beta\)  is  obtained by folding from a bounded complex \(\hat{\calS}_\beta\in \mathrm{D}^b(\widetilde{\FHilb}^{\rfree},B\times\Tqt,q^0t^0)\),
\(\calS_\beta=\mathrm{Fold}(\hat{\calS}_\beta)\) then we call such braid
{\it unfoldable}.  If \(\beta\) is unfoldable we can define an unfolded
version of the Chern functor as
\[\underline{\CH}^b(\beta)=\mu^{\oplus 2}(p^*(\hat{\calS}_\beta\otimes \hat{K}_\Delta)),\]
where \(\hat{K}_{\Delta}\in D^b(\underline{\scY},G\times B^2\times \Tqt,q^0t^0)\) is
the Koszul complex of \(k\in B\). The following statement follows immediately from the
above definitions:

\begin{proposition}\label{prop:vanishXY}
  For an unfoldable \(\beta\in \PBr_n\) such that \(\mathbb{H}^{>0}(\underline{\CH}^b(\beta)^G)=0\) we have:
  \[\chi_{\Tqt}((\underline{\CH}^b(\beta)\otimes \Lambda^i\calB)^G)=\dim_{q,t}(\mathrm{HXY}^i(\beta)).\]
\end{proposition}

In particular, JM braids \(\beta=\delta^{\vec{a}}\) are unfoldable since
\[\calS_{\beta}=\mathrm{Fold}(\hat{\calS}_\beta),\quad
  \hat{\calS}_\beta=\hat{K}([X,Y]_{++})\otimes \calL^{\vec{a}}\]
where \(Z_{++}\) is the strictly upper-triangular part of a matrix \(Z\) and
\[\hat{K}([X,Y]_{++})\in \mathrm{D}^b(\widetilde{\FHilb},B\times\Tqt,q^0t^0)\] is the Koszul complex for the subvariety defined by \([X,Y]=0\).

The subspace \(\widetilde{\Hilb}\subset \frg^2\times V\) is a closed smooth subvariety of codimension \(n^2-n\).
Let us denote by \(\calO_{\widetilde{\Hilb}}\in D^b(\frg^2\times V,G\times\Tqt,q^0t^0)\) the structure sheaf of the corresponding manifold.

\begin{proposition}\label{prop:pullbO}
  Let \(\beta=\delta^{\vec{a}}\) then
  \[\underline{\CH}^b(\beta)=\mathrm{K}(\widetilde{\Hilb})\otimes \mu^{\oplus 2}_*(\calL^{\vec{a}})\]
\end{proposition}
\begin{proof}
  It is enough to show the statement for the case \(\vec{a}=0\). The space \(\widetilde{\Hilb}\) is smooth hence there is a \(G\times \Tqt\)-invariant 
  cover \(\frg^2\times V=\bigcup_i U_i\) such that \(\wt{\Hilb}\cap U_i\) is defined by
  \(n^2-n\) equations \(E_i^j\), \(j=1,\dots,n^2-n\). Let \(\hat{K}(E_i)\) be the corresponding Koszul complex.
 
  The preimage of \(\wt{\Hilb}\) under \(\mu^{\oplus 2}\) is the flag Hilbert scheme
  \(\FHilb\) that is defined by the \(n^2-n\) equations: \(k\in\frb\), \([X,Y]_{++}=0\),
  \(X,Y\in\frb\). Hence there is isomorphism \(f_i: p^*(\hat{\mathrm{K}}([X,Y]_{++}))\otimes \hat{K}_{\Delta}\to \mu^{\oplus2*}(\hat{K}(E_i^\bullet)) \).

  On the intersection \(U_i\cap U_j\) there is an isomorphism
  \(h_{ji}: \hat{K}(E_i^\bullet)\to\hat{K}(E_j^\bullet)\) such that \(g_{ij}=\mu^{\oplus 2*}(h_{ji})\) is strictly upper-triangular.
  Thus \(\{g_{ij}\}\) form a \(G\times B^2\)-equivariant cocycle with values in the unipotent  group.
  
   On the other hand \(\underline{\scY}/B^2\) retracts to \(\Fl\times\Fl\) thus has no odd cohomology. Hence the cocycle \(\{g_{ij}\}\) is a coboundary: there are
   \(\phi_i\) such that \(\phi_i \phi_j^{-1}=f_i^{-1}g_{ij} f_i\).
   Thus \(\{f_i\phi_i\}\) define desires isomorphism between \(\mu^{\oplus 2*}(\calO_{\wt{\Hilb}})\) and \(p^*(\hat{\mathrm{K}}([X,Y]_{++}))\otimes \hat{K}_\Delta\).
\end{proof}

\begin{proposition}\label{prop:freeL}
  If \(\vec{b}\) satisfies conditions \(a_i\ge a_j-1\) for any \(i,j\) such that
  \(i>j\). Then
  \begin{enumerate}
  \item \(\mathbb{H}^{>0}(\underline{\CH}^b(\delta^{\vec{a}}))=0\)
  \item \(\underline{\CH}^b(\delta^{\vec{a}})^G\) is a free module over
    \(\CC[\mathbf{x}]\)
  \end{enumerate}
\end{proposition}
\begin{proof} 
  First we observe that \(\underline{\scY}/B^2\) is a open piece of \(\tilde{\frg}\times\tilde{\frg}\times V\). The map \(\mu^{\oplus 2}\) is equal to two copies of
  the moment map \(\mu^{\oplus 2}=\mu\times \mu\times 1\).

  The open condition defining
  \(\und{\scY}/B^2\) is pulled back from \(\frg^2\times V\) along \(\mu^{\oplus 2}\).
  Thus it is to prove the statement that  the sheaf
  \[\calG_{\vec{a}}=(\mu\times\mu)_*(\calO\boxtimes \calL_{\vec{a}})\]
  has no higher cohomology and free over \(\CC[\frg^{(1)}]^G\) where \(\frg^{(1)}\) is the first copy  \(\frg^2\).

  The second fact is standard: there is a \(G\)-equivariant natural projection from \(\tilde{\frg}\to \frh\) and the \(G\)-invariant  functions on the fibers are constant,
  see
  \cite{ChrisGinzburg}. The vanishing result is a modification of the result from
  \cite{Broer93}.
Indeed, it is shown in the cited paper that
\begin{equation}\label{eq:Bro}
  H^{>0}(\calL_{\vec{a}}|_{T^*\Fl})=0\end{equation}
if the conditions of the statement hold.
  On the other hand \(\calL_{\vec{a}}|_{\mathrm{T}^*\Fl}=\calL_{\vec{a}}\otimes \mu^*(\mathrm{K}(\mathcal{N}))\) where \(\mathrm{K}(\mathcal{N})\) the Koszul complex of
  the nilpotent cone. Thus \(H^*(\calL_{\vec{a}}|_{\mathrm{T}^*\Fl})\) is computed by a spectral sequence \(E_1\) terms \(H^j(\calL_{\vec{a}})\otimes\Lambda^i \frh\). Thus
  we must have \(H^{>0}(\calL_{\vec{a}})=0\) to ensure \eqref{eq:Bro}.
  \end{proof}

\subsection{Proof of theorem~\ref{thm:very-pos}}
\label{sec:proof-theorem}
Proposition~\ref{prop:pullbO} and proposition~\ref{prop:pullbO} there is a sheaf
\(\tilde{\calF}_{\delta^{\vec{a}}}\) on \(\wt{\Hilb}\) such that
\(\underline{\CH}^{b}(\delta_{\vec{a}})=j_{\wt{\Hilb}*}(\tilde{\calF}_{\delta^{\vec{a}}})\).
Let us denote by \(\underline{\calF}_{\delta^{\vec{a}}}\) be the corresponding sheaf on
\(\Hilb\). 
 
Since \(\det(\calB)\) is an ample bundle on \(\Hilb\) for sufficiently positive
\(c\) the sheaf \(\underline{\calF}_{\delta^{\vec{a}}}\otimes \Lambda^i\calB\otimes \det(\calB)^c\) has no higher  cohomology all \(i\). Hence for \(\delta^{\vec{a}+c\vec{1}}\) the conditions of proposition~\ref{prop:vanishXY} are satisfied and we have
\[H^0(\underline{\calF}_{\delta^{\vec{a}+c\vec{1}}}\otimes \Lambda^i(\calB))=\mathrm{HXY}^i(\delta^{\vec{a}+c\vec{1}}).\]

Finally, we observe by proposition~\ref{prop:freeL} the \(\CC[\mathbf{x},\mathbf{y}]\)-module
\(\mathrm{HXY}^\bullet(\delta^{\vec{a}+c\mathbf{1}})\) is a free over \(\CC[\mathbf{x}]\).
Hence by the symmetry property it is a free module over \(\CC[\mathbf{y}]\). Thus we have
\begin{equation}\label{eq:manyHi}
\mathrm{H}^\bullet(\delta^{\vec{a}+c\vec{1}})=\mathrm{HXY}^\bullet(\delta^{\vec{a}+c\vec{1}})|_{y=0}=\mathrm{H}^0(\underline{\calF}_{\delta^{\vec{a}}}\otimes\calO_{\Hilb_{y=0}}\otimes\det(\calB)^c\otimes \Lambda^\bullet\calB), \end{equation}
where \(\Hilb_{y=0}\) is defined by vanishing of \(y\in H^0(\calB)\). Thus the theorem follows with \(\calF_{\delta^{\vec{a}}}=\underline{\calF}_{\delta^{\vec{a}}}\otimes\calO_{\Hilb_{y=0}}\).

\subsection{Proof of theorem~\ref{thm:loc}}\label{sec:proof-JM}
\begin{corollary}\label{cor:chi-JM}
  Suppose \(\vec{a}\) and \(c\) satisfy conditions of theorem~\ref{thm:very-pos} then
  \[\dim_{q,t}\mathrm{H}^i(\delta^{\vec{a}+c\vec{1}})=\chi_{\Tqt}(\hat{K}([X,Y]_{++})\otimes\calL^{\vec{a}}\otimes \det(\calB)^c\otimes\Lambda^i\calB)\]
  where \[\hat{K}([X,Y]_{++})\in D^b(\FHilb_{y=0}^{\rfree},\Tqt,q^0t^0)\] is the \(\Tqt\)-equivariant Koszul complex of relations
  \([X,Y]_{i,j}=0\), \(j-i>0\). 
\end{corollary}
\begin{proof}
  Indeed by the equation \eqref{eq:manyHi} it is enough to show that
  \[\chi_{\Tqt}(\underline{\CH}^b(\delta_{\vec{a}})\otimes\det(\calB)^c\otimes\Lambda^i\calB)=\chi_{\Tqt}(\hat{K}([X,Y]_{++})\otimes \det(\calB)^c\otimes\Lambda^i\calB).\]
  The last equality is derived by the same argument as proposition~\ref{prop:HXY} since \[\mathrm{Fold}(\hat{K}([X,Y]_{++})\otimes \calL^{\vec{a}}\otimes
    \det(\calB)^c)=\mathbb{L}(\Phi(\delta_{\vec{a}+c\vec{1}})).\]
\end{proof}

To compute the Euler characteristics of a \(\Tqt\)-equivariant complex on \(\FHilb_{y=0}^{\rfree}\) one can apply localization formula.
A description of the \(\Tqt\)-fixed locus in \(\FHilb_{y=0}^{\rfree}\) is given \cite{OblomkovRozansky17a}.
In the same paper the localization formula for the \(\Tqt\)-equivariant Euler characteristics of \(\hat{K}([X,Y])\otimes \calL^{\vec{a}}\otimes
    \det(\calB)^c\) is derived. The formula in \cite{OblomkovRozansky17a} matches the formula in the statement of theorem~\ref{thm:loc}. 

\end{document}